\documentclass{article}
\usepackage{sty}

\title{Optimal Spectral Recovery of a Planted Vector in a Subspace}

\usepackage{authblk}
\author[1]{Cheng Mao\thanks{Email: \textit{cheng.mao@math.gatech.edu}. Partially supported by NSF grants DMS-2053333 and DMS-2210734.}}
\author[2]{Alexander S.\ Wein\thanks{Email: \textit{aswein@ucdavis.edu}. Part of this work was done while with the Courant Institute at NYU, partially supported by NSF grant DMS-1712730 and by the Simons Collaboration on Algorithms and Geometry.}}

\affil[1]{School of Mathematics, Georgia Institute of Technology}
\affil[2]{Department of Mathematics, University of California, Davis}

\date{}

\begin{document}

\maketitle

\begin{abstract}
Recovering a planted vector $v$ in an $n$-dimensional random subspace of $\mathbb{R}^N$ is a generic task related to many problems in machine learning and statistics, such as dictionary learning, subspace recovery, principal component analysis, and non-Gaussian component analysis. In this work, we study computationally efficient estimation and detection of a planted vector $v$ whose $\ell_4$ norm differs from that of a Gaussian vector with the same $\ell_2$ norm. For instance, in the special case where $v$ is an $N \rho$-sparse vector with Bernoulli--Gaussian or Bernoulli--Rademacher entries, our results include the following:

(1) We give an improved analysis of a slight variant of the spectral method proposed by Hopkins, Schramm, Shi, and Steurer (2016), showing that it approximately recovers $v$ with high probability in the regime $n \rho \ll \sqrt{N}$. This condition subsumes the conditions $\rho \ll 1/\sqrt{n}$ or $n \sqrt{\rho} \lesssim \sqrt{N}$ required by previous work up to polylogarithmic factors. 
We achieve $\ell_\infty$ error bounds for the spectral estimator via a leave-one-out analysis, from which it follows that a simple thresholding procedure exactly recovers $v$ with Bernoulli--Rademacher entries, even in the dense case $\rho = 1$.

(2) We study the associated detection problem and show that in the regime $n \rho \gg \sqrt{N}$, any spectral method from a large class (and more generally, any low-degree polynomial of the input) fails to detect the planted vector. This matches the condition for recovery and offers evidence that no polynomial-time algorithm can succeed in recovering a Bernoulli--Gaussian vector $v$ when $n \rho \gg \sqrt{N}$.
\end{abstract}

\newpage

{
\hypersetup{hidelinks}
\tableofcontents
}

\newpage

\section{Introduction}

Suppose that an $n$-dimensional subspace of $\RR^N$ contains a planted vector $v$ and is otherwise uniformly random. Under what conditions can we efficiently recover $v$ given (a basis of) the subspace? In particular, when the vector $v$ is sparse, the task becomes finding the sparsest vector in the subspace. There has been significant interest in studying this generic problem in recent years \cite{demanet2014scaling,barak2014rounding,qu2016finding,sos-spectral,qu2020finding} due to its own interest and its connection to a variety of problems in machine learning and statistics. For example, finding a sparse vector in a subspace is related to sparse principal component analysis \cite{d2020sparse}, non-Gaussian component analysis
\cite{blanchard2006search}, clustering in the Gaussian mixture model \cite{davis2021clustering}, and the spiked transport model \cite{niles2019estimation}. 
It is also a subtask of algorithms for dictionary learning \cite{spielman2012exact,sun2015complete}.
See the survey~\cite{qu2020finding} for other recent advances in planted vector recovery and connections to other problems. 
In this work, we take a statistical-computational point of view to study the problems of estimating and detecting the planted vector $v$ in the subspace using computationally efficient algorithms.
After presenting our models and main results, we discuss in detail a wide range of related work in Section~\ref{sec:ex-work}.

\subsection{Models}

To characterize the conditions for efficient estimation of a planted vector in a subspace, let us consider two observation models. 
Throughout the paper, $N$ and $n$ will denote positive integers with $1 \le n \le N$ and $N \ge N_0$ for a sufficiently large absolute constant $N_0$.

\begin{model}[Gaussian basis]
\label{mod:gauss}
Fix a vector $v \in \RR^N$ to be estimated. 
For $n \le N$, let $\Yg \in \RR^{N \times n}$ be the matrix whose first column is $v$ and whose other $n-1$ columns are i.i.d.\ $\mathcal{N}(0, \frac{1}{N} I_N)$ vectors. 
Let $Q \in \RR^{n \times n}$ be an arbitrary orthogonal matrix. 
Suppose that we observe the matrix $\Yt = \Yg Q$.
\end{model}

\noindent The column span of $\Yt$ is an $n$-dimensional subspace of $\RR^N$ which contains $v$ as the planted vector and is otherwise uniformly random. 
Note that the columns of $\Yg$ are scaled to have approximately unit norms. 
Correspondingly, the interesting scenario is when $v$ is roughly a unit vector. 

Compared to the above model where we observe a Gaussian basis of the subspace, a more practical and harder problem is to find the planted vector given an arbitrary basis of the subspace. 
Since we can orthonormalize any given basis, it suffices to consider the following observation model. 

\begin{model}[Orthonormal basis]
\label{mod:orth}
Fix a unit vector $v \in \RR^N$ to be estimated. 
For $n \le N$, let $\Yg \in \RR^{N \times n}$ be the matrix whose first column is $v$ and whose other $n-1$ columns are i.i.d.\ $\mathcal{N}(0, \frac{1}{N} I_N)$ vectors. 
Suppose that we observe an arbitrary matrix $\Yh \in \RR^{N \times n}$ whose columns form an orthonormal basis for the column span of $\Yg$. 
\end{model}

\noindent It is without loss of generality to assume that the vector $v$ has unit norm in Model~\ref{mod:orth}, because the norm of $v$ is unidentifiable when an arbitrary basis is given. 

Our main results in Section~\ref{sec:estimate} require fairly minimal assumptions on the structure of $v$.
However, in order to discuss precise conditions for finding a planted (sparse) vector and compare to previous work, we will focus especially on the cases where $v$ has Bernoulli--Gaussian or Bernoulli--Rademacher entries, defined as follows. 

\begin{definition}[Bernoulli--Gaussian vector]
\label{def:bg}
We say that a random vector $v \in \RR^N$ is a Bernoulli--Gaussian vector with parameter $\rho \in (0, 1]$ and write $v \sim \BG(N, \rho)$, if the entries of $v$ are i.i.d.\ with 
$$
\begin{cases}
v_i = 0 & \text{ with probability } 1 - \rho , \\
v_i \sim \mathcal{N}(0, \frac{1}{N \rho}) & \text{ with probability } \rho . 
\end{cases}
$$ 
\end{definition}

\begin{definition}[Bernoulli--Rademacher vector]
\label{def:br}
We say that a random vector $v \in \RR^N$ is a Bernoulli--Rademacher vector with parameter $\rho \in (0, 1]$ and write $v \sim \BR(N, \rho)$, if the entries of $v$ are i.i.d.\ with 
$$
v_i = 
\begin{cases}
0 & \text{ with probability } 1 - \rho , \\
1/\sqrt{N \rho} & \text{ with probability } \rho/2 , \\
-1/\sqrt{N \rho} & \text{ with probability } \rho/2 . 
\end{cases}
$$ 
\end{definition}

\noindent The scaling in the above definitions is chosen so that $\|v\| \approx 1$. 
Note that the smaller $\rho$ is, the sparser $v$ is (in expectation). 
In particular, $\rho = 1$ corresponds to the densest case: $\BG(N, 1)$ reduces to the Gaussian distribution $\mathcal{N}(0, \frac{1}{N} I_N)$, and $\BR(N, 1)$ reduces to the distribution of a vector with independent Rademacher entries. 
The knowledge of the sparsity parameter $\rho$ is \emph{not} required for our estimators of $v$ in Section~\ref{sec:estimate}, but we assume that $\rho$ is known in other parts of the paper to ease the discussion.

\subsection{Problems and Main Results}
\label{sec:prob-res}

We consider two types of problems: \emph{estimation} and \emph{detection} of the planted vector in the subspace. We are interested in the regime where $N$ is growing and the other parameters $n$ and $\rho$ may depend on $N$.

\begin{problem}[Estimation]
\label{prob:estimate}
Observing the matrix $\Yt$ under Model~\ref{mod:gauss} or the matrix $\Yh$ under Model~\ref{mod:orth}, we aim to estimate or exactly recover the planted vector $v$ (with high probability), up to an unidentifiable sign flip.
\end{problem}

\begin{problem}[Detection]
\label{prob:detect} 
For parameters $n \le N$ and $\rho \in (0,1]$, define the following \emph{null} and \emph{planted} distributions respectively: 
\begin{itemize}
\item Under $\cQ$, observe $\Yt \in \RR^{N \times n}$ whose entries are i.i.d.\ $\mathcal{N}(0, \frac{1}{N})$ random variables.

\item Under $\cP$, first draw $v \sim \BG(N, \rho)$ and $Q \in \RR^{n \times n}$ a uniformly random orthogonal matrix, and then observe $\Yt = YQ$ according to Model~\ref{mod:gauss}. 
\end{itemize}
We aim to test between the hypotheses $\cQ$ and $\cP$ (with vanishing error probability).
\end{problem}

\noindent Note that we are focusing on the case $v \sim \BG(N, \rho)$ for the detection task. We remark that if $\rho = 1$ then the planted model $\cP$ coincides with the null model $\cQ$ by the rotational invariance of Gaussian vectors. 

Estimation in Model~\ref{mod:orth} is at least as difficult as estimation in Model~\ref{mod:gauss} (since orthonormalizing the columns of $\tilde Y$ reduces Model~\ref{mod:gauss} to Model~\ref{mod:orth}). Also, for $v \sim \BG(N, \rho)$, estimation in Model~\ref{mod:gauss} is at least as difficult as detection (which we justify formally in Section~\ref{sec:reduction}). 
Therefore, to characterize the optimal condition for finding a planted Bernoulli--Gaussian vector in a subspace, it suffices to provide an upper bound for estimation under Model~\ref{mod:orth} and a matching lower bound for detection (Problem~\ref{prob:detect}). 
Our main results in the case $v \sim \BG(N, \rho)$ can be summarized in the following informal theorem. 

\begin{theorem}[Informal statement of main results]
The critical condition for efficient recovery of a planted vector $v \sim \BG(N, \rho)$ in an $n$-dimensional random subspace of $\RR^N$ is $n \rho \ll \sqrt{N}$, in the following sense: 
\begin{itemize}
\item
``Upper bound'':
Assume Model~\ref{mod:orth} with $v = v'/\|v'\|$, where $v' \sim \BG(N, \rho)$. 
If $n \rho \ll \sqrt{N}$, $n \ll N$, and $\rho$ is bounded away from $1$, then a spectral method can be used to accurately estimate $v$ entrywise with high probability. 

\item
``Lower bound'':
Consider Problem~\ref{prob:detect}. If $n \rho \gg \sqrt{N}$, then no algorithm from a class of spectral and low-degree methods can distinguish between $\cQ$ and $\cP$ with vanishing error probability. 
\end{itemize}
\end{theorem}

\noindent In the above informal theorem, we have omitted secondary assumptions and polylogarithmic factors for readability. 
See Sections~\ref{sec:appl} and~\ref{sec:spectral} for the full, rigorous statements. 
In light of the lower bound, we conjecture that in the regime $n \rho \gg \sqrt{N}$, no polynomial-time algorithm can estimate the planted vector $v$ consistently with high probability; see Section~\ref{sec:spectral} for further discussion. Throughout, our focus is on the case $n \ll N$, i.e., the dimension of the subspace is small compared to that of the ambient space. In Model~\ref{mod:gauss} (but not Model~\ref{mod:orth}), the regime $n > N$ makes sense but our upper and lower bounds do not match and so the critical condition remains open.

Our results in Section~\ref{sec:appl} also include an analogous upper bound for \emph{Bernoulli--Rademacher} (instead of Bernoulli--Gaussian) vectors with the following differences: $\rho$ must now be bounded away from $1/3$ instead of $1$, and the estimation guarantee can be strengthened to \emph{exact} recovery of $v$ (with high probability) due to the discreteness of $v$. An earlier version of this paper~\cite{mao2021optimal} gave a matching lower bound for detection of Bernoulli--Rademacher vectors, showing failure of spectral and low-degree methods when $n \rho \gg \sqrt{N}$; we have chosen to omit this here in light of the discussion on lattice basis reduction in Section~\ref{sec:ex-work}.

Let us now introduce the spectral method for recovering $v$ under Model~\ref{mod:orth}, which is a slight variant of the spectral method proposed by Hopkins, Schramm, Shi, and Steurer~\cite{sos-spectral}.  
Let $\yh_i^\top$ be the $i$th row of the matrix $\Yh \in \RR^{N \times n}$ that we observe. 
Define an $n \times n$ matrix 
\begin{equation}
\Mh := \sum_{i=1}^N \left( \| \yh_i \|^2 - \frac{n-1}{N} \right) \yh_i \yh_i^\top - \frac{3}{N} I_n .
\label{eq:def-mh}
\end{equation} 
Let $\hat{\ub}$ be the leading eigenvector of $\Mh$ with unit norm, that is, the eigenvector whose associated eigenvalue has the largest magnitude. 
Our main upper bound (Theorem~\ref{thm:spectral}) shows that, up to a sign flip of $\hat{\ub}$, the vector $\Yh \hat{\ub}$ is entrywise close to the planted vector $v$ with high probability.

Note that the matrix $\hat{M}$ above has entries that are degree-four polynomials in the entries of the input $\hat Y$. Intuitively, this spectral method is designed to exploit the fact that the planted vector has an $\ell_4$ norm differing from that of a Gaussian vector of the same $\ell_2$ norm; we explain in Section~\ref{sec:first} how this works on a more technical level. Simpler spectral methods (e.g.~based only on the first two or three moments of $\hat Y$) are doomed to fail because in our primary examples of interest (planted Bernoulli--Gaussian or Bernoulli--Rademacher vector), the planted vector has entries whose 1st, 2nd, and 3rd moments match that of a Gaussian, and so the 4th moment of the data is the first to contain useful information.

A related and much simpler algorithm is to threshold the values $\|\hat y_i\|^2$ in order to identify the support of $v$. Since $\|\tilde y_i\|^2 = v_i^2 + n/N + \tilde O(\sqrt{n}/N)$, this method succeeds when $\rho \ll 1/\sqrt{n}$. However, since $n \le N$ this is sub-optimal compared to our condition $n \rho \ll \sqrt{N}$.

\subsection{Statistical Reformulations}
\label{sec:reformulate}

As we will show in the proof of Lemma~\ref{lem:equiv-detect}, Model~\ref{mod:gauss} can be restated in the following equivalent form. This model is sometimes called \emph{non-Gaussian component analysis} (e.g.~\cite{blanchard2006search,diakonikolas2021non,dudeja2022statistical}), which we discuss further in Section~\ref{sec:ex-work}.

\begin{model}[Non-Gaussian component analysis]
\label{mod:gauss-alt}
Fix a vector $v \in \RR^N$ to be estimated. 
For $n \le N$, draw a random vector $u$ from the the uniform distribution over the unit sphere in $\RR^n$. 
Conditional on $u$, draw independent samples $\yt_1,\ldots,\yt_N$ where $\yt_i \sim \mathcal{N} \big( v_i u, \frac{1}{N}(I_n - uu^\top) \big)$.
Suppose that we observe the matrix $\Yt \in \RR^{N \times n}$ whose rows are $\yt_1^\top, \dots, \yt_N^\top$.
\end{model}

\noindent 
In other words, instead of viewing the $n$ columns of $\Yt$ as a basis of a subspace, we now view the $N$ rows of $\Yt$ as independent samples. 
Note that here the planted direction $u$ corresponds the first row of the matrix $Q$ in Model~\ref{mod:gauss}. 
Moreover, each vector $\yt_i$ has a component of size $v_i$ along the direction $u$, and is Gaussian in all directions orthogonal to $u$. 
As a result, recovering the hidden direction $u$ is called non-Gaussian component analysis (particularly in the case where $v$ has entries drawn i.i.d.\ from some non-Gaussian distribution).

This equivalent model brings a few benefits. Conceptually, Model~\ref{mod:gauss-alt} fits better in the typical statistical framework, where $n$ is the dimension, $N$ is the sample size, and the rows of $\Yt$ are the independent samples. 
Technically, estimating the planted vector $v$ boils down to estimating the planted direction $u$, which is what the spectral estimator is doing. 
Furthermore, our proof of the low-degree lower bounds awedit{(see Section~\ref{sec:proof-lower}) is greatly simplified by working with Model~\ref{mod:gauss-alt} instead of Model~\ref{mod:gauss}, as it eliminates the need to work with high-order moments of a random orthogonal matrix $Q$}.

We note that the assumption $n \le N$ in Model~\ref{mod:gauss-alt} is not necessary from the perspective of non-Gaussian component analysis, but it is natural from the perspective of a planted vector in a subspace (since the subspace must have smaller dimension than the ambient space) and so our focus will be on the case $n \le N$; we will in fact often take the liberty of assuming $n \ll N$.

Despite the equivalence of Models~\ref{mod:gauss} and~\ref{mod:gauss-alt}, let us emphasize that Model~\ref{mod:orth} (where we observe an orthonormal basis) is strictly more general than both. 
Our upper bounds for Model~\ref{mod:orth} in Section~\ref{sec:orth} require additional proof ideas to address the complication introduced by the orthonormalization. 

\begin{remark}[Noisy model]
\label{rmk:noisy}
In Model~\ref{mod:gauss-alt}, we can potentially add noise to the observations: Suppose that instead of $\yt_i$, we observe $\yt_i + z_i$ for $z_i \sim \mathcal{N}(0, \frac{\sigma^2}{N} I_n)$. 
Each observation now has distribution $\mathcal{N} \big( (v_i + z_i^\top u) u, \frac{1+\sigma^2}{N}(I_n - uu^\top) \big)$.
The problem then becomes recovering the vector $\tilde v$ with entries $\tilde v_i = v_i + z_i^\top u$ given independent observations from $\mathcal{N} \big( \tilde v_i u, \frac{1+\sigma^2}{N}(I_n - uu^\top) \big)$, which (by rescaling) is simply an instance of the original problem.
Moreover, our upper bounds in Section~\ref{sec:estimate} are valid for any fixed planted vector without a distributional assumption, so the same spectral method can be used to estimate $\tilde v$ with theoretical guarantees. 
If $v$ possesses structure---for example, it is sparse---then we can further threshold the estimator of $\tilde v$ to obtain an estimator of $v$. We will not discuss the noisy case further in this work.
\end{remark}

Another way of reformulating the planted vector problem is through the language of principal component analysis (PCA) \cite{candes2011robust,d2020sparse}. 
Recall the matrix $Y = [v \; g_2 \; \cdots \; g_n]$ from Model~\ref{mod:gauss}, where $g_2, \dots, g_n$ are i.i.d.\ $\mathcal{N}(0, \frac 1N I_N)$ vectors. 
Let $\tilde v = v - g_1$ where $g_1$ is another independent $\mathcal{N}(0, \frac 1N I_N)$ vector. 
Hence $Y = [\tilde v \; 0 \; \cdots \; 0] + G$ where $G = [g_1 \; \cdots \; g_n]$. 
Let $u \in \RR^n$ denote the first row of the orthogonal matrix $Q \in \RR^{n \times n}$. 
Then we can view Model~\ref{mod:gauss} as a spiked component with additive Gaussian noise: We observe $\tilde Y = Y Q = \tilde v u^\top + Z$, where $Z = GQ$ has i.i.d.\ $\mathcal{N}(0, \frac 1N I_N)$ columns.

A related viewpoint taken by~\cite{d2020sparse} is to write $\tilde Y = v u^\top + Z - g_1 u^\top$. Now the planted signal is $v$ instead of $\tilde v$, and the extra term $-g_1 u^\top$ is thought of as an adversarial perturbation. In this sense, the planted vector problem is a hard instance for robust sparse PCA; see Section~6.3 of~\cite{d2020sparse}.

\subsection{Organization}
In Section~\ref{sec:ex-work}, we discuss related models and results in the literature. 
In Section~\ref{sec:estimate}, we study estimation of a planted vector in a subspace assuming either Model~\ref{mod:gauss} or~\ref{mod:orth}. 
In particular, we provide entrywise error bounds for the spectral estimator and identify the conditions for recovery of a planted Bernoulli--Gaussian or Bernoulli--Rademacher vector. 
In Section~\ref{sec:lower}, we turn to the problem of detecting a planted vector in a subspace. We establish lower bounds for a class of spectral and low-degree methods that match our upper bounds for the estimation problem. 
The proofs of our results are provided in Sections~\ref{sec:proof-upper} and~\ref{sec:proof-lower}, as well as in the appendix.

\subsection{Notation}
For a positive integer $k$, let $[k] := \{1, 2, 3, \dots, k\}$. 
We use $C$ and $c$, possibly with subscripts, to denote positive constants.
We are interested in the regime where the ambient dimension $N$ is growing and the other parameters $n$ and $\rho$ may depend on $N$.
We use standard asymptotic notation $O(\cdot)$, $o(\cdot)$, $\Omega(\cdot)$, $\omega(\cdot)$; more precisely, $O(\cdot)$ and $\Omega(\cdot)$ hide constant factors (not depending on $N,n,\rho$), $o(1)$ stands for a quantity that depends only on $N$ and tends to $0$ as $N \to \infty$, and $\omega(1)$ stands for a quantity that depends only on $N$ and tends to $\infty$ as $N \to \infty$. We use $\tilde O(\cdot)$ to hide a polylogarithmic factor, that is, a factor $(\log N)^C$ for an absolute constant $C > 0$. (Throughout, we assume $N \ge 3$ so that $\log N > 1$.)
We write $a \ll b$ to mean $a \le \frac{b}{(\log N)^C}$ for an absolute constant $C > 0$.
Let $\land$ and $\lor$ denote the min and the max operator between two real numbers respectively.
We define the sign function $\mathsf{sign}(\cdot)$ by $\mathsf{sign}(x) = 1$ when $x > 0$; $\mathsf{sign}(x) = -1$ when $x < 0$; and $\mathsf{sign}(x) = 0$ when $x = 0$. We use $\mathbf{1}\{A\}$ to denote the $\{0,1\}$-valued indicator for an event $A$. We use the convention $\NN = \{0,1,2,\ldots\}$.

We use $\|v\|$ to denote the Euclidean norm of a vector $v$ and use $\|M\|$ to denote the spectral (operator) norm of a matrix $M$. 
We use $\|v\|_p$ to denote the $\ell_p$ norm of a vector $v$ where $p \ge 0$. 
Given a symmetric matrix $M$, the leading eigenvalue of $M$ refers to the eigenvalue that has the largest magnitude, and the leading eigenvector of $M$ refers to the eigenvector associated with the leading eigenvalue. 
All eigenvectors in consideration are taken to be unit vectors unless otherwise specified.
We use $I_k$ to denote the $k \times k$ identity matrix.
For a vector $u$, we will sometimes say that a statement is true ``up to a sign flip of $u$'', by which we mean the statement is either true for $u$ or $-u$.

\section{Related Work}
\label{sec:ex-work}

Recovery of a planted vector in a subspace has been widely studied in recent years and is related to many other problems in the literature, which we now discuss.

\subsubsection*{Planted Sparse Vector}
Qu, Sun, and Wright \cite{qu2016finding} showed that, information-theoretically, a non-convex $\ell_1/\ell_2$ minimization program recovers the planted vector $v$ as long as $\rho \le c$ and $n \le c N$ for a sufficiently small constant $c>0$ . 
However, $\ell_1/\ell_2$ minimization may be computationally infeasible, and existing polynomial-time algorithms can only recover $v$ under stronger assumptions. 
Qu, Sun, and Wright showed in the same work that, if $\rho \le c$ and $n \ll N^{1/4}$ where $\ll$ hides a logarithmic factor, then a non-convex approach based on alternating directions recovers $v$ with high probability. 
In the regime $\rho \ll 1/\sqrt{n}$ and $n \le c N$, Demanet and Hand \cite{demanet2014scaling} proposed an $\ell_1/\ell_\infty$ minimization program which recovers $v$ with high probability and can be solved efficiently using linear programming. 
The logarithmic factors in this result were improved by d'Orsi et al.~\cite{d2020sparse} using semidefinite programming, under a more general model for sparse PCA with adversarial perturbations. Also recall from the end of Section~\ref{sec:prob-res} that a much simpler thresholding algorithm also achieves the same condition $\rho \ll 1/\sqrt{n}$ up to logarithmic factors.
In the regime $\rho \le c$ and $n \sqrt{\rho} \le c \sqrt{N}$, Barak et al.\ \cite{barak2014rounding} proposed a sum-of-squares method to estimate $v$ based on $\ell_2/\ell_4$ minimization. Finally, motivated by the sum-of-squares method, Hopkins et al.\ \cite{sos-spectral} proposed a fast spectral method to estimate $v$, assuming $\rho \le c$ and $n \ll \sqrt{N}$. 

The set of conditions $n \rho \ll \sqrt{N}$ and $n \ll N$ needed for our upper bounds subsumes all those required by previous polynomial-time algorithms up to logarithmic factors.
Moreover, we establish a computational lower bound for a class of spectral and low-degree methods (discussed further below) in the regime $n \rho \gg \sqrt{N}$, demonstrating the optimality of our analysis.

\subsubsection*{Comparison to Hopkins et al.\ (2016)}
Since our spectral method is a slight variant of the one proposed in \cite{sos-spectral}, let us emphasize the improvements over this work. 
First, \cite{sos-spectral} studies the matrix $\hat M$ in \eqref{eq:def-mh} without the last term $- \frac{3}{N} I_n$. 
This centering turns out to be essential for recovering a Rademacher vector ($v \sim \BR(N,\rho)$ with $\rho = 1$), in which case we show that $v$ can be recovered using the \emph{bottom} eigenvalue of $\hat M$ provided that $n \ll \sqrt{N}$; this dense case was not treated in \cite{sos-spectral}.

Second, our analysis is sharper and more delicate than that of~\cite{sos-spectral} in order to achieve the condition $n \rho \ll \sqrt{N}$ rather than $n \ll \sqrt{N}$. This improvement is not so difficult for the case of a Gaussian basis, but requires new ideas for the case of an orthonormal basis. One key ingredient is a sharper bound on how much the ``leverage scores'' (the norms of the rows of the observed matrix) change when orthonormalizing a Gaussian basis. Specifically, our Equation~\eqref{eq:lev} improves upon Lemma~B.4 in \cite{sos-spectral}: for example, if $N \le \tilde O(n)$ and $\|v\|_\infty \le \tilde O(\frac{1}{ \sqrt{N \rho} })$, then \eqref{eq:lev} gives a bound of order $\tilde O(\frac{1}{\sqrt{N}} + \frac{1}{N \rho})$ while Lemma~B.4 in \cite{sos-spectral} gives a trivial bound of order $\tilde O(1)$. 

Finally, we establish an entrywise bound on the estimation error $\hat Y \hat{\mathbf{u}} - v$, which is stronger than the $\ell_2$ error bound in \cite{sos-spectral}. 
This requires a new leave-one-out analysis which is perhaps the most involved proof in this work; see the proof of Theorem~\ref{thm:spectral}. 
Thanks to the entrywise error bound, we can threshold $\hat Y \hat{\mathbf{u}}$ to recover a Bernoulli--Rademacher vector \emph{exactly}, improving over the results on approximate recovery in \cite{barak2014rounding,sos-spectral}.

\subsubsection*{Non-Gaussian Component Analysis}

Model~\ref{mod:gauss-alt} and variants are often referred to as \emph{non-Gaussian component analysis (NGCA)}, the problem of finding a low-dimensional non-Gaussian subspace in high-dimensional Gaussian data. Various algorithmic results exist; see e.g.~\cite{blanchard2006search,bean,VX-junta,TV-non,GS-entropy} and references therein. However, prior work has not pinned down the precise conditions for recovery that we establish here, focusing instead on coarser-grained objectives such as achieving polynomial sample complexity in certain variations of the model.

Starting from the work of~\cite{DKS-sq}, various special cases of NGCA have been shown to be computationally hard in the \emph{statistical query model}, suggesting a statistical-to-computational gap. In contrast to our work, the focus has been on settings where the planted non-Gaussian distribution matches many moments of a Gaussian.

The connection between NGCA and recovering a planted vector in a subspace (i.e., the equivalence between Models~\ref{mod:gauss} and~\ref{mod:gauss-alt} discussed in Section~\ref{sec:reformulate}) seems to have gone largely unnoticed, with one exception being Section~6.3 of~\cite{d2020sparse}.

A further generalization of NGCA is the \emph{spiked transport model} of~\cite{niles2019estimation}. Here, two high-dimensional distributions are assumed to differ only on a low-dimensional subspace (but the marginals need not be Gaussian). The results of~\cite{niles2019estimation} establish the optimal statistical rate for estimating the Wasserstein distance between the two distributions, and give evidence for statistical-to-computational gaps via the statistical query model.

\subsubsection*{Related Spectral Methods}

We discuss here a few related results that have also used the spectral method of~\cite{sos-spectral} or variants thereof. First, the spectral method in Algorithm~8.1 of~\cite{d2020sparse} is a generalization of this spectral method for a related sparse PCA model. After the initial version of our work appeared, a few other relevant results have emerged. Davis et al.~\cite{davis2021clustering} use the same spectral method for clustering a mixture of Gaussians with unknown covariance, a problem which includes Model~\ref{mod:gauss} with $v \sim \BR(N,1)$ as a special case. Even more recently, Dudeja and Hsu~\cite[Section~6.2]{dudeja2022statistical} gave a higher-order generalization of the spectral method which can solve NGCA when the planted non-Gaussian distribution has $k$th moment differing from that of a Gaussian; the basic version of the spectral method that we study here corresponds to $k=4$.

\subsubsection*{Eigenvector Perturbation}
Our entrywise bound on the error of the spectral estimator is akin to $\ell_\infty$ bounds for eigenvector perturbation, which have been widely studied in recent years \cite{koltchinskii2016perturbation,fan2018eigenvector,cape2019two,abbe2020entrywise,lei2019unified}. 
In particular, leave-one-out analysis is a useful technique for establishing bounds in the $\ell_\infty$ norm. 
See the survey \cite{chen2020spectral} for more recent developments in spectral methods and a discussion on the framework of leave-one-out analysis. 
Our analytic method is therefore a contribution to this line of research.

\subsubsection*{Dictionary Learning}
Finding a planted sparse vector in a subspace is a subroutine of several algorithms for dictionary learning \cite{spielman2012exact,barak2014rounding,sun2015complete}; see more references in the survey \cite{qu2020finding}. 
While our results in this work are specific to recovering a planted vector, we remark that maximizing the $\ell_4$ norm over unit vectors is a general strategy for recovering a sparse vector. It can be realized via non-spectral methods and has been successfully applied to problems such as dictionary learning and blind source separation \cite{barak2014rounding,zhai2019understanding,zhai2020complete}. 
We believe this class of methods deserve further study because, despite their methodological success, there are very few results regarding their optimality under statistical models.

\subsubsection*{Low-degree Lower Bounds}
The framework we adopt for probing the computational complexity of hypothesis testing tasks is to prove unconditional lower bounds against the class of \emph{low-degree polynomial algorithms} (which includes a large class of spectral methods). This idea arose from a line of work on the sum-of-squares hierarchy~\cite{sos-clique,HS-bayesian,sos-hidden,sam-thesis} (see also~\cite{kunisky2019notes} for a survey) and has proven successful at explaining statistical-to-computational gaps in many classical settings: For problems such as planted clique, community detection, sparse PCA, tensor PCA, and more, the conjectured ``hard'' regime (where no polynomial-time algorithms are known) coincides precisely with the regime in which low-degree polynomial algorithms have been proven to fail. In this work, we show low-degree hardness of the planted vector detection problem (Problem~\ref{prob:detect}) with $v \sim \BG(N,\rho)$ in the regime $n \rho \gg \sqrt{N}$. This implies in particular that \emph{any} spectral method from a large class must fail in this regime (see Theorem~\ref{thm:sp-lower-bd}), providing a strong converse to our upper bounds. Thus, when $n \rho \gg \sqrt{N}$, the planted vector problem suffers from a computational barrier akin to many other settings in the literature, leading us to expect that no polynomial-time algorithm can succeed in this regime; see Section~\ref{sec:spectral} for further discussion (and see the discussion on lattice basis reduction below for an important caveat).

We now discuss some related lower bounds. Prior to our work, there were low-degree lower bounds suggesting hardness of Problem~\ref{prob:detect} in the specific regime $n = \Omega(N)$ and $\rho \gg 1/\sqrt{N}$; see Remark~2.4 of~\cite{rip-cert}. Our results generalize this substantially by allowing $n \ll N$. There are also existing lower bounds against the related (but incomparable) class of sum-of-squares algorithms in the case $v \sim \BR(N,1)$ when $n \gg N^{2/3}$~\cite{ghosh2020sum,davis2021clustering}, which is sub-optimal compared to our condition $n \gg \sqrt{N}$. After the initial version of our work appeared, we became aware of a recent low-degree lower bound by d'Orsi et al.~\cite[Theorem~6.7]{d2020sparse}, showing hardness of detecting a particular planted vector under similar conditions to our result; one advantage of our result is that we treat all low-degree polynomials rather than just multilinear ones, which is crucial for proving failure of spectral methods (Theorem~\ref{thm:sp-lower-bd}). A byproduct of our lower bound that may be of independent interest is a general-purpose formula (Lemma~\ref{lem:formula}) for low-degree analysis of NGCA problems with an arbitrary non-Gaussian planted distribution and an arbitrary prior on the planted direction. We also remark that after the initial appearance of this work, Dudeja and Hsu~\cite[Section~F.4]{dudeja2022statistical} gave a similar low-degree lower bound for NGCA which gives sharp results even when the planted non-Gaussian distribution matches many moments of a Gaussian; similar results are also implicit in~\cite{sq-list-decode}.

\subsubsection*{Lattice Basis Reduction}

After the initial version of this work, a new algorithmic result appeared for recovering a \emph{discrete} vector, e.g., $v \sim \BR(N, \rho)$, under Model~\ref{mod:gauss-alt} (or equivalently, Model~\ref{mod:gauss}). Specifically, a method based on the Lenstra--Lenstra--Lov\'asz lattice basis reduction algorithm has been developed independently by Diakonikolas and Kane \cite{diakonikolas2021non} and Zadik et al.~\cite{zadik2021lattice}.
This method exactly recovers $v \sim \BR(N, \rho)$ in polynomial time provided that $N \ge n+1$, which seemingly challenges our conclusion that the fundamental condition for polynomial-time recovery is $n \rho \ll \sqrt{N}$.
However, the lattice basis reduction crucially relies on $v$ being discrete (or extremely close to discrete); in particular, it does not apply to the Bernoulli--Gaussian case $v \sim \BG(N, \rho)$. 
Moreover, lattice basis reduction does not fall in the scope of spectral or low-degree methods, so it does not contradict our lower bounds (the initial version of our work~\cite{mao2021optimal} gave a low-degree lower bound for the case $v \sim \BR(N,\rho)$).

In hindsight, this new algorithmic development does not undermine the low-degree framework: It was already established that low-degree algorithms are only conjectured to be optimal (among all polynomial-time algorithms) for problems that have at least a constant amount of ``noise'' per input~\cite{sam-thesis}; see Section~1.3 of~\cite{zadik2021lattice} for discussion. In light of this, we still expect that $n \rho \ll \sqrt{N}$ is the fundamental condition for polynomial-time recovery when $v$ has entries that are sufficiently noisy, e.g., Bernoulli--Rademacher convolved with a Gaussian of small constant variance, or Bernoulli--Gaussian. We have chosen to consider the Bernoulli--Gaussian prior for our low-degree lower bounds so as to avoid this issue of noise.

\section{Spectral Estimation of a Planted Vector}
\label{sec:estimate}

To study estimation of a planted vector in a subspace, we first consider Model~\ref{mod:gauss} where a Gaussian basis of the subspace is given, in Section~\ref{sec:upper}. 
Then, in Section~\ref{sec:orth}, we turn to Model~\ref{mod:orth} where an arbitrary orthonormal basis is given and show that our algorithmic guarantees continue to hold.
We then apply the general upper bound to the Bernoulli--Gaussian and Bernoulli--Rademacher models in Section~\ref{sec:appl}.

\subsection{Estimation from a Gaussian Basis}
\label{sec:upper}

First, let $v$ be a fixed vector in $\RR^N$. 
Under Model~\ref{mod:gauss}, to estimate the planted vector $v$, we consider the $n \times n$ matrix 
\begin{equation}
\Mt:= \sum_{i=1}^N \left( \|\yt_i\|^2 - \frac{n-1}{N} \right) \yt_i \yt_i^\top - \frac{3}{N} I_n , 
\label{eq:def-mt}
\end{equation}
where $\yt_i^\top$ denotes the $i$th row of the observed matrix $\Yt$. 
Let $\tilde{\ub}$ be the leading eigenvector of $\Mt$. Then the vector $\Yt \tilde{\ub}$ is an estimator of the planted vector $v$. 
The following theorem gives conditions under which $\Yt \tilde{\ub}$ is guaranteed to be close to $v$ entrywise with high probability (up to a sign flip).

\begin{theorem}
\label{thm:entrywise}
Fix any constant $c > 0$. Assume Model~\ref{mod:gauss} with $v$ satisfying $\|v\| = 1$ and $\left| \|v\|_4^4 - \frac{3}{N} \right| \ge \frac{c}{N}$. 
Suppose that 
\begin{equation}
\eps := \frac{1}{ \|v\|_4^4 } \frac{n}{N^{3/2}}  + \frac{\|v\|_\infty}{ \|v\|_4^2 } \sqrt{\frac{n}{N}} \ll 1 . 
\label{eq:def-eps}
\end{equation}
Let $\yt_j^\top$ be the $j$th row of $\Yt$ for $j \in [N]$, and let $\tilde{\ub}$ be the leading eigenvector of the matrix $\Mt$ defined in \eqref{eq:def-mt}. 
Up to a sign flip of $\tilde{\ub}$, it holds with probability $1 - N^{-\omega(1)}$ that 
\begin{equation}
\left| \tilde y_j^\top \tilde{\ub} 
- v_j \right|
\le \tilde O \left( \eps \cdot \Big( |v_j| + \frac{1}{\sqrt{N}} \Big) \right) \quad \text{ for all } j \in [N], 
\label{eq:entrywise-bound}
\end{equation}
and, as a result,
\begin{equation}
\|\Yt \tilde{\ub} - v\|
\le \tilde O \left( \epsilon \right) . 
\label{eq:simplified}
\end{equation} 
\end{theorem}

\noindent
Recall that we use the notation $\ll$ to hide a polylogarithmic factor in $N$.
We also use $\tilde{O}$ to hide polylogarithmic factors in $N$, as well as a factor of $1+1/c$. Moreover, for simplicity we restrict our attention to the case where $v$ has exactly unit norm.
In fact, these simplifications can be lifted; see Theorem~\ref{thm:entrywise-general} in the appendix for a more precise version of the bound \eqref{eq:simplified}. 

To interpret the above theorem, first note that the matrix $\Mt$ defined by \eqref{eq:def-mt} is entrywise a degree-four polynomial of the observations. Therefore, it is not surprising that the result depends on the $\ell_4$ norm of $v$. 
Moreover, $\EE \|g\|_4^4 = 3/N$ for $g \sim \mathcal{N}(0,  \frac{1}{N} I_N)$, so it is natural to assume that $\|v\|_4^4$ is bounded away from $3/N$, or else we cannot hope to identify $v$ using the first four moments alone. If $\|v\|_4^4 \ge \frac{3}{N} + \frac{c}{N}$ then the leading eigenvalue of $\tilde M$ will be the maximum eigenvalue (which is positive); if $\|v\|_4^4 \le \frac{3}{N} - \frac{c}{N}$ then the leading eigenvalue will be the minimum eigenvalue (which is negative).
The $\ell_4$ norm of a unit vector $v \in \RR^N$ can be viewed as an analytic notion of sparsity: We always have $1/N \le \|v\|_4^4 \le 1$, and the larger $\|v\|_4^4$ is, the ``sparser'' $v$ is. 
Hence, the bound \eqref{eq:simplified} says that it is easier to estimate a sparser vector planted in a random subspace, as expected.

We remark that the term $\frac{1}{ \|v\|_4^4 } \frac{n}{N^{3/2}}$ in the definition of $\epsilon$ dominates the term $\frac{\|v\|_\infty}{ \|v\|_4^2 } \sqrt{\frac{n}{N}}$ in some important cases: 
\begin{itemize}
\item
First, in the regime $n \ll \sqrt{N}$ (where the result of~\cite{sos-spectral} applies), we can use the bound $\|v\|_\infty \le \|v\|_4$ to show that $\frac{1}{\|v\|_4^4} \frac{n}{N^{3/2}} \ll 1$ implies $\frac{\|v\|_\infty}{\|v\|_4^2}\sqrt{\frac{n}{N}} \ll 1$. 

\item
Second, suppose that $\|v\|_\infty^2 \le \tilde{O}(\sqrt{n}/N)$. 
Note that this is a natural condition because if $v$ has an entry $v_i^2 \gg \sqrt{n}/N$, then $|v_i|$ can already be estimated accurately using $\|\tilde y_i\|^2 = v_i^2 + n/N + \tilde O(\sqrt{n}/N)$ alone. 
Under this condition, the bound $\|v\|_4^4 \le \|v\|_\infty^2 \|v\|^2 = \|v\|_\infty^2$ yields $\frac{\|v\|_\infty}{\|v\|_4^2} \sqrt{\frac{n}{N}} \le \frac{\|v\|^2_\infty}{\|v\|_4^4} \sqrt{\frac{n}{N}} \le \tilde{O}\left(\frac{1}{\|v\|_4^4} \frac{n}{N^{3/2}}\right)$.
\end{itemize}

Furthermore, note that \eqref{eq:entrywise-bound} is an entrywise bound for estimating each coordinate of $v$. 
The average entry size of a dense unit vector is $1/\sqrt{N}$, while the nonzero entries of a sparse vector typically have larger sizes. Hence the error $\eps \big(|v_j| + \frac{1}{\sqrt{N}}\big)$ of estimating $v_j$ is indeed small. Finally, summing the squares of the entrywise errors in \eqref{eq:entrywise-bound} yields \eqref{eq:simplified}.

\subsection{Estimation from an Orthonormal Basis}
\label{sec:orth}

We now study estimation of a planted vector in a subspace given an arbitrary orthonormal basis of the subspace as in Model~\ref{mod:orth}. The following result is comparable to Theorem~\ref{thm:entrywise} for the Gaussian basis, and much of the discussion in the previous section also applies here.

\begin{theorem}
\label{thm:spectral}
Fix any constant $c > 0$. Assume Model~\ref{mod:orth} with $v$ satisfying $\left| \|v\|_4^4 - \frac{3}{N} \right| \ge \frac{c}{N}$. 
Suppose that 
\begin{equation}
\eps := \frac{1}{ \|v\|_4^4 } \frac{n}{N^{3/2}}  + \frac{\|v\|_\infty}{ \|v\|_4^2 } \sqrt{\frac{n}{N}} \ll 1.
\label{eq:spec-assume}
\end{equation}
Let $\hat y_j^\top$ be the $j$th row of $\hat Y$ for $j \in [N]$, and let $\hat{\ub}$ be the leading eigenvector of the matrix $\Mh$ defined in \eqref{eq:def-mh}. 
Up to a sign flip of $\hat{\ub}$, it holds with probability $1 - o(1)$ that 
\begin{equation*}
\left| \hat y_j^\top \hat{\ub} 
- v_j \right|
\le \tilde O \left( \eps \cdot \Big( |v_j| + \frac{1}{\sqrt{N}} \Big) \right)  \quad \text{ for all } j \in [N],
\end{equation*}
and, as a result,
\begin{equation*}
\|\Yh \hat{\ub} - v\|
\le \tilde O \left( \epsilon \right) . 
\end{equation*} 
\end{theorem}

\noindent
The proof of the above theorem constitutes an important part of our theoretical contribution. 
In order to deal with the orthonormalized columns of $\Yh$ and capture the correct dependency on the parameters, we establish a series of delicate estimates based on the Sherman--Morrison formula in Section~\ref{sec:pre}. 
In particular, we use moments of an inverse Wishart matrix together with Chebyshev's inequality to establish Lemma~\ref{lem:lev} (see the proof of Lemma~\ref{lem:trace} in the appendix), a step that is not needed for the Gaussian basis or in the prior work~\cite{sos-spectral}. 
This also gives rise to the error probability $o(1)$ in the above result instead of the stronger $N^{-\omega(1)}$ as in the previous subsection. 
We do not think that this is fundamental, but improving the error probability to $N^{-\omega(1)}$ in Lemma~\ref{lem:lev} would require more technical work.
With the estimates proved in Section~\ref{sec:pre}, we eventually establish Theorem~\ref{thm:spectral} in Section~\ref{sec:pf-spectral} using a careful leave-one-out analysis.

\subsection{Application to Special Planted Vectors}
\label{sec:appl}

To further clarify the above theorems, we establish corollaries for the special cases where the planted vector $v$ has Bernoulli--Gaussian or Bernoulli--Rademacher entries as in Definitions~\ref{def:bg} and~\ref{def:br}, respectively. 
First, we have the following result for estimating a planted Bernoulli--Gaussian vector.

\begin{corollary} 
\label{cor:orth-bg}
Assume Model~\ref{mod:orth} with $v = v'/\|v'\|$ and $v' \sim \BG(N, \rho)$. 
Fix a constant $c \in (0, 0.1)$. Suppose that $\frac{1}{N} \ll \rho \le 1 - c$ and  
\begin{equation*}
\xi := \frac{n \rho}{\sqrt{N}}  + \sqrt{\frac{n}{N}} \ll 1 . 
\end{equation*}
Let $\hat y_j^\top$ be the $j$th row of $\hat Y$ for $j \in [N]$, and let $\hat{\ub}$ be the leading eigenvector of the matrix $\Mh$ defined in \eqref{eq:def-mh}. 
Up to a sign flip of $\hat{\ub}$, it holds with probability $1 - o(1)$ that 
\begin{equation*}
\left| \hat y_j^\top \hat{\ub} 
- v_j \right|
\le \tilde O \left( \frac{\xi}{\sqrt{N \rho}} \right)   \quad \text{ for all } j \in [N],
\end{equation*}
and
$$
\|\Yh \hat{\ub} - v\| \le \tilde{O} (\xi) .
$$
\end{corollary}

\noindent
In the above corollary, we assume $\rho \gg \frac{1}{N}$ to ensure that $v \sim \BG(N, \rho)$ is nonzero with high probability. 
Moreover, recall that if $\rho = 1$, then $\BG(N, \rho)$ reduces to the Gaussian distribution $\mathcal{N}(0, \frac{1}{N} I_N)$, so it is impossible to distinguish $v$ from other Gaussian vectors in the subspace. 
The above result says that, if the planted vector $v$ is $N \rho$-sparse where $\rho \le 1 - c$, and the dimension of the subspace $n$ is small compared to the ambient dimension ($n \ll N$) and
\begin{equation*} 
n \rho \ll \sqrt{N} , 
\end{equation*}
then we can estimate $v$ consistently. 
Note that a nonzero entry of $v \sim \BG(N, \rho)$ is of order $\frac{1}{\sqrt{N \rho}}$, so the entrywise error $\tilde{O} \big( \frac{\xi}{\sqrt{N \rho}} \big)$ is indeed much smaller if $\xi \ll 1$.

The next result specializes to the Bernoulli--Rademacher case. 
By virtue of the entrywise error bound in Theorem~\ref{thm:spectral} and the discrete nature of the planted vector $v$, we are able to exactly recover $v$ up to a sign flip. 

\begin{corollary} 
\label{cor:orth-br}
Assume Model~\ref{mod:orth} with $v = v'/\|v'\|$ and $v' \sim \BR(N, \rho)$. 
Fix a constant $c \in (0, 0.1)$. Suppose that $\rho \gg \frac{1}{N}$, $|\rho - \frac{1}{3}| \ge c$, and  
\begin{equation*}
\frac{n \rho}{\sqrt{N}}  + \sqrt{\frac{n}{N}} \ll 1 . 
\end{equation*}
Let $\hat y_j^\top$ be the $j$th row of $\hat Y$ for $j \in [N]$, and let $\hat{\ub}$ be the leading eigenvector of the matrix $\Mh$ defined in \eqref{eq:def-mh}. 
Up to a sign flip of $\hat{\ub}$, it holds with probability $1 - o(1)$ that 
$$
\frac{\hat v}{\|\hat v\|} = v \qquad \text{where} \qquad
\hat v := \mathsf{sign} \big(\hat Y \hat{\ub} \, \big) \cdot \mathbf{1} \left\{ \big| \hat Y \hat{\ub} \,
\big| \ge 0.5 \cdot \max_{i \in [N]} \Big| \big( \hat Y \hat{\ub} \, \big)_i \Big| \right\} , 
$$
where the functions $\mathsf{sign} (\cdot)$ and $\mathbf{1}\{|\cdot| \ge \mathrm{const}\}$ are applied entrywise. 
\end{corollary}

\noindent
Note that the estimator $\hat v / \|\hat v\|$ does not depend on $\rho$, so we do not need to know the sparsity in order to recover the planted vector $v$. 

In contrast to the Bernoulli--Gaussian case, here even in the densest case $\rho = 1$ where the planted vector $v$ has Rademacher entries, the estimator still recovers $v$ exactly as long as $n \ll \sqrt{N}$. 
This is because the $\ell_4$-norm of a Bernoulli--Rademacher vector $v \sim \BR(N, \rho)$ differs from that of a Gaussian vector 
whenever $\rho$ is bounded away from $1/3$. 
When $\rho = 1/3$, that is, when $\EE \|v\|_4^4$ coincides with that of a Gaussian vector, the spectral method based on degree-four polynomials is not expected to work. 
After the initial version of this work, a degree-six spectral method was shown to solve the case $\rho = 1/3$ when $n \ll N^{1/3}$, along with a matching low-degree lower bound~\cite[Section~6.2]{dudeja2022statistical}. Also, as discussed in Section~\ref{sec:ex-work}, new results appearing after the the initial version of this work~\cite{diakonikolas2021non,zadik2021lattice} gave an algorithm for recovering a Bernoulli--Rademacher vector under the weaker condition $N \ge n+1$ using completely different (and non-spectral) methodology that relies heavily on $v$ being (nearly) discrete.

\section{Detection of a Planted Vector and Low-Degree Lower Bounds}
\label{sec:lower}

Next, we turn to Problem~\ref{prob:detect}, detection of a planted Bernoulli--Gaussian vector in a subspace. 

\subsection{Reduction from Detection to Estimation}
\label{sec:reduction}

We first give a polynomial-time reduction showing that the estimation problem is at least as hard as the detection problem. That is, we show that any algorithm $\tilde v$ for the estimation problem can be turned into an algorithm $\tilde \psi$ for the detection problem that succeeds in the same regime of parameters.

\begin{theorem}
\label{thm:reduction}
Consider Problem~\ref{prob:detect}. 
Fix a constant $c_1 \in (0, 0.1)$. 
There exist constants $C_2, c_3 > 0$ depending only on $c_1$ such that the following holds. 
Suppose that $\frac{1}{N} \ll \rho \le 1 - c_1$ and $N \ge C_2 n$. 
Let $\tilde v = \tilde v(\tilde Y) \in \RR^N$ be any estimator of $v$ that takes values in the column span of $\tilde Y$. 
Define a test 
$$
\tilde \psi := 
\begin{cases}
\cQ & \text{ if } \Big| \|\tilde v\|_1 \big/ \|\tilde v\| - \sqrt{2 N / \pi} \Big| < c_1 \sqrt{N} \big/ 4 , \vspace{1mm} \\ 
\cP & \text{ if } \Big| \|\tilde v\|_1 \big/ \|\tilde v\| - \sqrt{2 N / \pi} \Big| \ge c_1 \sqrt{N} \big/ 4 .
\end{cases}
$$
If $\tilde v$ satisfies $\|\tilde v - v\| \le c_3$ with probability at least $1-\delta$ over $\cP$, then $\tilde \psi$ distinguishes between $\cQ$ and $\cP$ with error probability at most $\delta + N^{-\omega(1)}$; that is,
$$
\PP_{\cP}\left\{\tilde\psi = \cQ\right\} + \PP_{\cQ}\left\{\tilde\psi = \cP\right\} \le \delta + N^{-\omega(1)}.
$$
\end{theorem}

\noindent 
Recall that in Problem~\ref{prob:detect}, if $\rho = 1$, then the planted model $\cP$ coincides with the null model $\cQ$. 
Now suppose that $\rho$ is bounded away from $1$. 
The contrapositive of the above reduction says that, if there is no polynomial-time algorithm to distinguish between $\cQ$ and $\cP$ with vanishing error, then there is no polynomial-time algorithm to estimate $v$ consistently with high probability.
In the next subsection, we give evidence that the detection problem is computationally hard when $n \rho \gg \sqrt{N}$. In light of Theorem~\ref{thm:reduction}, this suggests that the estimation problem is also computationally hard in that regime. As a result, we conjecture that our upper bounds for estimation cannot be improved.

\begin{remark}
We note that our results do not formally rule out spectral and low-degree algorithms for \emph{estimating} $v$. Instead, our evidence for hardness of estimation takes the form of a two-step argument: The failure of spectral and low-degree methods for \emph{detection} (proved in the next section) leads us to conjecture that there is no polynomial-time algorithm for detection; this conjecture (if true) combined with Theorem~\ref{thm:reduction} implies that there is no polynomial-time algorithm for estimation. It may be possible to directly analyze the power of low-degree polynomials for estimation as in~\cite{schramm2022computational}, but we have not attempted this here.
\end{remark}

\subsection{Spectral Test and Lower Bounds}
\label{sec:spectral}

We now identify the regime in which the detection problem can be solved by spectral methods. To put our lower bound in perspective, we first show that if $n \rho \ll \sqrt{N}$ then the spectral method~\eqref{eq:def-mt} solves the detection problem. This follows readily from our existing analysis of the estimation problem (and is also implicit in the prior work of~\cite{sos-spectral}, as discussed at the end of Section~2.1 of~\cite{sos-spectral}).

\begin{theorem}
\label{thm:detect-upper}
Consider Problem~\ref{prob:detect}. 
Fix a constant $c \in (0, 0.1)$. 
Suppose that $\frac{1}{N} \ll \rho \le 1 - c$, $n \ll N$, and 
\begin{equation*}
n \rho \ll \sqrt{N} . 
\end{equation*}
Let $\tilde M$ be defined by \eqref{eq:def-mt}. Then it holds that
$$
\PP_{\cP} \left\{ \|\Mt\| < \frac{c}{N \rho} \right\} 
+ \PP_{\cQ} \left\{ \|\Mt\| > \frac{c}{4 N \rho} \right\}  
\le N^{-\omega(1)} . 
$$
\end{theorem}

\noindent As a result, in the regime $n \rho \ll \sqrt{N}$, thresholding $\|\tilde M\|$ at the level $\frac{c}{2 N \rho}$ succeeds in distinguishing between $\cP$ and $\cQ$ with error probability $N^{-\omega(1)}$.

Note that the entries of $\tilde M$ are degree-four polynomials in the input variables (the entries of $\tilde Y$). \emph{A priori}, we might hope to construct a better spectral method by choosing different polynomials, perhaps of higher degree. Our next result shatters these hopes, showing that when $n \rho \gg \sqrt{N}$, no spectral method of polynomial dimension and bounded-degree entries can distinguish between $\mathcal{P}$ and $\mathcal{Q}$ with error probability $N^{-\omega(1)}$. This provides a converse to Theorem~\ref{thm:detect-upper}.

\begin{theorem}
\label{thm:sp-lower-bd}
Consider Problem~\ref{prob:detect}. For any constants $\ell \ge 1$, $d \ge 1$, $\epsilon \in (0,1)$, there exist constants $N_0, C > 0$ such that the following holds. Let $N \ge N_0$, let $M = M(\Yt)$ be any real symmetric matrix of dimension at most $N^\ell$ whose entries are polynomials of degree at most $d$ in the entries of the observation $\Yt$, and let $t > 0$ be any choice of threshold. If
\begin{equation*}
n \rho \gg \sqrt{N} ,
\end{equation*}
then it is impossible for $M$ to simultaneously satisfy both of the following:
\begin{itemize}
    \item $\PP_{\cP}\left\{\|M\| < (1+\epsilon)t\right\} \le \frac{\epsilon}{4}$, and
    \item $\PP_{\cQ}\left\{\|M\| > t\right\} \le N^{-C}.$
\end{itemize}
\end{theorem}

\noindent Together, Theorems~\ref{thm:detect-upper} and~\ref{thm:sp-lower-bd} characterize the limits of spectral methods for the detection problem (modulo some technical conditions, such as $n \ll N$ and the requirement that the error probability be at most $N^{-C}$ instead of merely $o(1)$). While it is conceivable that a different class of polynomial-time algorithms could break the $n \rho \gg \sqrt{N}$ barrier, we find this unlikely because spectral methods seem to be optimal among all known polynomial-time algorithms for a wide array of high-dimensional detection problems with hidden low-dimensional structures (so long as the problems are sufficiently ``noise-robust''); see~\cite{sos-hidden}.
While the subsequent works \cite{diakonikolas2021non,zadik2021lattice} break the $n \rho \gg \sqrt{N}$ barrier for $v \sim \BR(N, \rho)$ using lattice basis reduction, the methodology relies strongly on $v$ taking (nearly) discrete values and does not apply to $v \sim \BG(N, \rho)$.

In order to prove failure of spectral methods, we actually prove failure of an even larger class of algorithms, namely low-degree polynomials. This follows a framework proposed by~\cite{HS-bayesian,sos-hidden,sam-thesis} (see also~\cite{kunisky2019notes} for a survey). Let $\RR[y]_{\le D}$ denote the space of polynomials $\RR^{N \times n} \to \RR$ (that is, multivariate polynomials whose variables are the entries of an $N \times n$ matrix) of degree at most $D$. We aim to understand the degree-$D$ ``advantage''
\begin{equation} 
\Adv_{\le D} := \max_{f \in \RR[y]_{\le D}} \frac{\EE_{\cP}[f]}{\sqrt{\EE_{\cQ}[f^2]}}. 
\label{eq:adv}
\end{equation}
This is often called the \emph{norm of the low-degree likelihood ratio}, although we do not use this term here because the likelihood ratio may not exist in our setting. This quantity can be thought of as measuring the detection power of the best possible degree-$D$ polynomial $f$, where $f$ takes the observed matrix $\tilde Y$ as input and aims to output a large value under $\mathcal{P}$ and a small (close to zero) value under $\mathcal{Q}$. If $\Adv_{\le D}$ remains bounded as $N \to \infty$ for some $D = D(N)$ that grows super-logarithmically in $N$, this means that no logarithmic-degree polynomial can effectively separate $\mathcal{P}$ from $\mathcal{Q}$, which is considered evidence that no polynomial-time algorithm can perform detection with vanishing error probability (see~\cite{sam-thesis,kunisky2019notes}). The following result shows that the planted vector problem is low-degree hard in this sense whenever $n \rho \gg \sqrt{N}$.

\begin{theorem}
\label{thm:lower}
Consider Problem~\ref{prob:detect}. 
For any constant $C_1 > 0$, there exists a constant $C_2 = C_2(C_1) > 0$ such that if $n \rho \ge \sqrt{N} (\log N)^{C_2}$ and $D \le (\log N)^{C_1}$, then $\Adv_{\le D} \le 2$. 
\end{theorem}

\noindent Theorem~\ref{thm:sp-lower-bd} will be deduced from Theorem~\ref{thm:lower} using a straightforward connection (made formal in~\cite{kunisky2019notes}) between spectral methods and logarithmic-degree polynomials: For a symmetric matrix $M$ and an integer $k = O(\log N)$, the polynomial $\Tr(M^{2k})$ serves as a good approximation to $\|M\|^{2k}$.
The proof of Theorem~\ref{thm:lower} is where most of the technical work takes place (see Sections~\ref{sec:low-deg-pf} and~\ref{sec:pf-adv}). 
In particular, we reformulate Problem~\ref{prob:detect} as Problem~\ref{prob:alt} (a fairly general variant of Model~\ref{mod:gauss-alt}).
Then, we establish a key ingredient of the proof---a new formula (Lemma~\ref{lem:formula}) for $\Adv_{\le D}$ in Problem~\ref{prob:alt}. Finally, by carefully controlling each term in this formula, we complete the proof in Section~\ref{sec:pf-adv}.

\medskip

Before ending the section, we remark that all the results in this section are valid with very minor changes if the planted vector $v$ is Bernoulli--Rademacher $\BR(N, \rho)$ instead of Bernoulli--Gaussian $\BG(N, \rho)$ in model $\cP$ in Problem~\ref{prob:detect}.
The only difference is that in Theorems~\ref{thm:reduction} and~\ref{thm:detect-upper}, instead of assuming $\rho \le 1 - c$ for a constant $c > 0$, we require $\rho$ to be bounded away from a different constant ($2/\pi$ and $1/3$, respectively).
The proofs carry over with straightforward modifications; for details, see an earlier version of this work \cite{mao2021optimal}.

\section{Proofs of Upper Bounds}
\label{sec:proof-upper}

In this section, we prove Theorem~\ref{thm:spectral} and establish other results along the way that provide intuition behind our upper bounds. 
The proofs of Theorem~\ref{thm:entrywise} and Corollaries~\ref{cor:orth-bg} and~\ref{cor:orth-br} are deferred to the appendix.

\subsection{A First Result and its Implication}
\label{sec:first}

As a first step towards proving upper bounds for our spectral estimator, we show that the matrix $\Mt$ defined in \eqref{eq:def-mt} is approximately rank-one.

\begin{proposition}
\label{prop:gauss-spec-norm}
Fix a vector $v \in \RR^N$. 
For $i \in [N]$, define $\yg_i^\top = ( v_i \; b_i^\top )$ for i.i.d.\ $b_i \sim \mathcal{N}(0, \frac{1}{N} I_{n-1})$. 
Define an $n \times n$ matrix
\begin{equation}
\Mg = \sum_{i=1}^N \left( \|\yg_i\|^2 - \frac{n-1}{N} \right) \yg_i \yg_i^\top - \frac{3}{N} I_n . 
\label{eq:def-mg}
\end{equation}
There is a universal constant $C>0$ such that the following holds. For any $\delta \in (0,1)$, define a quantity
\begin{align}
\eta = \eta(N,n,v,\delta) 
&:= C \bigg( \left( \|v\| + 1 \right) \frac{n \sqrt{\log(N/\delta)}}{N^{3/2}} + \|v\|_4^2 \frac{\sqrt{n \log(N/\delta)}}{N} \notag \\
& \qquad + \|v\|_6^3 \frac{ \sqrt{n} + \sqrt{\log(N/\delta)} }{\sqrt{N}} + \|v\|_\infty^2 \frac{n \log(N/\delta) + \log^2(N/\delta)}{N} \notag\\
& \qquad + \frac{ (n \log(N/\delta))^{3/2} }{N^2} + \frac{1}{N} \left| \|v\|^2 - 1 \right| \bigg).  \label{eq:def-eta}
\end{align} 
Suppose that $N \ge \log^7(N/\delta)$. 
Let $e_1$ be the first standard basis vector in $\RR^n$. 
Then it holds with probability at least $1-\delta$ that 
\begin{equation}
\left\| \Mg - \left( \|v\|_4^4 - \frac{3}{N} \right) e_1 e_1^\top \right\|
\le \eta . 
\label{eq:spec-bd}
\end{equation} 
\end{proposition}

\noindent
The proof of the above result is deferred to Section~\ref{sec:pf-gauss-spec-norm}. 
Since $\Mt$ defined in \eqref{eq:def-mt} is simply a rotated version of $\Mg$ defined in \eqref{eq:def-mg}, both are approximately rank-one by the above proposition. 
Therefore, the leading eigenvector of $\Mt$ is crucial in estimating the planted vector. 

Furthermore, in view of \eqref{eq:spec-bd}, it is reasonable to expect an $\ell_2$ error bound of order 
$$
\frac{ \eta }{ \left| \|v\|_4^4 - \frac{3}{N} \right|} , 
$$
where $\eta$ and $\left| \|v\|_4^4 - \frac{3}{N} \right|$ can be seen as the noise level and the signal level respectively. 
(Again, the discrepancy between $\|v\|_4^4$ and the corresponding quantity for a Gaussian vector is naturally the signal level.)
To see how the above order of error is related to the order of error \eqref{eq:def-eps} in Theorem~\ref{thm:entrywise}, we state a simple lemma whose proof is deferred to Section~\ref{sec:change-norm}.

\begin{lemma}
\label{lem:change-norm}
Let $c > 0$ and let $v$ be a unit vector $v \in \RR^N$ such that $\left|\|v\|_4^4 - \frac{3}{N} \right| \ge \frac{c}{N}$. Then
\begin{align*}
\frac{ 1 }{ \left| \|v\|_4^4 - \frac{3}{N} \right|} \bigg( \frac{n}{N^{3/2}} 
+ \|v\|_4^2 \frac{\sqrt{n}}{N} 
+ \|v\|_6^3 \sqrt{\frac{n}{N}} 
+&  \|v\|_\infty^2 \frac{n}{N} \bigg) \\
& \le \frac{9+3c}{c} \bigg( \frac{1}{ \|v\|_4^4 } \frac{n}{N^{3/2}} 
+ \frac{\|v\|_\infty}{ \|v\|_4^2 } \sqrt{\frac{n}{N}} \, \bigg) 
\end{align*}
provided that the right-hand side is at most $1$. 
\end{lemma}

\noindent
As a result, if $\|v\| = 1$, $\left| \|v\|_4^4 - \frac{3}{N} \right| \ge \frac{c}{N}$, and $\eta = \eta(N, n, v, N^{-\omega(1)})$, then 
\begin{align}
\frac{ \eta }{ \left| \|v\|_4^4 - \frac{3}{N} \right|} 
&\le \frac{ 1 }{ \left| \|v\|_4^4 - \frac{3}{N} \right|} \cdot \tilde O
\bigg( \frac{n}{N^{3/2}} + \|v\|_4^2 \frac{\sqrt{n}}{N} + \|v\|_6^3 \sqrt{\frac{n}{N}} +  \|v\|_\infty^2 \frac{n}{N} \bigg) \notag \\
&\le \tilde O \left( \frac{1}{ \|v\|_4^4 } \frac{n}{N^{3/2}}  + \frac{\|v\|_\infty}{ \|v\|_4^2 } \sqrt{\frac{n}{N}} \, \right) ,  \label{eq:l2-bd-simple}
\end{align}
which is why we have the bound \eqref{eq:simplified}.

\subsection{Setup and Lemmas for the Orthonormal Case}
\label{sec:pre}

We now establish some notation and preliminary results for estimation of a planted vector from an orthonormal basis. 
Assume that $N \ge C n$ for a sufficiently large constant $C > 0$ throughout this section. 
Fix a unit vector $v \in \RR^N$. 
Let $\yg_i^\top = (v_i \; b_i^\top)$ be the rows of $\Yg$ defined in Model~\ref{mod:orth}, where $b_i$ are i.i.d.\ $\mathcal{N}(0, \frac{1}{N} I_{n-1})$ vectors for $i \in [N]$. 
Let $\tilde b_i \in \RR^n$ be defined by $\tilde b_i^\top = (0 \; b_i^\top)$. 
Let $e_1$ be the first standard basis vector in $\RR^n$. 
Define
\begin{equation}
A := 
\sum_{i=1}^N \yg_i \yg_i^\top , \qquad 
B := \sum_{i=1}^N b_i b_i^\top
\label{eq:def-ab} . 
\end{equation}
Moreover, for $i \in [N]$, define
\begin{equation}
A_{-i} := \sum_{j \in [N] \setminus \{i\}} y_j y_j^\top + v_i^2 e_1 e_1^\top , \qquad 
B_{-i} := \sum_{j \in [N] \setminus \{i\}} b_j b_j^\top .
\label{eq:def-ab-i}
\end{equation}
We write $A_{-i}^{-1} = (A_{-i})^{-1}$ and $B_{-i}^{-1} = (B_{-i})^{-1}$. 
The following results hold, and their proofs can be found in Section~\ref{sec:prelim}.

\begin{lemma}
\label{lem:cov}
It holds with probability $1-N^{-\omega(1)}$ that,
\begin{equation}
\|A - I_n\| = \tilde O \left( \sqrt{ \frac{n}{N} } \, \right) .
\label{eq:a-a-i-norm}
\end{equation}
\end{lemma}

\begin{lemma}\label{lem:A-exp}
It holds with probability $1-N^{-\omega(1)}$ that, for every $i \in [N]$,
\begin{subequations} \label{eq:eqs-4}
\begin{equation}
e_1^\top A^{-1} \tilde e_1 = \frac{1}{1-n/N} + \tilde O \left( \frac{\sqrt{n}}{N} \right) , 
\end{equation}
\begin{equation} 
e_1^\top A^{-1} \tilde b_i =  - \frac{v_i n}{N - n} + \tilde O \left( \frac{\sqrt{n}}{N} \right)  ,  \label{eq:e-a-b}
\end{equation}
\begin{equation}
\tilde b_i^\top A^{-1} \tilde b_i =  b_i^\top B^{-1} b_i + \frac{v_i^2 n^2}{N(N-n)} + \tilde O \left( |v_i| \frac{n^{3/2}}{N^2} + \frac{n}{N^2} \right) , \label{eq:b-a-b-1}
\end{equation}
\begin{equation}
\yg_i^\top A^{-1} \yg_i = b_i^\top B^{-1} b_i + (1-n/N) v_i^2 + \tilde{O}\left(|v_i| \frac{\sqrt n}{N} + \frac{n}{N^2}\right) . \label{eq:y-a-y-1}
\end{equation}
\end{subequations}
\end{lemma}

\begin{lemma}\label{lem:A-j}
It holds with probability $1-N^{-\omega(1)}$ that, for any distinct $i, j \in [N]$,
\begin{subequations} \label{eq:eqs-2}
\begin{equation}
\left| e_1^\top (A^{-1} - A_{-j}^{-1}) e_1 \right| \le \tilde{O}\left( v_j^2 \frac{n}{N} + \frac{\sqrt{n}}{N} \right) , \label{eq:e-a-a-e}
\end{equation}
\begin{align}\label{eq:y-a-a-y}
&\left|y_i^\top (A^{-1} - A_{-j}^{-1}) y_i\right| \\
& \qquad\le \tilde O \left( \frac{n}{N^2} 
+ |v_i| \cdot |v_j| \frac{\sqrt{n}}{N} + \big( |v_i| + |v_j| \big) \frac{n}{N^2} 
+ v_i^2 v_j^2 \frac{n}{N} + v_j^2 \frac{n^2}{N^3} 
+ v_i^2 \frac{\sqrt{n}}{N} \right) . \notag
\end{align}
\end{subequations}
\end{lemma}

\begin{lemma}\label{lem:lev}
It holds with probability $1-o(1)$ that, for every $i \in [N]$,
\begin{subequations} \label{eq:eqs-3}
\begin{equation}
\Big| \yg_i^\top A^{-1} \yg_i - \|\yg_i\|^2 \Big| = \tilde{O}\left(\frac{n}{N^{3/2}} + |v_i| \frac{\sqrt n}{N} + v_i^2 \frac{n}{N}\right) , 
\label{eq:lev}
\end{equation}
\begin{equation}
\Big| \yg_i^\top A^{-1} \yg_i - \frac{n}{N} \Big| = \tilde{O}\left(\frac{\sqrt{n}}{N} + v_i^2 \right) ,  
\label{eq:lev2}
\end{equation}
\begin{equation}
\Big| \tilde b_i^\top A^{-1} \tilde b_i - \frac{n}{N} \Big| = \tilde{O}\left(\frac{\sqrt{n}}{N} + v_i^2 \frac{n^2}{N^2}\right) . 
\label{eq:b-a-b}
\end{equation}
\end{subequations}
\end{lemma}

\subsection{Proof of Theorem~\ref{thm:spectral}}
\label{sec:pf-spectral}

Throughout the proof, we condition on the event of probability $1 - o(1)$ that the bounds \eqref{eq:spec-bd}, \eqref{eq:a-a-i-norm}, \eqref{eq:eqs-4}, \eqref{eq:eqs-2}, and \eqref{eq:eqs-3} all hold. 

\subsubsection*{Change of basis} 
Let $A$ and $B$ be the matrices defined in \eqref{eq:def-ab}. 
Define 
\begin{equation*}
\yb_i = A^{-1/2} \yg_i \quad \text{ for } i \in [N] , 
\end{equation*}
and let $\Yb$ be the $N \times n$ matrix whose $i$th row is $\yb_i^\top$, that is, $\Yb = \Yg A^{-1/2}$. 
Note that 
$$
\Yb^\top \Yb = A^{-1/2} \Yg^\top \Yg A^{-1/2}
= A^{-1/2} A A^{-1/2} = I_n ,
$$
so the columns of $\Yb$ form an orthonormal basis of the column span of $\Yg$. 
The same is true for the columns of $\Yh$ observed in Model~\ref{mod:orth}, so there exists an orthogonal matrix $Q \in \RR^{n \times n}$ such that $\Yh Q = \Yb $, that is, 
\begin{equation*}
Q^\top \yh_i = \yb_i \quad \text{ for } i \in [N] .
\end{equation*}

Let us define  
an auxiliary matrix
\begin{equation*}
M' := A^{1/2} Q^\top \Mh Q A^{1/2} . 
\end{equation*}
By the definition of $\Mh$ in \eqref{eq:def-mh} and the relation $A^{1/2} Q^\top \yh_i = A^{1/2} \yb_i = \yg_i$, we obtain 
\begin{equation}
M' = \sum_{i=1}^N \left( \yg_i^\top A^{-1} \yg_i - \frac{n-1}{N} \right) \yg_i \yg_i^\top - \frac{3}{N} A . 
\label{eq:formula-m-p}
\end{equation}
Let $\ub'$ be the leading eigenvector of $M'$ so that 
$\hat{\ub} = \frac{Q A^{1/2} \ub'}{ \| Q A^{1/2} \ub' \| }$. 
Then we have 
\begin{equation*}
\Yh \hat{\ub} 
=  \left(\Yh Q A^{1/2}\right) \left(A^{-1/2} Q^\top \hat{\ub}\right) 
= \frac{ Y \ub'}{ \| Q A^{1/2} \ub' \| } . 
\end{equation*}
It follows from \eqref{eq:a-a-i-norm} that $\frac{1}{\| Q A^{1/2} \ub' \|} = 1 + \tilde O \big( \sqrt{\frac{n}{N}} \, \big)$ and therefore 
\begin{equation}
\hat y_j^\top \hat{\ub} = y_j^\top \ub' \left(1 + \tilde O \bigg( \sqrt{\frac{n}{N}} \, \bigg) \right) .  \label{eq:hat-prime-y}
\end{equation}
It remains to analyze $y_j^\top \ub'$. 

\subsubsection*{Setting up the leave-one-out analysis}
To study $y_j^\top \ub'$, we employ a leave-one-out analysis.  
In addition to the matrices $M$ and $M'$, we will define auxiliary matrices $F^{(j)}$, $G^{(j)}$, and $H^{(j)}$; as we will show, all the five matrices are close to $(\|v\|_4^4 - \frac{3}{N}) e_1 e_1^\top$. 

Recall that $y_j^\top = (v_j \; b_j^\top)$ where $b_j \sim \mathcal{N}(0, \frac{1}{N} I_{n-1})$. As before, let $\tilde b_j \in \RR^n$ be the vector defined by $\tilde b_j^\top = (0 \;\; b_j^\top)$. 
Let $A_{-j}$ be defined by \eqref{eq:def-ab-i}. 
Define matrices 
\begin{equation}
F^{(j)} := \sum_{i \in [N] \setminus \{j\}} \left( y_i^\top A_{-j}^{-1} y_i - \frac{n+2}{N} \right) y_i y_i^\top  + v_j^4 \left( e_1^\top A_{-j}^{-1} e_1  \right) e_1 e_1^\top ,
\label{eq:def-fj}
\end{equation}
\begin{equation}
G^{(j)} 
:= F^{(j)} + \left( y_j^\top A^{-1} y_j - \frac{n+2}{N} \right) \tilde b_j \tilde b_j^\top , \label{eq:def-gj}
\end{equation}
and
\begin{equation}
H^{(j)} := \sum_{i \in [N] \setminus \{j\}} \left( y_i^\top A_{-j}^{-1} y_i - \frac{n+2}{N} \right) y_i y_i^\top  
+ \left( y_j^\top A^{-1} y_j - \frac{n+2}{N} \right) y_j y_j^\top . 
\label{eq:def-hj}
\end{equation}
Since
\begin{align*}
&\left( y_j^\top A^{-1} y_j - \frac{n+2}{N} \right) y_j y_j^\top \\
&\qquad= v_j^4 \left( e_1^\top A^{-1} e_1 \right) e_1 e_1^\top 
+ v_j^2 \left( 2 v_j \cdot e_1^\top A^{-1} \tilde b_j + \tilde b_j^\top A^{-1} \tilde b_j - \frac{n+2}{N} \right) e_1 e_1^\top \\
& \qquad\quad \   + v_j \left( y_j^\top A^{-1} y_j - \frac{n+2}{N} \right) (e_1 \tilde b_j^\top + \tilde b_j e_1^\top) 
+ \left( y_j^\top A^{-1} y_j - \frac{n+2}{N} \right) \tilde b_j \tilde b_j^\top , 
\end{align*}
it is not hard to see that 
\begin{align}
H^{(j)} - G^{(j)}& \notag \\
&\hspace{-15pt} = v_j^4 \cdot e_1^\top \left(A^{-1} - A_{-j}^{-1}\right) e_1 \cdot e_1 e_1^\top 
+ v_j^2 \left( 2 v_j \cdot e_1^\top A^{-1} \tilde b_j + \tilde b_j^\top A^{-1} \tilde b_j - \frac{n+2}{N} \right) e_1 e_1^\top \notag \\
&\quad \hspace{-15pt}
+ v_j \left( y_j^\top A^{-1} y_j - \frac{n+2}{N} \right) (e_1 \tilde b_j^\top + \tilde b_j e_1^\top) . 
\label{eq:h-g-diff}
\end{align}
Moreover, in view of \eqref{eq:formula-m-p}, \eqref{eq:def-ab}, and \eqref{eq:def-hj}, we have 
\begin{equation}
M' - H^{(j)} = \sum_{i \in [N] \setminus \{j\}} y_i^\top (A^{-1} - A_{-j}^{-1}) y_i \cdot y_i y_i^\top . 
\label{eq:m-h-diff}
\end{equation}
By \eqref{eq:def-mg} and \eqref{eq:formula-m-p}, we also have 
\begin{equation}
M' - \Mg 
= \sum_{i=1}^N \left( \yg_i^\top A^{-1} \yg_i - \| \yg_i \|^2 \right) \yg_i \yg_i^\top - \frac{3}{N} (A - I_n) ,
\label{eq:diff}
\end{equation}

Let $\fb$ and $\gb$ be the leading eigenvectors of $F^{(j)}$ and $G^{(j)}$ respectively, where we suppress the dependency of $\fb$ and $\gb$ on $j$ for brevity. 
By virtue of the definition of $F^{(j)}$ in \eqref{eq:def-fj}, $y_j$ is independent from $F^{(j)}$ and thus from $\fb$. This independence allows us to study $y_j^\top \fb$. 
Then, to study $y_j^\top \ub'$, we write
\begin{equation*}
y_j^\top \ub' = y_j^\top \fb + y_j^\top (\gb - \fb) + y_j^\top (\ub' - \gb) 
\end{equation*}
so that 
\begin{equation}
\big| y_j^\top \ub' - v_j \big| \le \big| y_j^\top \fb - v_j \big| + \| y_j \| \cdot \big\| \gb - \fb \big\| + \| y_j \| \cdot \big\| \ub' - \gb \big\| . 
\label{eq:main-ineq}
\end{equation}
In the sequel, we bound the terms on the right-hand side. 

\subsubsection*{Spectral norm bounds}
We now show that the matrices $F^{(j)}, G^{(j)}, H^{(j)}, M',$ and $M$ are all close to each other in spectral norm. 
By \eqref{eq:def-gj}, \eqref{eq:lev2}, and $\|b_j\|^2 = \tilde O(n/N)$, we have 
\begin{equation}
\|G^{(j)} - F^{(j)}\| \le \left| y_j^\top A^{-1} y_j - \frac{n+2}{N} \right| \|b_j\|^2
\le \tilde{O}\left(\frac{n^{3/2}}{N^2} + v_j^2 \frac{n}{N}\right) . 
\label{eq:f-g-norm}
\end{equation}
Next, it follows from \eqref{eq:h-g-diff} that 
\begin{align*}
\|H^{(j)} - G^{(j)}\|
&\le v_j^4 \cdot \left|e_1^\top \left(A^{-1} - A_{-j}^{-1}\right) e_1\right| 
+ 2 |v_j|^3 \cdot \left|e_1^\top A^{-1} \tilde b_j\right| + v_j^2 \cdot \left|\tilde b_j^\top A^{-1} \tilde b_j - \frac{n+2}{N} \right| \\
& \quad \  
+ 2 |v_j| \cdot \left| y_j^\top A^{-1} y_j - \frac{n+2}{N} \right| \cdot \|b_j\| . 
\end{align*}
To bound the terms on the right-hand side, we apply \eqref{eq:e-a-a-e}, \eqref{eq:e-a-b}, \eqref{eq:b-a-b}, and \eqref{eq:lev2} to obtain 
\begin{align}
\|H^{(j)} - G^{(j)}\| & \notag \\
&\hspace{-40pt}\le \tilde{O}\left( v_j^6 \frac{n}{N} + v_j^4 \frac{\sqrt{n}}{N} 
+ v_j^4 \frac{n}{N} + |v_j|^3 \frac{\sqrt{n}}{N} 
+ v_j^2 \frac{\sqrt{n}}{N} + v_j^4 \frac{n^2}{N^2} 
+ |v_j| \frac{n}{N^{3/2}} + |v_j|^3 \frac{\sqrt{n}}{\sqrt{N}} \right) \notag \\
&\hspace{-40pt}\le \tilde{O}\left( v_j^4 \frac{n}{N} 
+ |v_j|^3 \frac{\sqrt{n}}{\sqrt{N}}
+ v_j^2 \frac{\sqrt{n}}{N}
+ |v_j| \frac{n}{N^{3/2}} \right) . \label{eq:h-g-norm}
\end{align}
Moreover, by \eqref{eq:m-h-diff}, \eqref{eq:y-a-a-y}, and \eqref{eq:a-a-i-norm}, we have 
\begin{align*}
\|M' - H^{(j)}\| 
&\le \bigg\| \sum_{i \in [N] \setminus \{j\}} \left| y_i^\top (A^{-1} - A_{-j}^{-1}) y_i\right| \cdot y_i y_i^\top \bigg\| \\
&\le \tilde O \left( \frac{n}{N^2} 
+ |v_j| \cdot \|v\|_\infty \frac{\sqrt{n}}{N} + \|v\|_\infty \frac{n}{N^2} 
+ v_j^2 \|v\|_\infty^2 \frac{n}{N} + \|v\|_\infty^2 \frac{\sqrt{n}}{N} \right) \\
&\qquad \times \bigg\| \sum_{i \in [N] \setminus \{j\}} y_i y_i^\top \bigg\| .
\end{align*}
Since $\big\| \sum_{i \in [N] \setminus \{j\}} y_i y_i^\top \big\| \le \|A\| + \|y_j\|^2 \le \tilde O(1) + v_j^2 + \|b_j\|^2 = \tilde O(1)$, it follows that 
\begin{equation}
\|M' - H^{(j)}\| \le \tilde O \left( \frac{n}{N^2} 
+ |v_j| \cdot \|v\|_\infty \frac{\sqrt{n}}{N} + \|v\|_\infty \frac{n}{N^2} 
+ v_j^2 \|v\|_\infty^2 \frac{n}{N} + \|v\|_\infty^2 \frac{\sqrt{n}}{N} \right) . \label{eq:mp-h-norm}
\end{equation}
Finally, in view of \eqref{eq:diff}, \eqref{eq:lev}, and \eqref{eq:a-a-i-norm}, it holds that 
\begin{align}
\left\| M' - \Mg \right\| &\le 
\left\| \sum_{i=1}^N \left| \yg_i^\top A^{-1} \yg_i - \| \yg_i \|^2 \right| \cdot \yg_i \yg_i^\top \right\| + \frac{3}{N} \left\| A - I_n \right\| \notag \\
& \le 
\tilde{O}\left(\frac{n}{N^{3/2}} + \|v\|_\infty \frac{\sqrt n}{N} + \|v\|_\infty^2 \frac{n}{N}\right)  \cdot \left\| \sum_{i=1}^N \yg_i \yg_i^\top \right\| + \frac{1}{N} \cdot \tilde O \left( \sqrt{ \frac{n}{N} } \, \right) \notag \\
&\le \tilde{O}\left(\frac{n}{N^{3/2}} + \|v\|_\infty \frac{\sqrt n}{N} + \|v\|_\infty^2 \frac{n}{N}\right) . \label{eq:b23} 
\end{align}

It is not hard to see that the bounds \eqref{eq:f-g-norm}, \eqref{eq:h-g-norm}, \eqref{eq:mp-h-norm}, and \eqref{eq:b23} sum to 
\begin{equation*}
\tilde{O}\left( \|v\|_\infty^3 \frac{\sqrt{n}}{\sqrt{N}}
+ \|v\|_\infty \frac{\sqrt n}{N} 
+ \frac{n}{N^{3/2}} + \|v\|_\infty^2 \frac{n}{N} \right) 
\le \tilde O\left( \kappa \right) , 
\end{equation*}
where
\begin{equation}
\kappa := \frac{n}{N^{3/2}} 
+ \|v\|_4^2 \frac{\sqrt{n}}{N} 
+ \|v\|_6^3 \sqrt{\frac{n}{N}} +  \|v\|_\infty^2 \frac{n}{N} + \|v\|_\infty \frac{\sqrt n}{N} .
\label{eq:def-kappa}
\end{equation}
In addition, recall that $\left\| M - \left( \|v\|_4^4 - \frac{3}{N} \right) e_1 e_1^\top \right\| \le \eta \le \tilde O(\kappa)$ by \eqref{eq:spec-bd}. 
Furthermore, by Lemma~\ref{lem:change-norm} and assumption \eqref{eq:spec-assume}, we have 
\begin{align}
\frac{\kappa}{\left| \|v\|_4^4 - \frac{3}{N} \right|} 
&\le O \left( \frac{1}{ \|v\|_4^4 } \frac{n}{N^{3/2}}  + \frac{\|v\|_\infty}{ \|v\|_4^2 } \sqrt{\frac{n}{N}} 
+  \frac{\|v\|_\infty}{ \|v\|_4^4 } \frac{\sqrt n}{N} \right) \nonumber \\
&\le O \left( \frac{1}{ \|v\|_4^4 } \frac{n}{N^{3/2}}  + \frac{\|v\|_\infty}{ \|v\|_4^2 } \sqrt{\frac{n}{N}} \right) \ll 1, 
\label{eq:kappa-l4-ratio}
\end{align}
where we have used the fact $\|v\|_4^4 \ge 1/N$ (since $v$ is a unit vector) to bound the third term by the second. Therefore, the auxiliary matrices $F^{(j)}, G^{(j)}, H^{(j)}$, and $M'$ are all close to $\left( \|v\|_4^4 - \frac{3}{N} \right) e_1 e_1^\top$. 
In particular, the gap between the two largest singular values of each of these matrices is larger than $0.9 \left| \|v\|_4^4 - \frac{3}{N} \right|$, and the eigenvectors $\fb$, $\gb$, $\ub'$, and $\ub$ (of $F^{(j)}, G^{(j)}, M'$, and $M$ respectively) can be chosen to be close to $e_1$ so that there is no sign ambiguity. 

\subsubsection*{Eigenvector perturbation}
More precisely, by the Davis--Kahan theorem (Lemma~\ref{lem:dk}), we have 
\begin{equation}
\|\fb - e_1\| 
\le 6 \frac{\|F^{(j)} - M\|}{\left| \|v\|_4^4 - \frac{3}{N} \right|} 
\le \frac{\tilde O\left(\kappa\right)}{\left| \|v\|_4^4 - \frac{3}{N} \right|}
\label{eq:f-e1}
\end{equation}
and, in view of \eqref{eq:mp-h-norm} and \eqref{eq:h-g-norm},  
\begin{align}
\|\ub' - \gb\| &\le 6 \frac{\|M' - H^{(j)}\| + \|H^{(j)} - G^{(j)}\|}{\left| \|v\|_4^4 - \frac{3}{N} \right|} \notag \\
&\le \frac{1}{\left| \|v\|_4^4 - \frac{3}{N} \right|}  \times  \label{eq:xip-g} \\
&\hspace{-40pt} \tilde{O} \bigg( \frac{n}{N^2} 
+ |v_j| \cdot \|v\|_\infty \frac{\sqrt{n}}{N} + \|v\|_\infty \frac{n}{N^2} 
+ v_j^2 \|v\|_\infty^2 \frac{n}{N} 
+ \|v\|_\infty^2 \frac{\sqrt{n}}{N}
+ |v_j|^3 \frac{\sqrt{n}}{\sqrt{N}}
+ |v_j| \frac{n}{N^{3/2}} \bigg) . \notag
\end{align}

Furthermore, since $b_j$ is independent from $F^{(j)}$ and thus from $\fb$, we have $\tilde b_j^\top \fb \sim \mathcal{N}(0, \|\fb_{-1}\|^2/N)$ conditional on $F^{(j)}$, where $\fb_{-1}$ denotes the subvector of $\fb$ with its first entry removed. 
Consequently, by \eqref{eq:f-e1}, 
\begin{equation}
\left| \tilde b_j^\top \fb \right| \le \tilde O \left( \frac{\|\fb_{-1}\|}{\sqrt{N}} \right) 
\le \frac{\tilde O\left(\kappa\right)}{\sqrt{N} \left| \|v\|_4^4 - \frac{3}{N} \right|}
\le \tilde O\left(\frac{1}{\sqrt{N}}\right)
\label{eq:b-f-ip}
\end{equation}
with conditional probability $1-N^{-\omega(1)}$. 
Let us further condition on \eqref{eq:b-f-ip}. 
Then, by \eqref{eq:def-gj}, Lemma~\ref{lem:rank-one-perturb}, and \eqref{eq:lev2}, we obtain 
\begin{align}
\|\gb - \fb\| &\le \frac{ 6 }{\left| \|v\|_4^4 - \frac{3}{N} \right|} \cdot  \left| y_j^\top A^{-1} y_j - \frac{n+2}{N} \right| \cdot \|b_j\| \cdot \left| \tilde b_j^\top \fb \right| \notag \\
&\le \frac{1}{\left| \|v\|_4^4 - \frac{3}{N} \right|} \cdot \tilde O\left( \left( \frac{\sqrt{n}}{N} + v_j^2 \right) \cdot \sqrt{\frac{n}{N}} \cdot \frac{1}{\sqrt{N}} \right) \notag \\
&\le \frac{1}{\left| \|v\|_4^4 - \frac{3}{N} \right|} \cdot \tilde O\left( \frac{n}{N^2} + v_j^2 \frac{\sqrt{n}}{N} \right) .  \label{eq:g-f}
\end{align}

\subsubsection*{Finishing the proof} 
Finally, it follows from \eqref{eq:f-e1} and \eqref{eq:b-f-ip} that 
$$
\left| y_j^\top \fb - v_j \right| 
\le |v_j| (1 - \fb_1) + \left| \tilde b_j^\top \fb \right| \le \frac{\tilde O\left(\kappa\right)}{\left| \|v\|_4^4 - \frac{3}{N} \right|} \left( |v_j| + \frac{1}{\sqrt{N}} \right) . 
$$
Plugging this together with \eqref{eq:xip-g}, \eqref{eq:g-f}, and $\|y_j\|^2 = v_j^2 + \tilde O(n/N)$ into \eqref{eq:main-ineq}, we get 
\begin{align*}
\big| y_j^\top \ub' - v_j \big| &
\le \frac{\tilde O\left(\kappa\right)}{\left| \|v\|_4^4 - \frac{3}{N} \right|} \left( |v_j| + \frac{1}{\sqrt{N}} \right) 
+ \tilde O \left( v_j^2 + \frac{n}{N} \right) \cdot \frac{1}{\left| \|v\|_4^4 - \frac{3}{N} \right|} \times \\
&\hspace{-50pt} \tilde{O} \bigg( \frac{n}{N^2} 
+ |v_j| \cdot \|v\|_\infty \frac{\sqrt{n}}{N} + \|v\|_\infty \frac{n}{N^2} 
+ v_j^2 \|v\|_\infty^2 \frac{n}{N} 
+ \|v\|_\infty^2 \frac{\sqrt{n}}{N}
+ |v_j|^3 \frac{\sqrt{n}}{\sqrt{N}}
+ |v_j| \frac{n}{N^{3/2}} \bigg) \\
&\le \frac{\tilde O\left(\kappa\right)}{\left| \|v\|_4^4 - \frac{3}{N} \right|} \left( |v_j| + \frac{1}{\sqrt{N}} \right)  .
\end{align*}
This combined with \eqref{eq:hat-prime-y} yields 
\begin{align*}
\big| \hat y_j^\top \hat{\ub} - v_j \big| 
&\le \big| y_j^\top \ub' - v_j \big| + \tilde O\left( \big| y_j^\top \ub' \big| \sqrt{\frac{n}{N}} \, \right) \\
&\le \frac{\tilde O\left(\kappa\right)}{\left| \|v\|_4^4 - \frac{3}{N} \right|} \left( |v_j| + \frac{1}{\sqrt{N}} \right) 
+ \tilde O\left( |v_j| \sqrt{\frac{n}{N}} \, \right) , 
\end{align*}
which completes the proof in view of \eqref{eq:kappa-l4-ratio}, along with the fact $\|v\|_\infty = \|v\|_\infty \cdot \|v\| \ge \|v\|_4^2$.

\section{Proofs of Lower Bounds}
\label{sec:proof-lower}

We prove Theorems~\ref{thm:sp-lower-bd} and~\ref{thm:lower} in this section and defer the proofs of Theorems~\ref{thm:reduction} and~\ref{thm:detect-upper} to the appendix.

\subsection{Proof of Theorem~\ref{thm:sp-lower-bd}}

Assume on the contrary that for some $\ell \ge 1$, $d \ge 1$, $\epsilon \in (0,1)$, such an $M$ exists. First note that the assumption $\PP_{\cP}\left\{\|M\| \ge (1+\epsilon)t\right\} \ge 1 - \frac{\epsilon}{4}$ implies $\EE_{\cP} \|M\| \ge (1+\epsilon/2)t$. By Theorem~4.4 of~\cite{kunisky2019notes}, if $k \in \mathbb{N}$ satisfies $N^{-C} \le \frac{1}{2} \cdot 3^{-4kd}$ then $\Adv_{\le 2kd} \ge N^{-\ell}(1+\epsilon/2)^{2k}$ (the requirement in~\cite{kunisky2019notes} that $\mathcal{P}$ be absolutely continuous with respect to $\mathcal{Q}$ is not actually needed). By Theorem~\ref{thm:lower}, we have a universal constant $C_2 > 0$ such that if $n \rho \ge \sqrt{N} (\log N)^{C_2}$ and $2kd \le (\log N)^2$, then $\Adv_{\le 2kd} \le 2$. Seeking a contradiction, choose
\[ k = \left\lfloor\frac{\ell \log N + \log 2}{2 \log(1+\epsilon/2)}\right\rfloor + 1 \]
so that $N^{-\ell} (1+\epsilon/2)^{2k} > 2$. Since $k = O(\log N)$, we can choose $N_0$ large enough so that $2kd \le (\log N)^2$ for all $N \ge N_0$. Finally, using again that $k = O(\log N)$, we can choose $C$ large enough so that $N^{-C} \le \frac{1}{2} \cdot 3^{-4kd}$ for all $N \ge N_0$.

\subsection{Setup and Lemmas for Low-Degree Lower Bounds}
\label{sec:low-deg-pf}

We begin by restating Problem~\ref{prob:detect} in an equivalent form based on Model~\ref{mod:gauss-alt}. 
Recall that in Model~\ref{mod:gauss-alt}, we observe i.i.d.\ samples $\tilde y_i \in \RR^n$, each of which has a non-Gaussian component along a planted direction $u$ and is otherwise Gaussian.

\begin{problem}
\label{prob:alt}
Let $\mathcal{U}$ be a distribution over the unit sphere in $\RR^n$, and let $\nu$ be a distribution over $\RR$.
Define the following null and planted distributions: 
\begin{itemize}
\item 
Under $\cQ$, observe i.i.d.\ samples $\yt_1,\ldots,\yt_N \sim \mathcal{N}(0,I_n)$.

\item 
Under $\cP$, first draw $u \sim \mathcal{U}$ and i.i.d.\ $x_1,\ldots,x_N \sim \nu$. Conditional on $u$ and $\{x_i\}$, draw independent samples $\yt_1,\ldots,\yt_N \in \RR^n$ where $\yt_i \sim \mathcal{N}(x_i u, I_n - uu^\top)$.
\end{itemize}
Suppose that we observe the matrix $\tilde Y \in \RR^{N \times n}$ with rows $\yt_1^\top, \dots, \yt_N^\top$. Our goal is to test between the hypotheses $\cQ$ and $\cP$. 
\end{problem}

\noindent In other words, in the planted distribution $\cP$, the samples (conditional on $u$) have distribution $\nu$ in the direction $u$ (in the sense that $\langle \yt_i, u \rangle \sim \nu$) but are Gaussian in all other directions (in the sense that $\langle \yt_i, w \rangle \sim \mathcal{N}(0, 1)$ for any unit vector $w \perp u$). 
Next, we will show that this problem is equivalent to our original detection problem if $\mathcal{U}$ is uniform on the sphere and $\nu = \BG(1, \rho)$. 
Recall that by Definition~\ref{def:bg}, a Bernoulli--Gaussian variable $x \sim \BG(1, \rho)$ is defined by
\begin{equation}
\begin{cases}
x = 0 & \text{ with probability } 1 - \rho , \\
x \sim \mathcal{N}(0, 1/\rho) & \text{ with probability } \rho. 
\end{cases}
\label{eq:def-x}
\end{equation}

\begin{lemma}
\label{lem:equiv-detect}
When $\mathcal{U}$ is the uniform distribution over the unit sphere in $\RR^n$ and $\nu$ is $\BG(1, \rho)$ as defined in \eqref{eq:def-x}, Problem~\ref{prob:alt} is equivalent to Problem~\ref{prob:detect} up to a rescaling of $\tilde Y$. 
In particular, the quantity $\Adv_{\le D}$ defined in \eqref{eq:adv} is the same for the two problems. 
\end{lemma}

\begin{proof}
We claim that, under either $\cQ$ or $\cP$, the observation $\tilde Y$ in Problem~\ref{prob:alt} is equal in distribution to $\sqrt{N}$ times the observation in Problem~\ref{prob:detect}, and thus the two detection problems are equivalent. 
This claim is trivial under $\cQ$. Now consider $\tilde Y = YQ$ drawn from $\cP$ in Problem~\ref{prob:detect}. Let $u^\top$ be the first row of $Q$, which is a uniformly random unit vector. Conditional on $Q$ and $v \sim \BG(N, \rho)$, the rows $\yt_i^\top$ of $\tilde Y$ are independently distributed as $\yt_i \sim \mathcal{N}(v_i u, \frac{1}{N}(I_n - uu^\top))$. Also note that $v$ has i.i.d.\ entries distributed as $\frac{1}{\sqrt N} \nu = \frac{1}{\sqrt N} \BG(1, \rho)$. As a result, this matrix $\tilde Y$ is equal in distribution to $\frac{1}{\sqrt N}$ times the matrix $\tilde Y$ in Problem~\ref{prob:alt} as desired.
Finally, the quantity $\Adv_{\le D}$ defined in \eqref{eq:adv} is invariant under a rescaling, so $\Adv_{\le D}$ is the same for the two problems.   
\end{proof}

We will make use of the \emph{Hermite polynomials} (see e.g., \cite{orthog-poly}), which are orthogonal polynomials for the Gaussian distribution. 
Let $\NN := \{0,1,2,\ldots\}$. 
Let $\{h_k\}_{k \in \NN}$ denote the Hermite polynomials, normalized so that 
$$\EE_{z \sim \mathcal{N}(0,1)}[h_k(z)h_\ell(z)] = \One_{k=\ell} .$$
For example, the first few Hermite polynomials are
\[ h_0(z) = 1, \quad h_1(z) = z, \quad   h_2(z) = \frac{1}{\sqrt{2}} (z^2 - 1), \]
\[ h_3(z) = \frac{1}{\sqrt{3!}} ( z^3 - 3z) , \quad  h_4(z) = \frac{1}{\sqrt{4!}}  (z^4 - 6z^2+3 ) . \]
The following lemma consists of standard facts about Hermite polynomials.

\begin{lemma}\label{lem:hermite-facts}
The Hermite polynomials $\{h_k\}$ satisfy
\begin{itemize}
    \item $\displaystyle h_k(x+z) = \sum_{\ell=0}^k \sqrt{\frac{\ell!}{k!}} \binom{k}{\ell} x^{k-\ell} h_\ell(z)$, and
    \item for $x \ge -1$, $\displaystyle \EE_{y \sim \mathcal{N}(0,1+x)} \left[ h_k(y) \right] = \begin{cases} (k-1)!! (k!)^{-1/2} x^{k/2} & k \text{ even,} \\ 0 & k \text{ odd.} \end{cases}$
\end{itemize}
\end{lemma}

\noindent We will also make use of the multivariate Hermite polynomials $H_\alpha: \RR^{N \times n} \to \RR$, which are indexed by $\alpha \in \NN^{N \times n}$ and defined by $H_\alpha(y) := \prod_{i=1}^N \prod_{j=1}^n h_{\alpha_{ij}}(y_{ij})$. Letting $|\alpha| := \sum_{i=1}^N \sum_{j=1}^n \alpha_{ij}$, we have that $\{H_\alpha\}_{|\alpha| \le D}$ is a basis for the subspace of polynomials $R^{N \times n} \to \RR$ of degree at most $D$. Furthermore, if $y \in \RR^{N \times n}$ has i.i.d.\ $\mathcal{N}(0,1)$ entries, these polynomials are orthonormal in the sense that
$$\EE[H_\alpha(y)H_\beta(y)] = \One_{\alpha=\beta}.$$

We now establish a general formula for $\Adv_{\le D}$, which is the key to our low-degree lower bounds.

\begin{lemma}\label{lem:formula}
Consider the distribution $\cP$ in Problem~\ref{prob:alt} and suppose the first $D$ moments of $\nu$ are finite. For $\alpha \in \NN^N$, let $|\alpha| := \sum_{i=1}^N \alpha_i$. Then
\begin{equation}\label{eq:ld-formula}
\Adv_{\le D}^2 = \sum_{d=0}^D \EE[\langle u,u' \rangle^d] \sum_{\substack{\alpha \in \NN^N \\ |\alpha| = d}} \prod_{i=1}^N \left(\EE_{x \sim \nu}[h_{\alpha_i}(x)]\right)^2
\end{equation}
where $u$ and $u'$ are drawn independently from $\mathcal{U}$.
\end{lemma}

\noindent The above result generalizes some formulas that have appeared before in the literature. In the case where $\mathcal{U}$ is arbitrary and $\nu$ is $\mathcal{N}(0,1+\beta)$ for some $\beta \ge -1$, Problem~\ref{prob:alt} becomes the \emph{spiked Wishart model} and~\eqref{eq:ld-formula} reduces to Lemma~5.9 of~\cite{sk-cert}. Also, in the case where $u \sim \mathcal{U}$ is i.i.d.\ $\mathrm{Unif}(\{\pm 1/\sqrt{n}\})$ and $\nu$ is arbitrary, the calculations in Appendix~B of~\cite{lifting-sos} yield a formula that is similar to~\eqref{eq:ld-formula}.

The proof of Lemma~\ref{lem:formula} below follows the standard strategy of expanding the likelihood ratio $L = \frac{d \cP}{d \cQ}$ in the basis of Hermite polynomials, but with an additional trick that helps to simplify the calculations: We first decompose $\|L^{\le D}\|^2$ as the sum of many terms (indexed by $u,u'$) and use a \emph{different} Hermite basis (aligned with $u,u'$) for each term.

\begin{proof}[Proof of Lemma~\ref{lem:formula}]
We first use a limiting argument to restrict to the case where $\nu$ is absolutely continuous (with respect to Lebesgue measure). If this is not the case, fix $N,n,D,\mathcal{U},\nu$ and consider a sequence $\{\nu_m\}_{m \in \NN}$ of absolutely continuous measures whose first $D$ moments converge to those of $\nu$, and define $\mathcal{P}_m$ to be the planted distribution from Problem~\ref{prob:alt} with $\nu_m$ in place of $\nu$. This means that for every $k \le D$, $\EE_{x \sim \nu_m} [h_k(x)] \to \EE_{x \sim \nu} [h_k(x)]$ as $m \to \infty$. Also, for every $\alpha \in \NN^{N \times n}$ with $|\alpha| \le D$, $\EE_{y \sim \mathcal{P}_m} [H_\alpha(y)] \to \EE_{y \sim \mathcal{P}} [H_\alpha(y)]$ as $m \to \infty$. Finally, $\Adv_{\le D}$ is a continuous function of the values $\EE_\mathcal{P}[H_\alpha]$; see Lemma~\ref{lem:ld-hermite}. As a result, it is sufficient to prove~\eqref{eq:ld-formula} in the case that $\nu$ is absolutely continuous, so we will assume that the likelihood ratio $L = \frac{d \cP}{d \cQ}$ exists in the sequel.

For functions $f,g: \RR^{N \times n} \to \RR$, define the inner product $\langle f,g \rangle := \EE_{y \sim \cQ}[f(y)g(y)]$ and the associated norm $\|f\| = \sqrt{\langle f,f \rangle}$. Now
\[ \Adv_{\le D} = \max_{f \in \RR[y]_{\le D}} \frac{\langle f,L \rangle}{\|f\|} = \frac{\langle L^{\le D},L \rangle}{\|L^{\le D}\|} = \|L^{\le D}\| \]
where $f^{\le D}$ denotes the orthogonal projection (with respect to $\langle \cdot,\cdot \rangle$) of a function $f$ onto $\RR[y]_{\le D}$ (the subspace of polynomials $\RR^{N \times n} \to \RR$ of degree at most $D$), and $L(y) := \frac{d \cP}{d \cQ}(y)$ is the likelihood ratio. 
Let $\cP_u$ denote the distribution of $y \sim \cP$ conditioned on a particular choice of $u$, and let $L_u(y) := \frac{d \cP_u}{d \cQ}(y)$ so that $L(y) = \EE_{u \sim \mathcal{U}} L_u(y)$. We have 
\begin{equation}\label{eq:L-E}
\|L^{\le D}\|^2 = \langle L^{\le D}, L^{\le D} \rangle = \langle \EE_u L_u^{\le D}, \EE_{u'} L_{u'}^{\le D} \rangle = \EE_{u,u'} \langle L_u^{\le D}, L_{u'}^{\le D} \rangle
\end{equation}
where $u,u' \sim \mathcal{U}$ independently.

We now fix $u,u'$ and focus on computing $\langle L_u^{\le D}, L_{u'}^{\le D} \rangle$. It will be helpful to work in an orthonormal basis $\{e_1,\ldots,e_n\}$ for which $u = e_1$ and $u' = \tau e_1 + \sqrt{1-\tau^2} e_2$ where $\tau := \langle u,u' \rangle$. Recall from above that an orthonormal basis for $\RR[y]_{\le D}$ is given by the multivariate Hermite polynomials $\{H_\alpha\}_{|\alpha| \le D}$, where $\alpha \in \NN^{N \times n}$. Thus,
\begin{equation}\label{eq:L}
\langle L_u^{\le D}, L_{u'}^{\le D} \rangle = \sum_{|\alpha| \le D} \langle L_u,H_\alpha \rangle \langle L_{u'},H_\alpha \rangle.
\end{equation}
Also,
\[ \langle L_u,H_\alpha \rangle = \EE_{y \sim \cP_u}[H_\alpha(y)] = \One_{\{\alpha \text{ supported on entries } (i,1)\}} \prod_{i=1}^N \EE_{x \sim \nu}[h_{\alpha_{i1}}(x)] \]
using the assumption $u = e_1$ along with the facts $h_0(z) = 1$ and $\EE_{z \sim \mathcal{N}(0,1)}[h_k(z)] = 0$ for $k > 0$. In particular, the only $\alpha$'s that contribute to the sum in~\eqref{eq:L} are those that are supported only on entries of the form $(i,1)$. For an $\alpha$ of this type,
\[ \langle L_{u'},H_\alpha \rangle = \EE_{y \sim \cP_{u'}}[H_\alpha(y)] = \prod_{i=1}^N \EE_{\substack{x \sim \nu \\ z \sim \mathcal{N}(0,1)}}[h_{\alpha_{i1}}(\tau x + \sqrt{1-\tau^2} z)]. \]
This can be simplified using the facts in Lemma~\ref{lem:hermite-facts}:
\begin{align}
\EE_{\substack{x \sim \nu \\ z \sim \mathcal{N}(0,1)}}h_k(\tau x + \sqrt{1-\tau^2} z)
&= \EE_x \sum_{\ell=0}^k \sqrt{\frac{\ell!}{k!}} \binom{k}{\ell} (\tau x)^{k-\ell} \EE_z[h_\ell(\sqrt{1-\tau^2}z)] \label{eq:tau-1} \\
&= \EE_x \sum_{\ell=0}^k \sqrt{\frac{\ell!}{k!}} \binom{k}{\ell} (\tau x)^{k-\ell} \One_{\ell \text{ even}} (\ell-1)!! (\ell!)^{-1/2} (-\tau^2)^{\ell/2} \nonumber \\
&= \tau^k \EE_x \sum_{\ell=0}^k \sqrt{\frac{\ell!}{k!}} \binom{k}{\ell} x^{k-\ell} \One_{\ell \text{ even}} (\ell-1)!! (\ell!)^{-1/2} (-1)^{\ell/2} \label{eq:tau-2} \\
&= \tau^k \EE_x[h_k(x)] \nonumber
\end{align}
where the last step follows by comparing~\eqref{eq:tau-2} with the left-hand side of~\eqref{eq:tau-1} in the case $\tau = 1$.

Putting it all together,
\[ \langle L_u^{\le D}, L_{u'}^{\le D} \rangle = \sum_{\substack{\alpha \in \NN^N \\ |\alpha| \le D}} \prod_{i=1}^N \EE_{x \sim \nu}[h_{\alpha_i}(x)] \, \tau^{\alpha_i} \EE_{x \sim \nu}[h_{\alpha_i}(x)] = \sum_{d=0}^D \tau^d \sum_{\substack{\alpha \in \NN^N \\ |\alpha| = d}} \prod_{i=1}^N \left(\EE_{x \sim \nu}[h_{\alpha_i}(x)]\right)^2. \]
The result now follows from~\eqref{eq:L-E}, recalling $\tau := \langle u,u' \rangle$.
\end{proof}

\subsection{Proof of Theorem~\ref{thm:lower}}
\label{sec:pf-adv}

By Lemma~\ref{lem:equiv-detect}, it suffices to study the quantity $\Adv_{\le D}$ for Problem~\ref{prob:alt}, where $\mathcal{U}$ is uniform on the sphere and $\nu$ is $\BG(1, \rho)$. 
According to Lemma~\ref{lem:formula}, to bound $\Adv_{\le D}^2$, we need to compute the term $\EE[\langle u, u' \rangle^d]$. 

\begin{lemma}
\label{lem:inn-prod}
Let $u$ and $u'$ be independent uniform random vectors on the unit sphere in $\RR^n$. For $d \in \NN$, if $d$ is odd, then $\EE[\langle u, u' \rangle^d] = 0$, and if $d$ is even, then 
$$ 
\EE[\langle u, u' \rangle^d] =
\frac{ \Gamma(\frac{n}{2}) \, \Gamma(\frac{d+1}{2}) }{ \sqrt{\pi} \,  \Gamma(\frac{n+d}{2}) } 
\le (d/n)^{d/2} . 
$$
\end{lemma}

\begin{proof}
Since $u$ and $u'$ are uniformly random and independent, it holds that $$\EE[\langle u, u' \rangle^d] = \EE[\langle u, e_1 \rangle^d] = \EE[u_1^d],$$ 
where $e_1$ is the first standard basis vector in $\RR^n$ and $u_1$ is the first entry of $u$. 
By symmetry, the expectation is zero if $d$ is odd. We assume that $d$ is even in the sequel.

Before proceeding, let us recall an identity in calculus \cite[3.251.1]{gradshteyn2014table}:
\begin{equation}
\int_{-1}^1 x^{a-1} (1-x^2)^{b-1} \, dx 
= B \left( \frac{a}{2} , b \right) 
= \frac{ \Gamma(\frac{a}{2}) \, \Gamma(b) }{ \Gamma(\frac{a}{2} + b) }
\label{eq:int-beta}
\end{equation}
where $a$ is a positive odd integer, $b > 0$, $B$ denotes the beta function, and $\Gamma$ denotes the gamma function.

Let $f_1$ denote the PDF of $u_1$. Then $f_1(x) \, dx$ is proportional to surface area of the belt on the unit sphere with the first coordinate in $[x, x+dx]$. 
As a result, $f_1(x)$ is proportional to $\frac{(1-x^2)^{\frac{n-2}{2}}}{\sqrt{1-x^2}} = (1-x^2)^{\frac{n-3}{2}}$. 
Applying \eqref{eq:int-beta} with $a = 1$ and $b = \frac{n-1}{2}$, together with the fact that $\int_{-1}^1 f_1(x) \, dx = 1$, we get 
$$
f_1(x) = \frac{\Gamma( \frac n2 )}{\sqrt{\pi} \, \Gamma( \frac{n-1}{2} )} (1-x^2)^{\frac{n-3}{2}}.
$$ 
Finally, we apply \eqref{eq:int-beta} with $a = d+1$ and $b = \frac{n-1}{2}$ to obtain 
$$
\EE[u_1^d] = \int_{-1}^1 x^d f_1(x) \, dx = \frac{\Gamma( \frac n2 )}{\sqrt{\pi} \, \Gamma( \frac{n-1}{2} )}  \frac{ \Gamma(\frac{d+1}{2}) \, \Gamma(\frac{n-1}{2}) }{ \Gamma(\frac{n+d}{2}) }
= \frac{ \Gamma(\frac{n}{2}) \, \Gamma(\frac{d+1}{2}) }{ \sqrt{\pi} \,  \Gamma(\frac{n+d}{2}) }.
$$

To obtain the final bound, it suffices to note that 
$$
\frac{ \Gamma(\frac{n}{2}) }{ \Gamma(\frac{n+d}{2}) } = \frac{1}{ \frac{n}{2} ( \frac{n}{2} + 1 ) \cdots (\frac{n+d}{2}-1) } \le \frac{1}{(n/2)^{d/2}}
$$
and 
$$
\Gamma \left( \frac{d+1}{2} \right) \le \Gamma \left( \frac{d}{2} + 1 \right)
= (d/2)! \le (d/2)^{d/2} . 
$$
\end{proof}

We collect basic estimates regarding the Hermite polynomials in the following lemma. 

\begin{lemma}
For the Bernoulli--Gaussian variable $x \sim \BG(1, \rho)$ defined in \eqref{eq:def-x}, we have 
\begin{equation}
\EE[h_{k}(x)] = 0 \quad \text{if $k$ is odd}, \qquad
\EE[h_0(x)] = 1, \qquad
\EE[h_2(x)] = 0.
\label{eq:herm}
\end{equation}
Moreover, there exists a universal constant $C > 0$ such that for any $k \ge 4$ and $\rho \in (0,1]$, 
\begin{equation}
\left( \EE[h_k(x)] \right)^2 \le k^{2k} \rho^{2-k} . 
\label{eq:sq-bd-2}
\end{equation}
\end{lemma}

\begin{proof}
The claim $\EE[h_{k}(x)] = 0$ for odd $k$ follows from the symmetry of $x \sim \BG(1, \rho)$ and the fact that odd-degree Hermite polynomials are odd functions. 
Also, we have 
$$
\EE[h_0(x)] = \EE[1] = 1, \qquad
\EE[h_2(x)] = \EE \left[ \frac{1}{\sqrt{2}}\left(x^2 - 1 \right) \right] = \frac{1}{\sqrt{2}} \left( \rho \cdot \frac{1}{\rho} - 1 \right) = 0.
$$

Next, let $c_r$ be the coefficient of $z^r$ in the polynomial $\sqrt{k!} \cdot h_k(z)$. 
Since $\EE[x^0] = 1$ and $\EE[x^r] = \rho^{1-r/2} (r-1)!!$ for any even integer $r \ge 2$, we have $\left|\EE[x^r]\right| \le \rho^{1-k/2} (k-1)!!$ for any $k \ge 2$ and $0 \le r \le k$. 
Therefore, for $k \ge 2$, 
$$
\left| \EE[h_k(x)] \right| 
= \frac{1}{\sqrt{k!}} \left| \sum_{r=0}^{k} c_r \EE[x^r] \right|
\le \frac{(k-1)!!}{\sqrt{k!}} \rho^{1-k/2} \sum_{r=0}^{k} |c_r| 
\le \rho^{1-k/2} \sum_{r=0}^{k} |c_r| .
$$
It is known (see e.g.,~\cite{hermite-involution}) that the sum $\sum_{r=0}^k |c_r|$ is equal to the telephone number $T(k)$ \cite[A000085]{oeis}, 
which is the number of self-inverse permutations on $[k]$. 
As a result, we have 
$$
\sum_{r=0}^k |c_r| = T(k) \le k! \le k^k .
$$
Combining the above two bounds finishes the proof.
\end{proof}

The following lemma introduces a crucial subset of $\NN^N$. 

\begin{lemma}
\label{lem:adm-card}
For $\alpha \in \NN^N$, let $|\alpha| = \sum_{i=1}^N \alpha_i = \|\alpha\|_1$, and let $\|\alpha\|_0$ be the size of the support of $\alpha$. 
For $m \in [d]$, define a set 
\begin{equation}
\mathcal{A}(d, m) : = \left\{ \alpha \in \NN^N : |\alpha| = d , \, \|\alpha\|_0 = m, \, \alpha_i \in \{0\} \cup \{ 4, 6, 8, 10, \dots \} \text{ for all } i \in [N] \right\} . 
\label{eq:def-adm}
\end{equation}
Then we have 
$\left| \mathcal{A}(d, m) \right| \le N^m d^{d/2} . $
\end{lemma}

\begin{proof}
Since there are $\binom{N}{m}$ ways to choose the support of $\alpha \in \mathcal{A}(d, m)$, and with the support fixed, there are at most $m^{d/2}$ ways to distribute the mass of $\alpha$, it is easily seen that 
\begin{equation*}
\left| \mathcal{A}(d, m) \right| \le \binom{N}{m} \cdot m^{d/2}
\le N^m d^{d/2} . 
\end{equation*} 
\end{proof}

We are ready to prove Theorem~\ref{thm:lower}.

\begin{proof}[Proof of Theorem~\ref{thm:lower}]
Define $\mathcal{A}(d, m)$ as in \eqref{eq:def-adm}. 
For $\alpha \in \NN^N$ with $|\alpha| = d$ and $\|\alpha\|_0 = m$, if $\alpha \notin \mathcal{A}(d, m)$, then we must have $\EE[h_{\alpha_i}(x)] = 0$ for some $i \in [N]$ by \eqref{eq:herm}. 
On the other hand, for $\alpha \in \mathcal{A}(d, m)$, we obtain from \eqref{eq:sq-bd-2} that 
\begin{equation}
\prod_{i=1}^N \left(\EE[h_{\alpha_i}(x)]\right)^2 
\le \prod_{i \in [N], \, \alpha_i \ne 0} \alpha_i^{2 \alpha_i} \rho^{2-\alpha_i}
\le d^{2d} \rho^{2 m - d} .
\label{eq:prod}
\end{equation}
It follows that for $d \ge 4$,
\begin{align*}
\sum_{\substack{\alpha \in \NN^N \\ |\alpha| = d}} \prod_{i=1}^N \left(\EE[h_{\alpha_i}(x)]\right)^2 
&= \sum_{m=1}^{\lfloor d/4 \rfloor} \sum_{\substack{\alpha \in \mathcal{A}(d, m)}} \prod_{i=1}^N \left(\EE[h_{\alpha_i}(x)]\right)^2 \\
&\le \sum_{m=1}^{\lfloor d/4 \rfloor} \left| \mathcal{A}(d, m) \right| \cdot d^{2d} \rho^{2 m - d} 
\le \sum_{m=1}^{\lfloor d/4 \rfloor} N^m d^{d/2} d^{2d} \rho^{2 m - d} 
\end{align*}
by \eqref{eq:prod} and Lemma~\ref{lem:adm-card}. 
Since we assume $N \rho \ge n \rho \gg \sqrt{N}$, it holds that $N \rho^2 \ge 2$. Then, computing the above sum yields 
\begin{align*}
\sum_{\substack{\alpha \in \NN^N \\ |\alpha| = d}} \prod_{i=1}^N \left(\EE[h_{\alpha_i}(x)]\right)^2 
&\le d^{5d/2} N \rho^{2-d} \frac{(N \rho^2)^{\lfloor d/4 \rfloor} - 1}{N \rho^2 - 1} \\
&\le d^{5d/2} N \rho^{2-d} \frac{(N \rho^2)^{d/4}}{\frac{1}{2} N \rho^2}
= 2 d^{5d/2} N^{d/4} \rho^{-d/2} .
\end{align*}
This together with Lemma~\ref{lem:inn-prod} gives 
\begin{equation*} 
\EE[\langle u,u' \rangle^d] \sum_{\substack{\alpha \in \NN^N \\ |\alpha| = d}} \prod_{i=1}^N \left(\EE[h_{\alpha_i}(x)]\right)^2 
\le (d/n)^{d/2} \cdot 2 d^{5d/2} N^{d/4} \rho^{-d/2}
= 2 \left( \frac{ d^{12} N }{ n^2 \rho^2 } \right)^{d/4} . 
\end{equation*} 
Finally, combining this with Lemma~\ref{lem:formula}, we obtain  
\begin{equation*} 
\Adv_{\le D}^2 = \sum_{d=0}^D \EE[\langle u,u' \rangle^d] \sum_{\substack{\alpha \in \NN^N \\ |\alpha| = d}} \prod_{i=1}^N \left(\EE[h_{\alpha_i}(x)]\right)^2 
\le 1 + 2 \sum_{d=4}^D \left( \frac{ d^{12} N }{ n^2 \rho^2 } \right)^{d/4} . 
\end{equation*}
It suffices to have the above sum bounded by $2$ when $D \le (\log N)^{C_1}$ for a constant $C_1 > 0$. 
This is true provided $n \rho \ge  \sqrt{N} (\log N)^{C_2}$ for a sufficiently large constant $C_2 = C_2(C_1) > 0$ so that $\frac{d^{12} N }{ n^2 \rho^2 } \le \frac{1}{4}$. 
\end{proof}

\appendix

\section{Additional Proofs}
\label{sec:add-proof}

\subsection{Proof of Proposition~\ref{prop:gauss-spec-norm}}
\label{sec:pf-gauss-spec-norm}

By the definition of $M$ in \eqref{eq:def-mg}, we have  
\begin{equation*}
\Mg = \sum_{i=1}^N \left( v_i^2 + \|b_i\|^2 - \frac{n-1}{N} \right) \yg_i \yg_i^\top - \frac{3}{N} I_n
= \sum_{i=1}^N \left( v_i^2 + \|b_i\|^2 - \frac{n-1}{N} \right) 
\begin{bmatrix}
v_i^2 & v_i b_i^\top \\
v_i b_i & b_i b_i^\top 
\end{bmatrix} 
- \frac{3}{N} 
\begin{bmatrix}
1 & 0 \\
0 & I_{n-1} 
\end{bmatrix}
,
\end{equation*}
so 
\begin{align*}
\Mg -  \left( \|v\|_4^4 - \frac{3}{N} \right) e_1 e_1^\top 
&= \sum_{i=1}^N \left( \|b_i\|^2 - \frac{n-1}{N} \right) 
\begin{bmatrix}
v_i^2 & 0 \\
0 & 0 
\end{bmatrix} 
\\
&\quad + \sum_{i=1}^N \left( v_i^2 + \|b_i\|^2 - \frac{n-1}{N} \right) 
\begin{bmatrix}
0 & v_i b_i^\top \\
v_i b_i & 0 
\end{bmatrix} 
\\
&\quad + \sum_{i=1}^N \left( v_i^2 + \|b_i\|^2 - \frac{n-1}{N} \right) 
\begin{bmatrix}
0 & 0 \\
0 & b_i b_i^\top 
\end{bmatrix} 
- \frac{3}{N} 
\begin{bmatrix}
0 & 0 \\
0 & I_{n-1} 
\end{bmatrix} . 
\end{align*}
As a result,
\begin{align*}
\left\| \Mg -  \left( \|v\|_4^4 - \frac{3}{N} \right) e_1 e_1^\top \right\| 
&\le 
\left| \sum_{i=1}^N \left( \|b_i\|^2 - \frac{n-1}{N} \right) v_i^2 \right|  \\
&\quad + 2 \left\| \sum_{i=1}^N v_i^3 b_i \right\| 
+ 2 \left\| \sum_{i=1}^N \left( \|b_i\|^2 - \frac{n-1}{N} \right) v_i b_i \right\| \\
& \quad + \left\| \sum_{i=1}^N v_i^2 b_i b_i^\top - \frac{1}{N} I_{n-1} \right\|
+ \left\| \sum_{i=1}^N \left( \|b_i\|^2 - \frac{n-1}{N} \right) b_i b_i^\top - \frac{2}{N} I_{n-1} \right\| . 
\end{align*}

Lemmas~\ref{lem:l1}--\ref{lem:l5} below provide bounds on the five terms on the right-hand side of the above inequality respectively. 
Hence, there is a universal constant $C>0$ such that for any $\delta \in (0,1)$, 
\begin{align*}
\left\| \Mg -  \left( \|v\|_4^4 - \frac{3}{N} \right) e_1 e_1^\top \right\| 
&\le C \bigg(  \|v\|_4^2 \frac{\sqrt{n \log(N/\delta)}}{N} + \frac{\|v\|_6^3}{\sqrt{N}} \left( \sqrt{n} + \sqrt{\log(N/\delta)} \right)  \\
&\qquad + \left( \|v\| + 1 \right) \frac{n \sqrt{\log(N/\delta)}}{N^{3/2}} + \|v\|_\infty \frac{n \log^{5/2}(N/\delta)}{N^{3/2}} \\
&\qquad + \|v\|_\infty^2 \frac{n \log(N/\delta) + \log^2(N/\delta)}{N}  + \frac{1}{N} \left| \|v\|^2 - 1 \right| 
+ \frac{ (n \log(N/\delta))^{3/2} }{N^2} \bigg)   
\end{align*}
with probability at least $1 - \delta$, where we have omitted repetitive bounds of the same or lower order for brevity using the assumption $N \ge \log^7(N/\delta)$. 
Moreover, 
using the fact $2xy \le x^2+y^2$ for $x,y \in \RR$ and the assumption $N \ge \log^7(N/\delta)$, 
we obtain
$$
\|v\|_\infty \frac{n \log^{5/2}(N/\delta)}{N^{3/2}} 
\le \|v\|_\infty^2 \frac{n \log(N/\delta)}{N} + \frac{n \log^4(N/\delta)}{N^2} 
\le \|v\|_\infty^2 \frac{n \log(N/\delta)}{N} + \frac{n \sqrt{\log(N/\delta)}}{N^{3/2}}.
$$
Further simplification using the above bound yields 
\begin{align*}
\left\| \Mg -  \left( \|v\|_4^4 - \frac{3}{N} \right) e_1 e_1^\top \right\| 
&\le C \bigg( \left( \|v\| + 1 \right) \frac{n \sqrt{\log(N/\delta)}}{N^{3/2}} + \|v\|_4^2 \frac{\sqrt{n \log(N/\delta)}}{N} + \|v\|_6^3 \frac{ \sqrt{n} + \sqrt{\log(N/\delta)} }{\sqrt{N}} \\
&\qquad \quad 
+  \|v\|_\infty^2 \frac{n \log(N/\delta) + \log^2(N/\delta)}{N} 
+ \frac{ (n \log(N/\delta))^{3/2} }{N^2} + \frac{1}{N} \left| \|v\|^2 - 1 \right| \bigg)  ,
\end{align*}
finishing the proof of Proposition~\ref{prop:gauss-spec-norm}.

\begin{lemma}
\label{lem:l1}
There is a universal constant $C>0$ such that for any $\delta \in (0,1)$, 
$$
\left| \sum_{i=1}^N \left( \|b_i\|^2 - \frac{n-1}{N} \right) v_i^2 \right|
\le C \left(  \frac{\sqrt{n \log(1/\delta)}}{N} \|v\|_4^2 + \frac{\sqrt{n} \log(1/\delta)}{N} \|v\|_\infty^2 \right)
$$
with probability at least $1 - \delta$. 
\end{lemma}

\begin{proof}
Since $N \|b_i\|^2$ is a chi-square random variable with $n-1$ degrees of freedom, we have
$$
\EE \exp \Big( \lambda\left( N \|b_i\|^2 - (n-1) \right) \Big) 
= \left( \frac{ e^{-\lambda} }{ \sqrt{1-2\lambda} } \right)^{n-1} 
\le e^{2(n-1) \lambda^2} 
$$
for $\lambda \in [-1/3,1/3]$. 
Hence, by definition \eqref{eq:def-sub-exp}, $N \|b_i\|^2 - (n-1)$ is a sub-exponential random variable with parameter $C_1 \sqrt{n}$ for a constant $C_1>0$. By Bernstein's inequality (Lemma~\ref{lem:bern}), we obtain
$$
\left| \sum_{i=1}^N v_i^2 \left( N \|b_i\|^2 - (n-1) \right) \right| 
\le C_2 \left( \|v\|_4^2 \sqrt{n \log(1/\delta)} + \|v\|_\infty^2 \sqrt{n} \log(1/\delta) \right)
$$
with probability at least $1 - \delta$ for  a constant $C_2 > 0$. 
Dividing both sides by $N$ finishes the proof.
\end{proof}

\begin{lemma}
\label{lem:l2}
There is a universal constant $C>0$ such that for any $\delta \in (0,1)$, 
\begin{equation*}
\left\| \sum_{i=1}^N v_i^3 b_i \right\|
\le C \frac{\|v\|_6^3}{\sqrt{N}} \left( \sqrt{n} + \sqrt{\log(1/\delta)} \right) 
\end{equation*}
with probability at least $1 - \delta$. 
\end{lemma}

\begin{proof}
Since $\sum_{i=1}^N v_i^3 b_i \sim \mathcal{N} \left( 0, \frac{\|v\|_6^6}{N} I_{n-1} \right)$, by the tail bound for the chi-square distribution (Lemma~\ref{lem:chi-sq}), 
\begin{equation*}
\left\| \sum_{i=1}^N v_i^3 b_i \right\|^2
\le \frac{\|v\|_6^6}{N} \left[ n + C_1 \left( \sqrt{n \log(1/\delta)} + \log(1/\delta) \right) \right] 
\end{equation*}
with probability at least $1 - \delta$ for a constant $C_1>0$. 
Taking the square root on both sides completes the proof. 
\end{proof}

\begin{lemma}
\label{lem:l3}
There is a universal constant $C>0$ such that for any $\delta \in (0,1)$, 
$$
\left\| \sum_{i=1}^N \left( \|b_i\|^2 - \frac{n-1}{N} \right) v_i b_i \right\| 
\le C \left( \|v\| \frac{n \sqrt{\log(N/\delta)}}{N^{3/2}} + \|v\|_\infty \frac{n \log^{5/2}(N/\delta)}{N^{3/2}} \right)
$$
with probability at least $1 - \delta$. 
\end{lemma}

\begin{proof}
We will apply the matrix Bernstein inequality with truncation (Lemma~\ref{lem:tmb}) with the $n \times 1$ matrices $S_i = \left( \|b_i\|^2 - \frac{n-1}{N} \right) v_i b_i$.  
For each $i \in [N]$, let us define an event 
\begin{equation}
\mathcal{E}_i = \left\{ \left| \|b_i\|^2 - \frac{n-1}{N} \right| \le C_1 \frac{\sqrt{n \log(N/\delta)} + \log(N/\delta)}{N} \right\} 
\label{eq:ei}
\end{equation}
so that $\PP\{\mathcal{E}_i\} \ge 1 - \delta/N^2$ for  a constant $C_1 > 0$ by Lemma~\ref{lem:chi-sq}. 
Moreover, we define 
$$\tilde S_i = \left( \|b_i\|^2 - \frac{n-1}{N} \right) v_i b_i \cdot \mathbf{1} \{\mathcal{E}_i\} .$$ 
The entries of $S_i$ and $\tilde S_i$ have distributions symmetric around zero, so $\EE[S_i] = \EE[ \tilde S_i] = 0$. In addition, 
$$
\left\| \tilde S_i \right\| 
\le C_2 |v_i| \frac{\sqrt{n \log(N/\delta)} + \log(N/\delta)}{N} \sqrt{ \frac{n + \log(N/\delta)}{N} }
\le C_3 \|v\|_\infty \frac{n \sqrt{\log(N/\delta)} + \log^{3/2}(N/\delta)}{N^{3/2}} =: L 
$$
for  constants $C_2, C_3 > 0$.

Moreover, to bound the variance 
$$
\sigma^2 := \left\| \sum_{i=1}^N \EE[ \tilde S_i \tilde S_i^\top ]  \right\| \lor \left\| \sum_{i=1}^N \EE[ \tilde S_i^\top \tilde S_i ]  \right\|
= \sum_{i=1}^N \EE \left[ \left( \|b_i\|^2 - \frac{n-1}{N} \right)^2 v_i^2 \|b_i\|^2 \cdot \mathbf{1} \{\mathcal{E}_i\} \right] , 
$$
we compute 
\begin{align*}
\EE \left[ \left( \|b_i\|^2 - \frac{n-1}{N} \right)^2  \|b_i\|^2 \cdot \mathbf{1} \{\mathcal{E}_i\} \right] 
&\le \EE \left[ \left( \|b_i\|^2 - \frac{n-1}{N} \right)^2  \|b_i\|^2 \right] \\
&=  \EE \left[ \|b_i\|^6 \right]
- 2 \frac{n-1}{N} \EE \left[ \|b_i\|^4 \right]
+ \frac{(n-1)^2}{N^2} \EE \left[ \|b_i\|^2 \right] \\
&= \frac{1}{N^3} \left[ (n-1)(n+1)(n+3) - 2(n-1)^2(n+1) + (n-1)^3 \right] \\
&= \frac{2(n-1)(n+3)}{N^3} 
\end{align*}
in view of the moments of the chi-square distribution (Lemma~\ref{lem:two-exp}). 
Hence we obtain
$$
\sigma^2 \le \sum_{i=1}^N v_i^2 \frac{2(n-1)(n+3)}{N^3} = \|v\|^2 \frac{2(n-1)(n+3)}{N^3}.
$$

Lemma~\ref{lem:tmb} implies that 
$$
\left\| \sum_{i=1}^N S_i \right\| 
\le C_4 \left( \sigma \sqrt{\log(N/\delta)} + L \log(N/\delta) \right)
$$
with probability at least $1 - \delta$ for a constant $C_4>0$. 
Combining this with the above bound on $\sigma^2$ and the definition of $L$, we complete the proof. 
\end{proof}

\begin{lemma}
\label{lem:l4}
There is a universal constant $C>0$ such that for any $\delta \in (0,1)$, 
$$
\left\| \sum_{i=1}^N v_i^2 b_i b_i^\top - \frac{1}{N} I_{n-1} \right\| 
= C \left( \frac{\sqrt{n \log(N/\delta)}}{N} \|v\|_4^2 + \frac{n \log(N/\delta) + \log^2(N/\delta)}{N} \|v\|_\infty^2 \right) + \frac{1}{N} \left| \|v\|^2 - 1 \right| 
$$
with probability at least $1 - \delta$. 
\end{lemma}

\begin{proof}
We will apply the matrix Bernstein inequality with truncation (Lemma~\ref{lem:tmb}) with symmetric matrices
$S_i = v_i^2 b_i b_i^\top$. 
Let us define the event $\mathcal{E}_i$ according to \eqref{eq:ei} for each $i \in [N]$. Recall that the constant $C_1$ in the definition is chosen so that $\PP\{\mathcal{E}_i\} \ge 1 - \delta/N^2$. Moreover, we define
$$
\tilde S_i = v_i^2 b_i b_i^\top \cdot \mathbf{1}(\mathcal{E}_i) . 
$$

Note that all the off-diagonal entries of $S_i$ have distributions symmetric around zero and all the diagonal entries are identically distributed. The same is true for $\tilde S_i$ as well. Hence we have 
$$
\EE [ S_i ] = \alpha_i I_{n-1} , 
\qquad
\EE [ \tilde S_i ] = \beta_i I_{n-1} 
$$
for scalars $\alpha_i, \beta_i \in \RR$. 
Let us denote the $j$th entry of $b_i$ by $b_{i,j}$ and compute 
\begin{equation*}
\alpha_i = \EE \left[ v_i^2 b_{i,1}^2 \right] 
= \frac{v_i^2}{N} . 
\end{equation*}
Moreover, we have 
\begin{equation*}
\alpha_i - \beta_i = \EE \left[ v_i^2 b_{i,1}^2 \right] - \EE \left[ v_i^2 b_{i,1}^2 \cdot \mathbf{1}(\mathcal{E}_i) \right] 
= v_i^2 \EE \left[ b_{i,1}^2 \cdot \mathbf{1}(\mathcal{E}_i^c) \right] 
\le v_i^2 \EE \left[ \|b_i\|^2 \cdot \mathbf{1}(\mathcal{E}_i^c) \right] .
\end{equation*} 
By Lemma~\ref{lem:chi-sq-exp}, the expectation $\EE \left[ \|b_i\|^2 \cdot \mathbf{1}(\mathcal{E}_i^c) \right]$ is very small in view of our definition of $\mathcal{E}_i$. More precisely, as long as $C_1$ in the definition \eqref{eq:ei} is chosen to be larger than some universal constant, we can make $\EE \left[ \|b_i\|^2 \cdot \mathbf{1}(\mathcal{E}_i^c) \right]$ smaller than $N^{-10}$, so that 
\begin{equation*}
\Delta := \left\| \EE[S_i] - \EE[\tilde S_i] \right\| = \alpha_i - \beta_i 
\le \|v\|_\infty^2 N^{-10} . 
\end{equation*}

On the event $\mathcal{E}_i$ defined by \eqref{eq:ei}, we have 
\begin{align*}
\left\| \tilde S_i - \EE[ \tilde S_i ] \right\| 
&\le \left\| v_i^2 b_i b_i^\top \cdot \mathbf{1}\{\mathcal{E}_i\} \right\| + \| \beta_i I_{n-1} \| \\
&\le C_2 v_i^2 \frac{n + \log(N/\delta)}{N} + |\alpha_i| + |\alpha_i - \beta_i| \\
&\le C_3 \|v\|_\infty^2 \frac{n + \log(N/\delta)}{N} =: L 
\end{align*}
for constants $C_2 , C_3 > 0$.

Next, we bound the variance in the matrix Bernstein inequality. We have 
\begin{equation*}
\EE \left[ \tilde S_i^2 \right] 
= \EE \left[ v_i^4 \|b_i\|^2 b_i b_i^\top \cdot \mathbf{1}(\mathcal{E}_i) \right] 
= \gamma_i I_{n-1}
\end{equation*}
for a scalar $\gamma_i \in \RR$ by the same argument as before. Moreover, 
\begin{equation*}
\gamma_i \le \EE \left[ v_i^4 \|b_i\|^2 b_{i,1}^2 \right] 
= v_i^4 \frac{n+1}{N^2} 
\end{equation*}
by Lemma~\ref{lem:two-exp}. 
Therefore, the variance is 
$$
\sigma^2 
= \left\| \sum_{i=1}^N \left( \EE[ \tilde S_i^2 ] - \EE[ \tilde S_i ]^2 \right) \right\|
= \left\| \sum_{i=1}^N \left( \gamma_i I_{n-1} - \beta_i^2 I_{n-1} \right) \right\|
\le \sum_{i=1}^N \gamma_i 
\le \sum_{i=1}^N v_i^4 \frac{n+1}{N^2} 
= \|v\|_4^4 \frac{n+1}{N^2} . 
$$

Lemma~\ref{lem:tmb} implies that 
\begin{align*}
\left\| \sum_{i=1}^N v_i^2 b_i b_i^\top - \frac{1}{N} I_{n-1} \right\| 
&\le \left\| \sum_{i=1}^N \left( v_i^2 b_i b_i^\top - \frac{v_i^2}{N} I_{n-1} \right) \right\| + \left\| \sum_{i=1}^N \frac{v_i^2}{N} I_{n-1} - \frac{1}{N} I_{n-1} \right\| \\
&= \left\| \sum_{i=1}^N \left( S_i - \EE[S_i] \right) \right\| + \frac{1}{N} \left| \|v\|^2 - 1 \right| \\
&\le C_4 \left( \sigma \sqrt{\log(N/\delta)} + L \log(N/\delta) + N \Delta \right) + \frac{1}{N} \left| \|v\|^2 - 1 \right|
\end{align*}
with probability at least $1 - \delta$ for a constant $C_4>0$. 
Combining this with the above estimates for $\sigma^2$, $L$, and $\Delta$, we complete the proof. 
\end{proof}

\begin{lemma}
\label{lem:l5}
There is a universal constant $C>0$ such that for any $\delta \in (0,1)$, 
\begin{equation*}
\left\| \sum_{i=1}^N \left( \|b_i\|^2 - \frac{n-1}{N} \right) b_i b_i^\top - \frac{2}{N} I_{n-1} \right\| 
\le C \left( \frac{n \sqrt{\log(N/\delta)} }{N^{3/2}} + \frac{ (n \log(N/\delta))^{3/2} + \log^3(N/\delta) }{N^2} \right) 
\end{equation*}
with probability at least $1 - \delta$. 
\end{lemma}

\begin{proof}
The proof is very similar to that of the previous lemma, but let us present it for completeness. We will again apply the matrix Bernstein inequality with truncation (Lemma~\ref{lem:tmb}) with symmetric matrices
$S_i = \left( \|b_i\|^2 - \frac{n-1}{N} \right) b_i b_i^\top$. 
Let us define the event $\mathcal{E}_i$ according to \eqref{eq:ei} for each $i \in [N]$. Recall that the constant $C_1$ in the definition is chosen so that $\PP\{\mathcal{E}_i\} \ge 1 - \delta/N^2$. Moreover, we define
$$
\tilde S_i = \left( \|b_i\|^2 - \frac{n-1}{N} \right) b_i b_i^\top \cdot \mathbf{1}(\mathcal{E}_i) . 
$$

Note that all the off-diagonal entries of $S_i$ have distributions symmetric around zero and all the diagonal entries are identically distributed. The same is true for $\tilde S_i$ as well. Hence we have 
$$
\EE [ S_i ] = \alpha I_{n-1} , \qquad
\EE [ \tilde S_i ] = \beta I_{n-1} 
$$
for scalars $\alpha, \beta \in \RR$. 
Let us denote the $j$th entry of $b_i$ by $b_{i,j}$ and compute 
\begin{equation*}
\alpha = \EE \left[ \left( \|b_i\|^2 - \frac{n-1}{N} \right) b_{i,1}^2 \right]
= \EE \left[ \|b_i\|^2  b_{i,1}^2 \right] 
- \frac{n-1}{N} \EE \left[ b_{i,1}^2 \right]
= \frac{n+1}{N^2}  - \frac{n-1}{N^2}  
= \frac{2}{N^2} 
\end{equation*}
using Lemma~\ref{lem:two-exp}. 
Moreover, we have 
\begin{equation*}
\alpha - \beta = \EE \left[ \left( \|b_i\|^2 - \frac{n-1}{N} \right) b_{i,1}^2 \right] - \EE \left[ \left( \|b_i\|^2 - \frac{n-1}{N} \right) b_{i,1}^2 \cdot \mathbf{1}(\mathcal{E}_i) \right] 
= \EE \left[ \left( \|b_i\|^2 - \frac{n-1}{N} \right) b_{i,1}^2 \cdot \mathbf{1}(\mathcal{E}_i^c) \right] ,
\end{equation*} 
so 
\begin{align*}
\left| \alpha - \beta \right| 
&\le \EE \left[ \|b_i\|^2 b_{i,1}^2 \cdot \mathbf{1}(\mathcal{E}_i^c) \right] + \frac{n-1}{N} \EE \left[ b_{i,1}^2 \cdot \mathbf{1}(\mathcal{E}_i^c) \right] \\
&= \frac{1}{n-1} \sum_{j=1}^{n-1} \left( \EE \left[ \|b_i\|^2 b_{i,j}^2 \cdot \mathbf{1}(\mathcal{E}_i^c) \right] + \frac{n-1}{N} \EE \left[ b_{i,j}^2 \cdot \mathbf{1}(\mathcal{E}_i^c) \right] \right) \\
&= \frac{1}{n-1} \left( \EE \left[ \|b_i\|^4\cdot \mathbf{1}(\mathcal{E}_i^c) \right] + \frac{n-1}{N} \EE \left[ \|b_i\|^2 \cdot \mathbf{1}(\mathcal{E}_i^c) \right] \right) . 
\end{align*}
By Lemma~\ref{lem:chi-sq-exp}, the quantities $\EE \left[ \|b_i\|^4\cdot \mathbf{1}(\mathcal{E}_i^c) \right]$ and $\EE \left[ \|b_i\|^2 \cdot \mathbf{1}(\mathcal{E}_i^c) \right]$ are very small in view of our definition of $\mathcal{E}_i$. More precisely, 
as long as $C_1$ in the definition \eqref{eq:ei} is chosen to be larger than some universal constant, we can make both $\EE \left[ \|b_i\|^4\cdot \mathbf{1}(\mathcal{E}_i^c) \right]$ and $\EE \left[ \|b_i\|^2 \cdot \mathbf{1}(\mathcal{E}_i^c) \right]$ smaller than $\frac{1}{2} \cdot N^{-10}$, so that 
\begin{equation*}
\Delta := \left\| \EE[S_i] - \EE[\tilde S_i] \right\| = \left| \alpha - \beta \right| 
\le N^{-10} . 
\end{equation*}

On the event $\mathcal{E}_i$ defined by \eqref{eq:ei}, we have 
\begin{align*}
\left\| \tilde S_i - \EE[ \tilde S_i ] \right\| 
&\le \left\| \left( \|b_i\|^2 - \frac{n-1}{N} \right) b_i b_i^\top \cdot \mathbf{1}\{\mathcal{E}_i\} \right\| + \| \beta I_{n-1} \| \\
&\le C_2 \frac{\sqrt{n \log(N/\delta)} + \log(N/\delta)}{N} \cdot \frac{n + \log(N/\delta)}{N} + |\alpha| + |\alpha - \beta| \\
&\le C_3 \frac{ n^{3/2} \sqrt{\log(N/\delta)} + \log^2(N/\delta) }{N^2} =: L 
&\end{align*}
for constants $C_2 , C_3 > 0$.

Next, we bound the variance in the matrix Bernstein inequality. We have 
\begin{equation*}
\EE \left[ \tilde S_i^2 \right] 
= \EE \left[ \left( \|b_i\|^2 - \frac{n-1}{N} \right)^2 \|b_i\|^2 b_i b_i^\top \cdot \mathbf{1}(\mathcal{E}_i) \right] 
= \gamma I_{n-1}
\end{equation*}
for a scalar $\gamma \in \RR$ by the same argument as before. Moreover, 
\begin{align*}
\gamma &\le \EE \left[ \left( \|b_i\|^2 - \frac{n-1}{N} \right)^2 \|b_i\|^2 b_{i,1}^2 \right] \\
&=  \EE \left[ \|b_i\|^6 b_{i,1}^2 \right]
- 2 \frac{n-1}{N} \EE \left[ \|b_i\|^4 b_{i,1}^2 \right]
+ \frac{(n-1)^2}{N^2} \EE \left[ \|b_i\|^2 b_{i,1}^2 \right] \\
&= \frac{1}{N^4} \left[ (n+1)(n+3)(n+5) - 2(n-1)(n+1)(n+3) + (n-1)^2 (n+1) \right] \\
&= \frac{2(n+1)(n+11)}{N^4} 
\end{align*}
by Lemma~\ref{lem:two-exp}. 
Therefore, the variance is 
$$
\sigma^2 
= \left\| \sum_{i=1}^N \left( \EE[ \tilde S_i^2 ] - \EE[ \tilde S_i ]^2 \right) \right\|
= \left\| N \left( \gamma I_{n-1} - \beta^2 I_{n-1} \right) \right\|
\le N \gamma 
\le \frac{2(n+1)(n+11)}{N^3} . 
$$

Lemma~\ref{lem:tmb} implies that 
$$
\left\| \sum_{i=1}^N \left( \|b_i\|^2 - \frac{n-1}{N} \right) b_i b_i^\top - \frac{2}{N} I_{n-1} \right\| 
= \left\| \sum_{i=1}^N \left( S_i - \EE[S_i] \right) \right\| 
\le C_4 \left( \sigma \sqrt{\log(N/\delta)} + L \log(N/\delta) + N \Delta \right)
$$
with probability at least $1 - \delta$ for a constant $C_4>0$. 
Combining this with the above estimates for $\sigma^2$, $L$, and $\Delta$, we complete the proof. 
\end{proof}

\begin{lemma}
\label{lem:two-exp}
Given a Gaussian vector $g \sim \mathcal{N}(0, I_n)$, we have that, for any positive integer $k$, 
$$
\EE \left[ \|g\|^{2k} \right] = n (n+2) (n+4) \cdots (n+2k-2) 
$$
and 
$$
\EE \left[ \|g\|^{2k-2} g_1^2 \right] = (n+2) (n+4) \cdots (n+2k-2) .
$$
\end{lemma}

\begin{proof}
The first claim is simply the formula for the $k$th moment of the chi-square distribution with $n$ degrees of freedom. 
For the second claim, note that 
$$
n \EE \left[ \|g\|^{2k-2} g_1^2 \right] = \sum_{i=1}^n \EE \left[ \|g\|^{2k-2} g_i^2 \right] = \EE \left[ \|g\|^{2k} \right] ,
$$
so the result follows from the first claim.
\end{proof}

\begin{lemma}
\label{lem:chi-sq-exp}
Let $X$ be a chi-square random variable with $n$ degrees of freedom. For $B>0$, define an event  
$$
\mathcal{E} = \left\{ |X - n| \le B \right\} . 
$$
Then there exist universal constants $C, c > 0$ such that 
$$
\EE \left[ X^2 \cdot \mathbf{1}\{\mathcal{E}^c\} \right] \le C \left( n^2 e^{-c B^2/n} + (B \lor n) e^{-c (B \lor n) } \right) 
$$
and 
$$
\EE \left[ X \cdot \mathbf{1}\{\mathcal{E}^c\} \right] \le C \left( n e^{-c B^2/n} + e^{-c (B \lor n) } \right)  .
$$
\end{lemma}

\begin{proof}
If $B \ge n$, then we have 
$$
\EE \left[ X^2 \cdot \mathbf{1} \left\{ X \le n - B \right\} \right] 
= 0 
$$
and 
\begin{equation*}
\EE \left[ X^2 \cdot \mathbf{1} \left\{ X \ge n + B \right\} \right] 
= \int_{B}^\infty 2 (n+t) \cdot \PP \left\{ X \ge n + t \right\} \, dt 
\le C_1 \int_B^\infty (n+t) \exp \left( - c_1 t \right) \, dt 
\le C_2 B e^{-c_1 B} 
\end{equation*}
for universal constants $C_1, C_2, c_1 > 0$. 

If $B < n$, then we have 
$$
\EE \left[ X^2 \cdot \mathbf{1} \left\{ X \le n - B \right\} \right] 
\le n^2 \PP \left\{ X \le n - B \right\} 
\le n^2 e^{-c_2 B^2/n} 
$$
for a universal constant $c_2 > 0$, and 
\begin{align*}
\EE \left[ X^2 \cdot \mathbf{1} \left\{ X \ge n + B \right\} \right] 
&= \int_{B}^n 2 (n+t) \cdot \PP \left\{ X \ge n + t \right\} \, dt
+ \int_{n}^\infty 2 (n+t) \cdot \PP \left\{ X \ge n + t \right\} \, dt \\
&\le C_3 n^2 e^{-c_3 B^2/n} + C_3 n e^{-c_3 n} 
\end{align*}
for universal constants $C_3, c_3 > 0$. 

Combining the two parts finishes the proof of the first bound. 
The second bound is proved similarly. 
\end{proof}

\subsection{Bounds in the Euclidean Norm}
We first establish an error bound in the Euclidean norm for our spectral estimator. Although this result is essentially subsumed by the entrywise error bound in Theorem~\ref{thm:entrywise}, we hope the much simpler proof provides clearer intuition. 

\begin{theorem}
\label{thm:gaussian-general}
Assume Model~\ref{mod:gauss}. 
Let $\tilde{\ub}$ be the leading eigenvector of the matrix $\Mt$ defined in \eqref{eq:def-mt}. 
There exists a universal constant $C > 0$ such that the following holds. Let $\delta \in (0,1)$ and suppose that $N \ge \log^7(N/\delta)$. Then, up to a sign flip of $\tilde{\ub}$, we have
\begin{equation*}
\|\Yt \tilde{\ub} - v\| \le C (\|v\| + 1) \cdot \frac{ \eta }{ \left| \|v\|_4^4 - \frac{3}{N} \right|} 
\end{equation*}
with probability at least $1-\delta$, where $\eta = \eta(N,n,v,\delta)$ is defined in \eqref{eq:def-eta}. 
\end{theorem}

\begin{proof}
Recall Model~\ref{mod:gauss} and the definition of $\Mt$ in \eqref{eq:def-mt}. 
Let $\yg_i^\top$ and $\yt_i^\top$ be the $i$th row of $\Yg$ and $\Yt$ respectively, for $i \in [N]$. 
Then we have $\yt_i = Q^\top \yg_i$. 
Furthermore, let 
$$
\Mg = Q \Mt Q^\top = \sum_{i=1}^N \left( \|\yt_i\|^2 - \frac{n-1}{N} \right) Q \yt_i \yt_i^\top Q^\top - \frac{3}{N} I_n
= \sum_{i=1}^N \left( \|\yg_i\|^2 - \frac{n-1}{N} \right) \yg_i \yg_i^\top - \frac{3}{N} I_n 
$$
which is precisely the matrix defined in \eqref{eq:def-mg}, 
and let $\ub$ be the leading eigenvector of $\Mg$ so that $\tilde{\ub} = Q^\top \ub$. 
Then we have 
$$
\|\Yt \tilde{\ub} - v\| 
= \|\Yg \ub - \Yg e_1\| 
\le \|\Yg\| \cdot \|\ub - e_1\| 
\le (\|v\| + \|\Yg_{-1}\|) \|\ub - e_1\| ,
$$
where $\Yg_{-1} \in \RR^{N \times (n-1)}$ denotes the submatrix of $\Yg$ with its first column removed. 
The Davis--Kahan theorem (Lemma~\ref{lem:dk}) and Proposition~\ref{prop:gauss-spec-norm} together imply that, up to a sign flip of $\ub$, 
\begin{equation*}
\|\ub - e_1\| 
\le \frac{ 2\sqrt{2} }{ \left| \|v\|_4^4 - \frac{3}{N} \right|} 
\cdot \left\| \Mg -  \left( \|v\|_4^4 - \frac{3}{N} \right) e_1 e_1^\top \right\| 
\le \frac{ 2\sqrt{2} \, \eta }{ \left| \|v\|_4^4 - \frac{3}{N} \right|} 
\end{equation*}
with probability at least $1-\delta$, 
where $\eta$ is defined in \eqref{eq:def-eta}. 
Moreover, by a standard bound on the spectral norm of a matrix with Gaussian entries (see, for example, Theorem~4.4.5 of \cite{vershynin2018high}), we have 
\begin{equation}\label{eq:Y-1-norm-bound}
\|\Yg_{-1}\| \le C \bigg( 1 + \sqrt{\frac{\log(1/\delta)}{N}} \, \bigg) \le 2 C
\end{equation}
with probability at least $1-\delta$ for a constant $C>0$. Combining the above three bounds completes the proof.
\end{proof}

\subsection{Proof of Lemma~\ref{lem:change-norm}}
\label{sec:change-norm}
By the assumption $\left|\|v\|_4^4 - \frac{3}{N} \right| \ge \frac{c}{N}$, it holds that
$$
\left|\|v\|_4^4 - \frac{3}{N} \right| \ge \left|\|v\|_4^4 - \frac{3}{N} \right| \cdot \frac{\|v\|_4^4}{\frac{3}{N} + \left|\|v\|_4^4 - \frac{3}{N} \right|} \ge \inf_{t \ge \frac{c}{N}} \frac{t}{\frac{3}{N} + t} \|v\|_4^4 = \frac{c}{3+c} \|v\|_4^4 .
$$
Consequently, it suffices to bound 
\begin{equation*}
\frac{3+c}{c} \bigg( \frac{1}{ \|v\|_4^4 } \frac{n}{N^{3/2}} 
+ \frac{1}{ \|v\|_4^2 } \frac{\sqrt{n}}{N} 
+ \frac{\|v\|_6^3}{\|v\|_4^4} \sqrt{\frac{n}{N}}
+ \frac{\|v\|_\infty^2}{ \|v\|_4^4 } \frac{n}{N} 
\bigg) .
\end{equation*}
We have 
\begin{itemize}
\item
$\frac{1}{ \|v\|_4^2 } \frac{\sqrt{n}}{N} \le \frac{\|v\|_\infty}{ \|v\|_4^2 } \sqrt{\frac{n}{N}}$ since $\|v\|_\infty \ge \frac{1}{\sqrt{N}}$ for any unit vector $v$;

\item
$\frac{\|v\|_6^3}{\|v\|_4^4} \sqrt{\frac{n}{N}} \le \frac{\|v\|_\infty}{ \|v\|_4^2 } \sqrt{\frac{n}{N}}$
since $\|v\|_6^3 \le \|v\|_\infty \|v\|_4^2$;

\item
$\frac{\|v\|_\infty^2}{ \|v\|_4^4 } \frac{n}{N} \le \frac{\|v\|_\infty}{ \|v\|_4^2 } \sqrt{\frac{n}{N}}$ since (due to the fact $\frac{9+3c}{c} \ge 3$) it is assumed that $\frac{\|v\|_\infty}{ \|v\|_4^2 } \sqrt{\frac{n}{N}} \le 1$.
\end{itemize}
Combining the above bounds completes the proof.

\subsection{Rank-One Perturbation}
\label{sec:rank-one-perturb-pf}

As a preparation for proving our upper bounds, we state a bound for eigenvector perturbation under a rank-one update. 
Recall that for a symmetric matrix, the singular values are simply the absolute values of the eigenvalues.
The following form of the Davis--Kahan theorem for rank-one perturbation is known; for example, see the original proof by Davis and Kahan \cite{davis1970rotation}, or more recent versions such as Theorem~8.5 of \cite{wainwright2019high} and Theorem~2.2.1 of \cite{chen2020spectral}. 
We provide a self-contained proof for completeness.

\begin{lemma}
\label{lem:rank-one-perturb}
Fix a symmetric matrix $A \in \RR^{n \times n}$, a scalar $\rho \in \RR$, and a vector $b \in \RR^{n \times n}$. 
Let $(\lambda_1, u_1)$ be the leading eigenpair of $A$.
Let $\Delta > 0$ denote the gap (i.e., absolute difference) between the largest and the second largest singular value of $A$. 
Let $\tilde u_1$ be the leading eigenvector of $A+ \rho bb^\top$. 
Suppose that $|\rho| \cdot \|b\|^2 \le \Delta/4$. 
Up to a sign flip of $\tilde u_1$, we have 
$$
\|u_1 - \tilde u_1\|
\le 2 \sqrt{2} \, \frac{|\rho| \cdot \|b\| \cdot |b^\top u_1|}{\Delta} .
$$
\end{lemma}

\begin{proof}
Let $A = U \Lambda U^\top$ be the eigendecomposition of $A$, and let $(\tilde \lambda_1, \tilde u_1)$ be the leading eigenpair of $A + \rho b b^\top$. Then we have 
$$
\tilde \lambda_1 \tilde u_1 
= (A + \rho bb^\top) \tilde u_1 
= A \tilde u_1 + \rho b  b^\top \tilde u_1 .
$$
Note that $I_n = U_{-1} U_{-1}^\top + u_1 u_1^\top$ where $U_{-1} \in \RR^{n \times (n-1)}$ denotes the submatrix of $U$ with its first column $u_1$ removed. 
Multiplying the above equation by $U_{-1}^\top$ from the left yields
\begin{align*}
\tilde \lambda_1 U_{-1}^\top \tilde u_1 
&= \Lambda_{-1} U_{-1}^\top \tilde u_1 + \rho U_{-1}^\top b b^\top (U_{-1} U_{-1}^\top + u_1 u_1^\top) \tilde u_1 \\
&= \Lambda_{-1} U_{-1}^\top \tilde u_1 + \rho U_{-1}^\top b b^\top U_{-1} U_{-1}^\top \tilde u_1 + \rho U_{-1}^\top b b^\top u_1 u_1^\top \tilde u_1 ,
\end{align*}
where $\Lambda_{-1} \in \RR^{(n-1) \times (n-1)}$ denotes the principal submatrix of $\Lambda$ with its first row and first column removed. 
Rearranging terms, we obtain
$$
( \tilde \lambda_1 I_{n-1} - \Lambda_{-1} - \rho U_{-1}^\top b b^\top U_{-1} ) U_{-1}^\top \tilde u_1 = \rho U_{-1}^\top b b^\top u_1 u_1^\top \tilde u_1 . 
$$
Assume $\lambda_1 \ge 0$ without loss of generality. By Weyl's inequality and the triangle inequality, we have $\tilde \lambda_1 - \lambda_i \ge \lambda_1 - \lambda_i - |\rho| \cdot \|b\|^2$ for any eigenvalue $\lambda_i$ of $A$, and so
$$
\tilde \lambda_1 I_{n-1} - \Lambda_{-1} - \rho U_{-1}^\top b b^\top U_{-1} 
\succeq \tilde \lambda_1 I_{n-1} - \Lambda_{-1} - |\rho| \cdot \|b\|^2 I_{n-1}
\succeq (\Delta - 2 |\rho| \cdot \|b\|^2) I_{n-1}
\succeq (\Delta/2) I_{n-1}
$$
by the assumption $|\rho| \cdot \|b\|^2 \le \Delta/4$, where $\succeq$ denotes the Loewner order. 
Moreover, we have 
$$
\|\rho U_{-1}^\top b b^\top u_1 u_1^\top \tilde u_1 \|
\le |\rho| \cdot \|b\| \cdot |b^\top u_1| \cdot 1 . 
$$
Combining the above three displays yields 
$$
\|U_{-1}^\top \tilde u_1\| 
\le 2 \frac{|\rho| \cdot \|b\| \cdot |b^\top u_1|}{\Delta} .
$$
Let $\Theta(u_1, \tilde u_1)$ denote the angle between $u_1$ and $\tilde u_1$. 
In view of the relation
$$
\|u_1 - \tilde u_1\| \le \sqrt{2} \sin \Theta(u_1, \tilde u_1) 
= \sqrt{2} \|U_{-1}^\top \tilde u_1\| 
$$
(which holds up to a sign flip of $\tilde u_1$), the proof is complete. 
\end{proof}

\subsection{Proof of Theorem~\ref{thm:entrywise}}
We now prove the entrywise error bound \eqref{eq:entrywise-bound} in Theorem~\ref{thm:entrywise}. 
It is an immediate consequence of the following more precise theorem with explicit logarithmic factors.

\begin{theorem}
\label{thm:entrywise-general}
Assume Model~\ref{mod:gauss}. 
Let $\yt_j^\top$ be the $j$th row of $\Yt$ for $j \in [N]$, and let $\tilde{\ub}$ be the leading eigenvector of the matrix $\Mt$ defined in \eqref{eq:def-mt}. 
There exist universal constants $c > 0$ and $C > 0$ such that the following holds. Let $\delta \in (0,1)$. Suppose that
$N \ge \log^7(N/\delta)$ and 
\begin{equation}
\frac{\eta}{\left| \|v\|_4^4 - \frac{3}{N} \right|} \le c ,
\label{eq:eta-cond}
\end{equation}
where $\eta = \eta(N,n,v,\delta)$ is defined in \eqref{eq:def-eta}. 
Then it holds with probability at least $1 - \delta$ that, up to a sign flip of $\tilde{\ub}$, for all $j \in [N]$, we have 
\begin{equation}
\left| \tilde y_j^\top \tilde{\ub} 
- v_j \right|
\le \frac{C \, \eta}{\left| \|v\|_4^4 - \frac{3}{N} \right|} \left( |v_j| + \frac{\sqrt{\log(N/\delta)}}{\sqrt{N}} \right) . 
\label{eq:simple}
\end{equation}
\end{theorem}

\begin{proof}[Proof of Theorem~\ref{thm:entrywise}]
Take $\delta = N^{-\log N}$ and $\|v\| = 1$ in the above theorem.  
As we have seen in \eqref{eq:l2-bd-simple}, it holds that
$$
\frac{\eta}{\left| \|v\|_4^4 - \frac{3}{N} \right|} \le \tilde O \left( \frac{1}{ \|v\|_4^4 } \frac{n}{N^{3/2}}  + \frac{\|v\|_\infty}{ \|v\|_4^2 } \sqrt{\frac{n}{N}} \, \right) 
$$
by Lemma~\ref{lem:change-norm}. 
Therefore, Theorem~\ref{thm:entrywise} follows immediately. 
\end{proof}

\begin{proof}[Proof of Theorem~\ref{thm:entrywise-general}]
Following the notation in the proof of Theorem~\ref{thm:gaussian-general}, we have $\tilde Y \tilde{\ub} = Y \ub$ where $\ub$ and $\tilde{\ub}$ are the leading eigenvectors of $M$ and $\tilde M$ respectively. 
Therefore, for $j \in [N]$, it suffices to study $y_j^\top \ub = \tilde y_j^\top \tilde{\ub}$, where $y_j^\top$ is the $j$th row of $Y$. 
The proof is based on a leave-one-out analysis. 

\subsubsection*{Setting up the leave-one-out analysis}
Recall that $y_j = (v_j \; b_j^\top)$ where $b_j \sim \mathcal{N}(0, \frac{1}{N} I_{n-1})$, so 
$$
\left( \|\yg_j\|^2 - \frac{n-1}{N} \right) y_j y_j^\top 
= \left( v_j^2 + \|b_j\|^2 - \frac{n-1}{N} \right)
\begin{bmatrix}
v_j^2 & v_j b_j^\top \\
v_j b_j & b_j b_j^\top  
\end{bmatrix} . 
$$
Let $e_1$ be the first standard basis vector in $\RR^n$, and define a matrix 
\begin{equation*}
M^{(j)} := \sum_{i \in [N] \setminus \{j\}} \left( \|\yg_i\|^2 - \frac{n-1}{N} \right) \yg_i \yg_i^\top - \frac{3}{N} I_n + v_j^4 e_1 e_1^\top . 
\end{equation*}
Furthermore, let $\tilde b_j$ be the vector in $\RR^n$ defined by $\tilde b_j^\top = (0 \;\; b_j^\top)$, and define a matrix 
\begin{equation}
L^{(j)} := M^{(j)} 
+ \left( v_j^2 + \|b_j\|^2 - \frac{n-1}{N} \right) \tilde b_j \tilde b_j^\top . 
\label{eq:update}
\end{equation}
In view of the definition of $M$ in \eqref{eq:def-mg} and the above three equations, we can write
\begin{equation}
\Mg 
= L^{(j)} + v_j^2 
\begin{bmatrix}
0 & v_j b_j^\top \\
v_j b_j & 0  
\end{bmatrix} 
+ \left( \|b_j\|^2 - \frac{n-1}{N} \right) 
\begin{bmatrix}
v_j^2 & v_j b_j^\top \\
v_j b_j & 0 
\end{bmatrix} .
\label{eq:update-2}
\end{equation}
Let the leading eigenvectors of $M^{(j)}$, $L^{(j)}$, and $M$ be denoted by $\ub^{(j)}$, $\wb^{(j)}$, and $\ub$ respectively. 
Then we have 
\begin{equation*}
y_j^\top \ub 
= y_j^\top \ub^{(j)} + y_j^\top (\wb^{(j)} - \ub^{(j)}) + y_j^\top (\ub - \wb^{(j)}) 
\end{equation*}
and thus 
\begin{equation}
\left| y_j^\top \ub 
- v_j \right|
\le \big| y_j^\top \ub^{(j)} - v_j \big| 
+ \|y_j\| \cdot \big\|\wb^{(j)} - \ub^{(j)}\big\|
+ \|y_j\| \cdot \big\|\ub - \wb^{(j)}\big\| . 
\label{eq:key-eq}
\end{equation}
The key to the leave-one-out analysis is that $y_j$ is independent of $\ub^{(j)}$, so it is easy to analyze $y_j^\top \ub^{(j)}$. 
In addition, it is not hard to bound $\|\ub - \wb^{(j)}\|$. 
The interesting part of the proof is to analyze $\|\wb^{(j)} - \ub^{(j)}\|$ by studying the rank-one update \eqref{eq:update}. 
Combining all these pieces then yields our final result regarding $y_j^\top \ub$.

\subsubsection*{Spectral norm bounds}
Define an event
\begin{equation}
\mathcal{E}_1 := \left\{ \left\| \Mg - \left( \|v\|_4^4 - \frac{3}{N} \right) e_1 e_1^\top \right\|
\le \eta \right\}
\label{eq:event-1}
\end{equation}
where $\eta = \eta(N,n,v,\delta)$ is defined in \eqref{eq:def-eta}. 
By Proposition~\ref{prop:gauss-spec-norm}, we know that $\mathcal{E}_1$ occurs with probability at least $1 - \delta$. 
Next, by standard concentration for the chi-square distribution (Lemma~\ref{lem:chi-sq}), there exists a constant $C_1 > 0$ such that the event 
$$
\mathcal{E}_2 := 
\left\{ \left| \|b_j\|^2 - \frac{n-1}{N} \right| 
\le C_1 \frac{ \sqrt{n \log(N/\delta)} + \log(N/\delta) }{ N } \text{ for all } j \in [N] \right\} 
$$
occurs with probability at least $1 - \delta$. 
On the event $\mathcal{E}_2$, we have 
\begin{equation}
\|b_j\| \le C_2 \frac{\sqrt{n} + \sqrt{\log(N/\delta)}}{\sqrt{N}}
\label{eq:b-bd}
\end{equation}
for some constant $C_2 > 0$ (this can be seen by squaring both sides of~\eqref{eq:b-bd}).
As a result, 
\begin{align}
\| M - L^{(j)} \| 
&\le 2 |v_j|^3 \|b_j\| + v_j^2 \left| \|b_j\|^2 - \frac{n-1}{N} \right| + 2 |v_j| \left| \|b_j\|^2 - \frac{n-1}{N} \right| \|b_j\| \notag \\
&\le C_3 \left( |v_j|^3 \frac{\sqrt{n} + \sqrt{\log(N/\delta)}}{\sqrt{N}} + |v_j| \frac{n \sqrt{\log(N/\delta)} + \log^{3/2}(N/\delta)}{N^{3/2}} \right) 
=: \eta_2(j) 
\label{eq:sp-bd-1}
\end{align}
for a constant $C_3 > 0$ and $\eta_2(j) = \eta_2(N,n,v_j,\delta)$, where the first inequality follows from \eqref{eq:update-2}, and the second inequality follows from the definition of $\mathcal{E}_2$ and applying the inequality $2xy \le x^2+y^2$ for $x, y \in \RR$ a few times. 
Furthermore, in view of \eqref{eq:update}, we have on the event $\mathcal{E}_2$ that 
\begin{align}
\| L^{(j)} - M^{(j)} \| 
&\le v_j^2 \|b_j\|^2 + \left| \|b_j\|^2 - \frac{n-1}{N} \right| \|b_j\|^2 \notag \\
&\le C_4 \left( v_j^2 \frac{n + \log(N/\delta)}{N} + \frac{n^{3/2} \sqrt{\log(N/\delta)} + \log^2(N/\delta)}{N^2} \right)
=: \eta_3(j) 
\label{eq:sp-bd-2}
\end{align}
for a constant $C_4 > 0$ and $\eta_3(j) = \eta_3(N,n,v_j,\delta)$. 

Using $|v_j|^3 \le \|v\|_6^3$, $|v_j| \le \|v\|_\infty$, and the inequality $2xy \le x^2+y^2$, it is not hard to see that  
\begin{equation}
\eta_2(j) + \eta_3(j) \le \eta 
\label{eq:eta-relation}
\end{equation}
for all $j \in [N]$, 
provided that the constant $C$ in \eqref{eq:def-eta} is taken to be sufficiently large. 
Therefore, on the event $\mathcal{E}_1 \cap \mathcal{E}_2$, it follows from \eqref{eq:event-1}, \eqref{eq:sp-bd-1} and \eqref{eq:sp-bd-2} that 
\begin{equation}
\left\| M^{(j)} - \left( \|v\|_4^4 - \frac{3}{N} \right) e_1 e_1^\top \right\|
\le 2 \eta . 
\label{eq:sp-bd-3} 
\end{equation}
Moreover, under the assumption \eqref{eq:eta-cond} with $c$ sufficiently small, the matrices $M$, $L^{(j)}$, and $M^{(j)}$ are all close to $\left( \|v\|_4^4 - \frac{3}{N} \right) e_1 e_1^\top$ on the event $\mathcal{E}_1 \cap \mathcal{E}_2$. 
In particular, by Weyl's inequality the gap between the leading singular value and any other singular value of each of these matrices is larger than $\frac 12 \left| \|v\|_4^4 - \frac{3}{N} \right|$. 
Recall that $\ub^{(j)}$, $\wb^{(j)}$, and $\ub$ denote the leading eigenvectors of $M^{(j)}$, $L^{(j)}$, and $M$ respectively.  
The Davis--Kahan theorem (Lemma~\ref{lem:dk}) together with \eqref{eq:sp-bd-1} and \eqref{eq:sp-bd-3} then implies that $\mathcal{E}_1 \cap \mathcal{E}_2 \subseteq \mathcal{E}_3 \cap \mathcal{E}_4$ where 
\begin{align}
\mathcal{E}_3 &:= \bigg\{ \|\ub - \wb^{(j)}\| \le \frac{6 \, \eta_2(j)}{\left| \|v\|_4^4 - \frac{3}{N} \right|} \text{ for all } j \in [N] \bigg\} , \label{eq:perturb-1} \\
\mathcal{E}_4 &:= \bigg\{ \|\ub^{(j)} - e_1\| \le \frac{12 \, \eta}{\left| \|v\|_4^4 - \frac{3}{N} \right|} \text{ for all } j \in [N] \bigg\} ,
\notag
\end{align}
if we choose the signs of the eigenvectors so that they are all close to $e_1$. 
Also, we have $\PP \{ \mathcal{E}_3 \cap \mathcal{E}_4 \} \ge \PP \{ \mathcal{E}_1 \cap \mathcal{E}_2 \} \ge 1 - 2 \delta$.

\subsubsection*{Rank-one perturbation}
We now study $\|\wb^{(j)} - \ub^{(j)}\|$ using the rank-one update \eqref{eq:update} and Lemma~\ref{lem:rank-one-perturb}. 
As discussed above, on the event $\mathcal{E}_1 \cap \mathcal{E}_2$, the gap between the largest and the second largest singular value of $M^{(j)}$ is at least $\frac 12 \left| \|v\|_4^4 - \frac{3}{N} \right|$. 
Also, we have 
$$
\left| v_j^2 + \|b_j\|^2 - \frac{n-1}{N} \right| \cdot \|b_j\|^2 
\le \eta_3(j) \le \frac 18 \left| \|v\|_4^4 - \frac{3}{N} \right| 
$$
by \eqref{eq:sp-bd-2}, \eqref{eq:eta-relation}, and \eqref{eq:eta-cond}. 
Therefore, Lemma~\ref{lem:rank-one-perturb} can be applied to give 
\begin{equation}
\|\wb^{(j)} - \ub^{(j)}\| \le \frac{4 \sqrt{2}}{\left| \|v\|_4^4 - \frac{3}{N} \right|} \left| v_j^2 + \|b_j\|^2 - \frac{n-1}{N} \right| \cdot \|b_j\| \cdot \left| \tilde b_j^\top \ub^{(j)}\right| . 
\label{eq:zeta-xi}
\end{equation}
Let $\ub^{(j)}_{-1}$ denote the subvector of $\ub^{(j)}$ with its first entry removed. 
Conditioning on any instance of $M^{(j)}$ such that the event $\mathcal{E}_4$ occurs, we have $\tilde b_j^\top \ub^{(j)} \sim \mathcal{N}(0, \frac{1}{N} \|\ub^{(j)}_{-1}\|^2)$ where $\|\ub^{(j)}_{-1}\| \le \frac{12 \, \eta}{\left| \|v\|_4^4 - \frac{3}{N} \right|}$.  
Hence, there exists a constant $C_5 > 0$ such that the event 
\begin{equation*}
\mathcal{E}_5 := \left\{ \left| \tilde b_j^\top \ub^{(j)} \right| \le C_5 \frac{\eta \sqrt{\log(N/\delta)}}{\left| \|v\|_4^4 - \frac{3}{N} \right| \sqrt{N}} \text{ for all } j \in [N] \right\} 
\end{equation*}
occurs with probability at least $1 - 3 \delta$. 
On the event $\mathcal{E}_2 \cap \mathcal{E}_5$, it follows from \eqref{eq:zeta-xi} and \eqref{eq:eta-cond} that  
\begin{align}
\|\wb^{(j)} - \ub^{(j)}\| 
&\le \frac{C_6}{\left| \|v\|_4^4 - \frac{3}{N} \right|} \left( v_j^2 + \frac{ \sqrt{n \log(N/\delta)} + \log(N/\delta) }{ N }  \right) \frac{ \sqrt{n} + \sqrt{\log(N/\delta)} }{ \sqrt{N} } \cdot \frac{\eta \sqrt{\log(N/\delta)}}{\left| \|v\|_4^4 - \frac{3}{N} \right| \sqrt{N}} \notag \\
&\le \frac{C_7}{\left| \|v\|_4^4 - \frac{3}{N} \right|} \left( v_j^2 \frac{ \sqrt{n \log(N/\delta)} + \log(N/\delta) }{ N } + \frac{ n \log(N/\delta) + \log^2(N/\delta) }{ N^2 }  \right) 
=: \frac{\eta_4(j)}{\left| \|v\|_4^4 - \frac{3}{N} \right|}
\label{eq:eta-4}
\end{align}
for constants $C_6, C_7 > 0$ and $\eta_4(j) = \eta_4(N,n,v_j,\delta)$. 

\subsubsection*{Finishing the proof}
Finally, on the event $\mathcal{E}_4 \cap \mathcal{E}_5$, we have
\begin{equation}
|\yg_j^\top \ub^{(j)} - v_j| \le |v_j| \cdot \left|1-\ub^{(j)}_1\right| + \Big| \tilde b_j^\top \ub^{(j)} \Big| 
\le \frac{C_8 \eta}{\left| \|v\|_4^4 - \frac{3}{N} \right|} \left( |v_j| + \frac{\sqrt{\log(N/\delta)}}{\sqrt{N}} \right) 
\label{eq:y-xi-ip}
\end{equation}
for a constant $C_8 > 0$. 
On the event $\mathcal{E}_1 \cap \mathcal{E}_2 \cap  \mathcal{E}_3 \cap \mathcal{E}_4 \cap \mathcal{E}_5$ which occurs with probability at least $1 - 3 \delta$, we combine \eqref{eq:key-eq}, 
\eqref{eq:y-xi-ip}, 
\eqref{eq:b-bd}, 
\eqref{eq:eta-4}, and
\eqref{eq:perturb-1} to obtain 
\begin{equation*}
\left| y_j^\top \ub 
- v_j \right|
\le \frac{C_8 \, \eta}{\left| \|v\|_4^4 - \frac{3}{N} \right|} \left( |v_j| + \frac{\sqrt{\log(N/\delta)}}{\sqrt{N}} \right) + \left(|v_j| + C_2 \frac{\sqrt{n} + \sqrt{\log(N/\delta)}}{\sqrt{N}} \right) \frac{\eta_4(j) + 6 \, \eta_2(j)}{\left| \|v\|_4^4 - \frac{3}{N} \right|}  . 
\end{equation*}
Comparing the above bound to our final result \eqref{eq:simple}, we see that it remains to show that the second summand on the right-hand side is dominated by the first. 
To this end, first note that $\eta_4(j) + 6 \, \eta_2(j) \le 7 \,\eta$ using \eqref{eq:eta-relation}, \eqref{eq:eta-4}, and \eqref{eq:def-eta}, provided that the constant $C$ in \eqref{eq:def-eta} is taken to be sufficiently large. 
Moreover, we have 
\begin{align*}
&\frac{\sqrt{n} + \sqrt{\log(N/\delta)}}{\sqrt{N}} \cdot \eta_4(j)  \\
&= C_7 \frac{\sqrt{n} + \sqrt{\log(N/\delta)}}{\sqrt{N}} \left( v_j^2 \frac{ \sqrt{n \log(N/\delta)} + \log(N/\delta) }{ N } + \frac{ n \log(N/\delta) + \log^2(N/\delta) }{ N^2 }  \right) \\
&\le C_9 \left( |v_j| \cdot |v_j| \frac{ n \sqrt{\log(N/\delta)} + \log^{3/2}(N/\delta) }{ N^{3/2} } + \frac{\sqrt{\log(N/\delta)}}{\sqrt{N}} \frac{ n^{3/2} + \log^2(N/\delta) }{ N^2 }  \right) \\
&\le C_9 \cdot \left( |v_j| + \frac{\sqrt{\log(N/\delta)}}{\sqrt{N}} \right) \cdot \eta 
\end{align*}
for a constant $C_9 > 0$, and 
\begin{align*}
&\frac{\sqrt{n} + \sqrt{\log(N/\delta)}}{\sqrt{N}} \cdot \eta_2(j)  \\
&= C_3  \frac{\sqrt{n} + \sqrt{\log(N/\delta)}}{\sqrt{N}}  \left( |v_j|^3 \frac{\sqrt{n} + \sqrt{\log(N/\delta)}}{\sqrt{N}} + |v_j| \frac{n \sqrt{\log(N/\delta)} + \log^{3/2}(N/\delta)}{N^{3/2}} \right) \\
&\le C_{10}  \cdot |v_j| \left( v_j^2 \frac{n + \log(N/\delta)}{N} + \frac{n^{3/2} \sqrt{\log(N/\delta)} + \log^2(N/\delta)}{N^2} \right) \\
&\le C_{10} \cdot |v_j| \cdot \eta 
\end{align*}
for a constant $C_{10}>0$. 
Combining the above estimates yields 
\begin{equation*}
\left| y_j^\top \ub 
- v_j \right|
\le \frac{C_{11} \, \eta}{\left| \|v\|_4^4 - \frac{3}{N} \right|} \left( |v_j| + \frac{\sqrt{\log(N/\delta)}}{\sqrt{N}} \right) 
\end{equation*}
for a constant $C_{11} > 0$. 
\end{proof}

\subsection{Preliminaries for the Orthonormal Case}
\label{sec:prelim}

We now prove lemmas stated in Section~\ref{sec:pre}.

\begin{proof}[Proof of Lemma~\ref{lem:cov}]
By the definitions of $A$ and $B$ in \eqref{eq:def-ab}, we have 
$$
A - I_n = \sum_{i=1}^N
\begin{bmatrix}
v_i^2 & v_i b_i^\top \\
v_i b_i & b_i b_i^\top 
\end{bmatrix} 
- I_n 
= 
\begin{bmatrix}
0 & w^\top \\
w & B - I_{n-1} 
\end{bmatrix}
,
$$
where $w = \sum_{i=1}^N v_i b_i$. Therefore, 
$
\|A - I_n\| \le 2 \|w\| + \|B - I_{n-1}\| . 
$
Lemmas~\ref{lem:chi-sq} and~\ref{lem:sum-mat} can then be applied to bound $\|w\|$ and $\|B - I_{n-1}\|$ respectively to the desired order. 
\end{proof}

\begin{lemma}\label{lem:rand-proj}
Fix vectors $x,y \in \RR^N$, and define $\alpha := \langle x,y \rangle$. Let $P \in \RR^{N \times N}$ denote the projection matrix for orthogonal projection onto a uniformly random $n$-dimensional subspace of $\RR^N$. Then with probability $1-N^{-\omega(1)}$,
\[ \left| x^\top P y - \alpha \frac{n}{N} \right| \le \|x\| \cdot \|y\| \cdot \tilde O\left(\frac{\sqrt n}{N}\right). \]
\end{lemma}

\begin{proof}
Without loss of generality, we may assume that $x$ and $y$ are unit vectors. 
Furthermore, by rotational invariance, it is equivalent to take $x,y$ to be uniformly random unit vectors conditional on $\langle x,y \rangle = \alpha$, and take $P$ to be the projection matrix for orthogonal projection onto a fixed $n$-dimensional subspace---we will choose simply the span of the first $n$ standard basis vectors. We can sample $(x,y)$ as follows. First let $x = \tilde x / \|\tilde x\|$ where $\tilde x \sim \mathcal{N}(0, \frac{1}{N} I_N)$. Then sample $z$ orthogonal to $x$ via
\[ \tilde z \sim \mathcal{N} \left( 0, \frac{1}{N} I_N \right) , \qquad \bar z = \tilde z - \langle \tilde z,x \rangle x, \qquad z = \bar z / \|\bar z\|. \]
Finally, let $y = \alpha x + \sqrt{1-\alpha^2} z$ so that $\langle x,y \rangle = \alpha$. By standard concentration, we have with probability $1-N^{-\omega(1)}$,
\[ \Big| \|\tilde x\| - 1 \Big| \le \tilde{O}\left(\frac{1}{\sqrt N}\right), \qquad \Big| \|\tilde z\| - 1 \Big| \le \tilde{O}\left(\frac{1}{\sqrt N}\right), \qquad |\langle \tilde x, \tilde z \rangle| \le \tilde{O}\left(\frac{1}{\sqrt N}\right), \]
and as a result,
\[ \Big| \|\bar z\| - 1 \Big| \le \tilde{O}\left(\frac{1}{\sqrt N}\right). \] 
Leveraging our convenient choice for $P$, we also have with probability $1-N^{-\omega(1)}$,
\[ \left| \tilde x^\top P \tilde x - \frac{n}{N} \right| = \left| \sum_{i=1}^n \tilde x_i^2 - \frac{n}{N} \right| \le \tilde{O}\left(\frac{\sqrt n}{N}\right), \]
\[ \left| \tilde x^\top P \tilde z \right| = \left| \sum_{i=1}^n \tilde x_i \tilde z_i \right| \le \tilde{O}\left(\frac{\sqrt n}{N}\right). \]
As a result,
\begin{align*}
x^\top P y &= \alpha x^\top P x + \sqrt{1-\alpha^2} x^\top P z \\
&= \frac{\alpha}{\|\tilde x\|^2} \tilde x^\top P \tilde x + \frac{\sqrt{1-\alpha^2}}{\|\bar z\|} x^\top P (\tilde z - \langle \tilde z,x \rangle x) \\
&= \frac{\alpha}{\|\tilde x\|^2} \tilde x^\top P \tilde x + \frac{\sqrt{1-\alpha^2}}{\|\tilde x\| \|\bar z\|} \tilde x^\top P \tilde z - \frac{\sqrt{1-\alpha^2}}{\|\tilde x\|^3 \|\bar z\|} \langle \tilde z,\tilde x \rangle \tilde x^\top P \tilde x.
\end{align*}
Using the bounds from above,
\[ \left| x^\top P y - \alpha \frac{n}{N} \right| \le \tilde{O}\left(\alpha \frac{n}{N^{3/2}} + \alpha \frac{\sqrt n}{N} + \frac{\sqrt n}{N} + \frac{n}{N^{3/2}}\right) \le \tilde{O}\left(\frac{\sqrt n}{N}\right) \]
as desired.
\end{proof}

\begin{proof}[Proof of Lemma~\ref{lem:A-exp}]
Since 
\begin{equation}
A = \sum_{i=1}^N
\begin{bmatrix}
v_i^2 & v_i b_i^\top \\
v_i b_i & b_i b_i^\top 
\end{bmatrix} 
=
\begin{bmatrix}
1 & w^\top \\
w & B 
\end{bmatrix} 
\quad \text{ where } \quad
w:= \sum_{i=1}^N v_i b_i ,
\label{eq:def-w}
\end{equation}
by a standard formula for the inverse of a block matrix, we have 
$$
A^{-1} =
\begin{bmatrix}
(1 - w^\top B^{-1} w)^{-1}  & - (1 - w^\top B^{-1} w)^{-1} w^\top B^{-1} \\
- B^{-1} w  (1 - w^\top B^{-1} w)^{-1} & B^{-1} + B^{-1} w (1 - w^\top B^{-1} w)^{-1} w^\top B^{-1}  
\end{bmatrix} . 
$$
Therefore, letting 
\begin{equation} 
\Phi := w^\top B^{-1} w, \qquad \phi_i := w^\top B^{-1} b_i  \quad \text{ for } i \in [N],  \label{eq:Phi-phi}
\end{equation}
we obtain 
\begin{subequations} \label{eq:4eqs}
\begin{equation}\label{eq:eAe1}
e_1^\top A^{-1} \tilde e_1 = \frac{1}{1-\Phi} , 
\end{equation}
\begin{equation}\label{eq:eAb}
e_1^\top A^{-1} \tilde b_i = - \frac{\phi_i}{1-\Phi} , 
\end{equation}
\begin{equation}\label{eq:bAb}
\tilde b_i^\top A^{-1} \tilde b_i = b_i^\top B^{-1} b_i + \frac{\phi_i^2}{1-\Phi} , 
\end{equation}
\begin{equation}\label{eq:A-exp}
\yg_i^\top A^{-1} \yg_i = b_i^\top B^{-1} b_i + \frac{\phi_i^2}{1-\Phi} - \frac{2 v_i \phi_i}{1-\Phi} + \frac{v_i^2}{1-\Phi}
 = b_i^\top B^{-1} b_i + \frac{(\phi_i - v_i)^2}{1-\Phi} . 
\end{equation}
\end{subequations}

Let $U$ denote the $N \times (n-1)$ matrix with rows $b_i^\top$. Define $P = U(U^\top U)^{-1}U^\top$, which is the projection matrix for orthogonal projection onto the column span of $U$. Since $B = U^\top U$, $w = U^\top v$, and $b_i = U^\top \mathbf{e}_i$ where $\mathbf{e}_i$ is the $i$th standard basis vector in $\RR^N$, we see that 
\[ \Phi = v^\top P v, \qquad \phi_i = v^\top P \mathbf{e}_i. \]
Since the column span of $U$ is a uniformly random $(n-1)$-dimensional subspace of $\RR^N$, we have by Lemma~\ref{lem:rand-proj} that with probability $1-N^{-\omega(1)}$,
\begin{equation}
\left| \Phi - \frac{n}{N} \right| \le \tilde{O}\left(\frac{\sqrt n}{N}\right), \qquad \left|\phi_i - v_i \frac{n}{N}\right| \le \tilde{O}\left(\frac{\sqrt n}{N}\right),
\label{eq:phi-bd-1} 
\end{equation}
and as a result,
\begin{equation}
\left|\frac{1}{1-\Phi} - \frac{1}{1-n/N}\right| \le \tilde{O}\left(\frac{\sqrt n}{N}\right). 
\label{eq:phi-bd-2}
\end{equation}
Plugging these bounds into \eqref{eq:4eqs} yields 
\begin{equation*}
e_1^\top A^{-1} \tilde e_1 = \frac{1}{1-n/N} + \tilde O \left( \frac{\sqrt{n}}{N} \right) , 
\end{equation*}
\begin{align*}
e_1^\top A^{-1} \tilde b_i &= - \frac{v_i n}{N (1-n/N)} + \tilde O \left( |v_i| \frac{n}{N} \frac{\sqrt{n}}{N} + \frac{1}{1-n/N} \frac{\sqrt{n}}{N} + \frac{n}{N^2} \right) \\
&= - \frac{v_i n}{N - n} + \tilde O \left( \frac{\sqrt{n}}{N} \right) , 
\end{align*}
\begin{align*}
\tilde b_i^\top A^{-1} \tilde b_i 
&= b_i^\top B^{-1} b_i + v_i^2 \frac{n^2}{N^2} \frac{1}{1-n/N} + \tilde O \left( \left(v_i^2 \frac{n^2}{N^2} + \frac{n}{N^2} + |v_i| \frac{n}{N} + \frac{\sqrt{n}}{N} \right) \frac{\sqrt{n}}{N} \right) \\ 
&= b_i^\top B^{-1} b_i + \frac{v_i^2 n^2}{N(N-n)} + \tilde O \left( |v_i| \frac{n^{3/2}}{N^2} + \frac{n}{N^2} \right)  , 
\end{align*}
\begin{align*}
\yg_i^\top A^{-1} \yg_i &= b_i^\top B^{-1} b_i + \left(v_i \frac{n}{N} - v_i\right)^2 \frac{1}{1-n/N} + \tilde{O}\left( \left( v_i^2 + |v_i| \frac{\sqrt{n}}{N} + \frac{n}{N^2} \right) \frac{\sqrt n}{N} + |v_i| \frac{\sqrt{n}}{N} + \frac{n}{N^2} \right) \\
&= b_i^\top B^{-1} b_i + (1-n/N) v_i^2 + \tilde{O}\left(|v_i| \frac{\sqrt n}{N} + \frac{n}{N^2}\right) ,
\end{align*}
as desired.
\end{proof}

\begin{proof}[Proof of Lemma~\ref{lem:A-j}]
Part of the proof is very similar to that of Lemma~\ref{lem:A-exp}, so we will omit some computations. 
Since 
$$
A_{-j} = \sum_{i \in [N] \setminus \{j\}}
\begin{bmatrix}
v_i^2 & v_i b_i^\top \\
v_i b_i & b_i b_i^\top 
\end{bmatrix} 
+ v_j^2 e_1 e_1^\top 
=
\begin{bmatrix}
1 & w_{-j}^\top \\
w_{-j} & B_{-j} 
\end{bmatrix} 
\quad \text{ where }
w_{-j} := \sum_{i \in [N] \setminus \{j\}}  v_i b_i ,
$$
letting 
\begin{equation}
\Phi_{-j} := w_{-j}^\top B_{-j}^{-1} w_{-j}, \qquad \phi_{-j}^{(i)} := w_{-j}^\top B_{-j}^{-1} b_i  \quad \text{ for } i \in [N] \setminus \{j\} ,  \label{eq:Phi-phi-j}
\end{equation}
we obtain from the formula for the inverse of a block matrix that, for any $i \in [N] \setminus \{j\}$, 
\begin{subequations} \label{eq:2eqs}
\begin{equation}
e_1^\top A_{-j}^{-1} e_1 = \frac{1}{1 - \Phi_{-j}} , 
\end{equation}
\begin{equation}
\yg_i^\top A_{-j}^{-1} \yg_i 
= b_i^\top B_{-j}^{-1} b_i + \frac{(\phi_{-j}^{(i)} - v_i)^2}{1-\Phi_{-j}}  . 
\end{equation}
\end{subequations}
Let $\Phi$ and $\phi_i$ be defined by \eqref{eq:Phi-phi}. 
As a result of \eqref{eq:eAe1}, \eqref{eq:A-exp}, and \eqref{eq:2eqs}, 
\begin{subequations}
\begin{equation} \label{eq:eAe}
e_1^\top (A^{-1} - A_{-j}^{-1}) e_1 = \frac{\Phi - \Phi_{-j}}{(1-\Phi) (1-\Phi_{-j})} , 
\end{equation}
\begin{align}
\yg_i^\top (A^{-1} - A_{-j}^{-1}) \yg_i 
&= b_i^\top (B^{-1} - B_{-j}^{-1}) b_i 
+ \frac{(\phi_i - v_i)^2 - (\phi_{-j}^{(i)} - v_i)^2}{1-\Phi} 
+ (\phi_{-j}^{(i)} - v_i)^2 \left( \frac{1}{1-\Phi} - \frac{1}{1-\Phi_{-j}} \right) \notag \\
&= - \frac{ (b_i^\top B_{-j}^{-1} b_j)^2 }{ 1 + b_j^\top B_{-j}^{-1} b_j }
+ \frac{(\phi_i - \phi_{-j}^{(i)}) (\phi_i + \phi_{-j}^{(i)} - 2 v_i)}{1-\Phi} 
+ (\phi_{-j}^{(i)} - v_i)^2 \frac{\Phi - \Phi_{-j}}{(1-\Phi) (1-\Phi_{-j})} ,  \label{eq:ayya-temp}
\end{align}
\end{subequations}
where the last step follows from the Sherman--Morrison formula. 

Let $U_{-j}$ denote the $(N-1) \times (n-1)$ matrix with rows $b_i^\top$ for $i \in [N] \setminus \{j\}$. Let $P_{-j} = U_{-j}(U_{-j}^\top U_{-j})^{-1}U_{-j}^\top$, which is the projection matrix for orthogonal projection onto the column span of $U_{-j}$. Since $B_{-j} = U_{-j}^\top U_{-j}$, $w_{-j} = U_{-j}^\top v_{-j}$ where $v_{-j}$ is the subvector of $v$ with its $j$th entry removed, and $b_i = U_{-j}^\top \mathbf{e}_i$ where $\mathbf{e}_i$ is the $i$th standard basis vector in $\RR^{N-1}$ indexed by $i \in [N] \setminus \{j\}$, we see that 
\[ \Phi_{-j} = v_{-j}^\top P_{-j} v_{-j}, \qquad \phi_{-j}^{(i)} = v_{-j}^\top P_{-j} \mathbf{e}_i. \]
Since the column span of $U_{-j}$ is a uniformly random $(n-1)$-dimensional subspace of $\RR^{N-1}$, we have by Lemma~\ref{lem:rand-proj} that with probability $1-N^{-\omega(1)}$,
\begin{equation}
\left| \Phi_{-j} - \|v_{-j}\|^2 \frac{n}{N} \right| \le \tilde{O}\left(\frac{\sqrt n}{N}\right), \qquad \left|\phi_{-j}^{(i)} - v_i \frac{n}{N}\right| \le \tilde{O}\left(\frac{\sqrt n}{N}\right) , 
\label{eq:phi-bd-3}
\end{equation}
and as a result,
\begin{equation}
\left|\frac{1}{1-\Phi_{-j}} \right| \le \tilde{O}\left(1\right). 
\label{eq:phi-bd-4}
\end{equation}
Plugging \eqref{eq:phi-bd-1}, \eqref{eq:phi-bd-2}, \eqref{eq:phi-bd-3}, and \eqref{eq:phi-bd-4} into \eqref{eq:eAe} yields that, with probability $1 - N^{-\omega(1)}$, 
\begin{equation*}
\left| e_1^\top (A^{-1} - A_{-j}^{-1}) e_1 \right| 
\le \tilde{O}\left( v_j^2 \frac{n}{N} + \frac{\sqrt{n}}{N} \right) ,
\end{equation*}
proving \eqref{eq:e-a-a-e}. 
Next, plugging \eqref{eq:phi-bd-1}, \eqref{eq:phi-bd-2}, \eqref{eq:phi-bd-3}, and \eqref{eq:phi-bd-4} into \eqref{eq:ayya-temp} yields that, with probability $1 - N^{-\omega(1)}$, 
\begin{equation}
\yg_i^\top (A^{-1} - A_{-j}^{-1}) \yg_i 
= - \frac{ (b_i^\top B_{-j}^{-1} b_j)^2 }{ 1 + b_j^\top B_{-j}^{-1} b_j }
+ \tilde O \left( \left| \phi_i - \phi_{-j}^{(i)} \right| \left(|v_i| + \frac{\sqrt{n}}{N} \right)
+ \left(|v_i| + \frac{\sqrt{n}}{N} \right)^2 \left( v_j^2 \frac{n}{N} + \frac{\sqrt{n}}{N} \right) \right) . \label{eq:yaay-2}
\end{equation}

Let $w$ be defined as in \eqref{eq:def-w}. 
To bound \eqref{eq:yaay-2}, note that by \eqref{eq:Phi-phi}, \eqref{eq:Phi-phi-j}, and the Sherman--Morrison formula, 
\begin{equation}
\phi_i - \phi_{-j}^{(i)} 
= w^\top (B^{-1} - B_{-j}^{-1}) b_i + (w - w_{-j})^\top B_{-j}^{-1} b_i 
= - \frac{ (w^\top B_{-j}^{-1} b_j) (b_i^\top B_{-j}^{-1} b_j) }{ 1 + b_j^\top B_{-j}^{-1} b_j } + v_j b_j^\top B_{-j}^{-1} b_i . 
\label{eq:phi-phi-j}
\end{equation}
By Lemma~\ref{lem:sum-mat}, we have $\|B_{-j}^{-1}\| \le 2$ with probability $1 - N^{-\omega(1)}$. Together with Lemma~\ref{lem:chi-sq}, this implies
\begin{equation}
\left| b_j^\top B_{-j}^{-1} b_j \right| 
\le \|b_j\|^2 \|B_{-j}^{-1}\| \le \tilde O(n/N) 
\label{eq:bbb1}
\end{equation}
with probability $1 - N^{-\omega(1)}$. 
Moreover, by the independence of $b_j$ from $b_i$ and $B_{-j}$, we have $b_i^\top B_{-j}^{-1} b_j \sim \mathcal{N}(0, \|B_{-j}^{-1} b_i\|^2 / N)$ conditional on $b_j$ and $B_{-j}$. As a result, 
\begin{equation}
\left| b_i^\top B_{-j}^{-1} b_j \right| 
\le \tilde O \left( \frac{\|B_{-j}^{-1} b_i\|}{\sqrt{N}} \right) 
\le \tilde O \left( \frac{\|b_i\|}{\sqrt{N}} \right) 
\le \tilde O \left( \frac{\sqrt{n}}{N} \right) 
\label{eq:bbb2}
\end{equation}
with probability $1 - N^{-\omega(1)}$. 
Next, we have $w = w_{-j} + v_j b_j$, where $w_{-j} = \sum_{i \in [N] \setminus \{j\}} v_i b_i$ so that $\|w_{-j}\|^2 = \tilde O(n/N)$ with probability $1 - N^{-\omega(1)}$. 
Since $b_j$ is independent from $w_{-j}$ and $B_{-j}^{-1}$, we have $w_{-j}^\top B_{-j}^{-1} b_j \sim \mathcal{N}(0, \|B_{-j}^{-1} w_{-j}\|^2 / N)$ conditional on $w_{-j}$ and $B_{-j}$. As a result, 
\begin{equation}
\left| w^\top B_{-j}^{-1} b_j \right| 
\le \left| v_j b_j^\top B_{-j}^{-1} b_j \right| + \left| w_{-j}^\top B_{-j}^{-1} b_j \right| 
\le \tilde O \left( |v_j| \frac{n}{N} \right) + \tilde O \left( \frac{\|B_{-j}^{-1} w_{-j}\|}{\sqrt{N}} \right) \le \tilde O \left( |v_j| \frac{n}{N} + \frac{\sqrt{n}}{N} \right) 
\label{eq:bbb3}
\end{equation}
with probability $1 - N^{-\omega(1)}$. 
Then \eqref{eq:bbb1}, \eqref{eq:bbb2}, \eqref{eq:bbb3}, and \eqref{eq:phi-phi-j} together imply 
$$
\frac{ (b_i^\top B_{-j}^{-1} b_j)^2 }{ 1 + b_j^\top B_{-j}^{-1} b_j } 
= \tilde O \left( \frac{n}{N^2} \right) , \qquad 
\left| \phi_i - \phi_{-j}^{(i)} \right|
= \tilde O \left( |v_j| \frac{\sqrt{n}}{N} + \frac{n}{N^2} \right) . 
$$
Combining these bounds with \eqref{eq:yaay-2}, we obtain that, with probability $1 - N^{-\omega(1)}$, 
\begin{align*}
\left| \yg_i^\top (A^{-1} - A_{-j}^{-1}) \yg_i \right| 
&\le \tilde O \left( \frac{n}{N^2} 
+ \left( |v_j| \frac{\sqrt{n}}{N} + \frac{n}{N^2} \right) \left(|v_i| + \frac{\sqrt{n}}{N} \right)
+ \left(v_i^2 + \frac{n}{N^2} \right) \left( v_j^2 \frac{n}{N} + \frac{\sqrt{n}}{N} \right) \right) \\
&\le \tilde O \left( \frac{n}{N^2} 
+ |v_i| \cdot |v_j| \frac{\sqrt{n}}{N} + \big( |v_i| + |v_j| \big) \frac{n}{N^2} 
+ v_i^2 v_j^2 \frac{n}{N} + v_j^2 \frac{n^2}{N^3} 
+ v_i^2 \frac{\sqrt{n}}{N} \right) , 
\end{align*}
which proves \eqref{eq:y-a-a-y}. 
\end{proof}

\begin{lemma}\label{lem:trace}
It holds with probability $1-o(1)$ that, for every $i \in [N]$,
\[ \left|\Tr \left( B_{-i}^{-1} \right) - \frac{n-1}{1-n/N} \right| \le \tilde{O}\left(\sqrt{\frac{n}{N}} \, \right). \]
\end{lemma}

\begin{proof}
Note that $NB$ follows the Wishart distribution with $N$ degrees of freedom and covariance $I_{n-1}$, and so $\frac{1}{N} B^{-1}$ follows the inverse Wishart distribution (with the same parameters). We will make use of the following formulas for moments of the inverse Wishart distribution~\cite{multivar-book}:
\[ \EE \left[ (B^{-1})_{ii} \right] = \frac{1}{1-n/N} , \]
\[ \Var \left( (B^{-1})_{ii} \right) = \frac{2N^2}{(N-n)^2(N-n-2)} , \]
\[ \Cov \left( (B^{-1})_{ii}, (B^{-1})_{jj} \right) = \frac{2N^2}{(N-n+1)(N-n)^2(N-n-2)} , \qquad i \ne j. \]
As a result, we conclude
\[ \EE \left[ \Tr(B^{-1}) \right] = \frac{n-1}{1-n/N} , \]
\[ \Var \left( \Tr(B^{-1}) \right) = \sum_{i=1}^{n-1} \Var \left( (B^{-1})_{ii} \right) + 2 \sum_{1 \le i < j \le n-1} \Cov \left( (B^{-1})_{ii}, (B^{-1})_{jj} \right) \le \tilde{O}\left(\frac{n}{N} + \frac{n^2}{N^2}\right) \le \tilde{O}\left(\frac{n}{N}\right). \]
By Chebyshev's inequality, we obtain 
\begin{equation}\label{eq:cheb}
\left|\Tr(B^{-1}) - \frac{n-1}{1-n/N}\right| \le \tilde{O}\left(\sqrt{\frac{n}{N}} \, \right).
\end{equation}

We now compare $\Tr(B_{-i}^{-1})$ to $\Tr(B^{-1})$. Using the Sherman--Morrison rank-one update formula and the cyclic property of the trace,
\[ \Tr(B^{-1}) = \Tr\left(B_{-i}^{-1} - \frac{B_{-i}^{-1} b_i b_i^\top B_{-i}^{-1}}{1 + b_i^\top B_{-i}^{-1} b_i}\right) = \Tr(B_{-i}^{-1}) - \frac{b_i^\top B_{-i}^{-2} b_i}{1 + b_i^\top B_{-i}^{-1} b_i}. \]
We have with probability $1-N^{-\omega(1)}$ that $\|b_i\|^2 = \tilde O(n/N)$ and $\|B_{-i}^{-1}\| = \tilde O(1)$, which implies 
\begin{equation*}  
\left|\Tr(B_{-i}^{-1}) - \Tr(B^{-1})\right| \le \tilde O\left(\frac{n}{N}\right). 
\end{equation*}
The result follows by combining this with~\eqref{eq:cheb}.
\end{proof}

\begin{lemma}\label{lem:B-exp}
It holds with probability $1-o(1)$ that, for every $i \in [N]$,
\[ \Big| b_i^\top B^{-1} b_i - \|b_i\|^2 \Big| \le \tilde{O}\left(\frac{n}{N^{3/2}}\right) . \]
\end{lemma}

\begin{proof}
Using the Sherman--Morrison rank-one update formula,
\begin{equation}\label{eq:B-exp}
b_i^\top B^{-1} b_i = b_i^\top \left(B_{-i}^{-1} - \frac{B_{-i}^{-1} b_i b_i^\top B_{-i}^{-1}}{1 + b_i^\top B_{-i}^{-1} b_i}\right) b_i = \psi_i - \frac{\psi_i^2}{1+\psi_i}
= \frac{\psi_i}{1+\psi_i}
\end{equation}
where
\[ \psi_i := b_i^\top B_{-i}^{-1} b_i. \]
Note that $b_i$ is independent from $B_{-i}^{-1}$. 
Let us first condition on $B_{-i}^{-1}$. 
Then we have 
$$
\EE \left[ \psi_i \mid B_{-i}^{-1} \right] = \frac{\Tr(B_{-i}^{-1})}{N} .
$$
Moreover, by the Hanson--Wright inequality (Lemma~\ref{lem:hs}), we obtain that with probability $1-N^{-\omega(1)}$ over $b_i$,
\begin{equation}
\bigg| \psi_i - \frac{\Tr(B_{-i}^{-1})}{N} \bigg| 
\le \|B_{-i}^{-1}\|_F \cdot \tilde O \bigg( \frac{1}{N} \bigg) 
\le \|B_{-i}^{-1}\| \cdot \tilde O \bigg( \frac{\sqrt{n} }{N} \bigg) ,
\label{eq:hs-1}
\end{equation}
\begin{align}
\bigg| \psi_i - \frac{\|b_i\|^2}{1-n/N} \bigg|
&= \bigg| b_i^\top \Big( B_{-i}^{-1} - \frac{I_{n-1}}{1-n/N} \Big) b_i \bigg| \notag \\
&\le \bigg| \frac{\Tr(B_{-i}^{-1})}{N} - \frac{n-1}{N(1-n/N)} \bigg| 
+ \bigg\| B_{-i}^{-1} - \frac{I_{n-1}}{1-n/N}  \bigg\| \cdot \tilde O \bigg( \frac{ \sqrt{n} }{N} \bigg) .
\label{eq:hs-2}
\end{align}
Next, Lemma~\ref{lem:sum-mat} gives that, with probability $1-N^{-\omega(1)}$, for all $i \in [N]$,
\[ \bigg\| B_{-i}^{-1} - \frac{I_{n-1}}{1-n/N}  \bigg\| \le \|B_{-i}^{-1} - I_{n-1}\| + O \left( \frac{n}{N} \right) \le \tilde{O}\left(\sqrt{\frac{n}{N}} \, \right) , \]
and Lemma~\ref{lem:trace} gives that, with probability $1-o(1)$, for all $i \in [N]$,
\[ \bigg| \frac{\Tr(B_{-i}^{-1})}{N} - \frac{n}{N-n} \bigg|
\le \bigg| \frac{\Tr(B_{-i}^{-1})}{N} - \frac{n-1}{N(1-n/N)} \bigg| + O \left( \frac{n}{N^2} \right)
\le \tilde{O}\left( \frac{\sqrt{n}}{N^{3/2}} \right) . \]
Combining the above two bounds with \eqref{eq:hs-1} and \eqref{eq:hs-2} respectively, 
we obtain that, with probability $1-o(1)$, for all $i \in [N]$,
\begin{equation}\label{eq:psi-2}
\left| \psi_i - \frac{n}{N-n} \right| \le \tilde{O}\left(\frac{\sqrt n}{N}\right) \quad \Longrightarrow \quad \left|\frac{1}{1+\psi_i} - \left(1 + \frac{n}{N-n} \right)^{-1}\right| \le \tilde{O}\left(\frac{\sqrt n}{N}\right) , 
\end{equation}
\begin{equation}\label{eq:psi-1}
\left| \psi_i - \frac{\|b_i\|^2}{1-n/N} \right| 
\le \tilde{O}\left(\frac{n}{N^{3/2}}\right) . 
\end{equation}
Finally, plugging the bounds~\eqref{eq:psi-2},~\eqref{eq:psi-1}, and $\|b_i\|^2 \le 2n/N$ (with probability $1-N^{-\omega(1)}$ by Lemma~\ref{lem:chi-sq}) into~\eqref{eq:B-exp} yields
\begin{equation*}
b_i^\top B^{-1} b_i 
= \psi_i \cdot \frac{1}{1+\psi_i}
= \frac{\|b_i\|^2}{1-n/N} \left(1 + \frac{n}{N-n} \right)^{-1} + \Lambda_i 
= \|b_i\|^2 + \Lambda_i
\end{equation*}
for an error term that satisfies
\[ |\Lambda_i| \le \tilde{O}\left(\frac{n}{N^{3/2}} + \frac{n^{3/2}}{N^2}\right) \le \tilde{O}\left(\frac{n}{N^{3/2}}\right) \]
as desired. 
\end{proof}

 \begin{proof}[Proof of Lemma~\ref{lem:lev}]
Combining \eqref{eq:y-a-y-1} and Lemma~\ref{lem:B-exp}, we have with probability $1-o(1)$ that for every $i \in [N]$,
\begin{align*}
\Big| \yg_i^\top A^{-1} \yg_i - \|\yg_i\|^2 \Big| 
&\le \Big| b_i^\top B^{-1} b_i + (1-n/N) v_i^2 - v_i^2 - \|b_i\|^2 \Big| + \tilde{O}\left(\frac{n}{N^2} + |v_i|\frac{\sqrt n}{N}\right) \\
&\le v_i^2 \frac{n}{N} + \tilde{O}\left(\frac{n}{N^{3/2}} + |v_i|\frac{\sqrt n}{N}\right) ,
\end{align*}
which proves \eqref{eq:lev}. 
Moreover, we have $\left| \|b_i\|^2 - n/N \right| = \tilde O( \frac{\sqrt{n}}{N} )$ with probability $1 - N^{-\omega(1)}$ by Lemma~\ref{lem:chi-sq}. 
Then \eqref{eq:lev2} follows from the concentration of $\|y_i\|^2 = v_i^2 + \|b_i\|^2$. 

Similarly, by \eqref{eq:b-a-b-1} and Lemma~\ref{lem:B-exp}, we have with probability $1-o(1)$ that for every $i \in [N]$, 
\begin{align*}
\Big| \tilde b_i^\top A^{-1} \tilde b_i - \frac{n}{N} \Big| 
&\le \Big| b_i^\top B^{-1} b_i - \frac{n}{N} \Big| + \frac{v_i^2 n^2}{N(N-n)} + \tilde O \left( |v_i| \frac{n^{3/2}}{N^2} + \frac{n}{N^2} \right) \\
&\le \Big| \|b_i\|^2 - \frac{n}{N} \Big| + \tilde O \left( \frac{n}{N^{3/2}} \right) + \tilde O \left( v_i^2 \frac{n^2}{N^2} + |v_i| \frac{n^{3/2}}{N^2} + \frac{n}{N^2} \right) \\
&\le \tilde O \left( \frac{\sqrt{n}}{N} + v_i^2 \frac{n^2}{N^2} \right) ,
\end{align*}
which finishes the proof. 
\end{proof}

\subsection{Proof of Corollary~\ref{cor:orth-bg}}

By assumption, we have $\frac{1}{N} \ll \rho \le 1 - c$ for a constant $c \in (0, 0.1)$. 
Since $v' \sim \BG(N, \rho)$ and $v = v'/\|v'\|$, both vectors $v$ and $v'$ are supported on a set $\mathcal{S} \subset [N]$ with $|\mathcal{S}| \sim \mathsf{Binomial}(N, \rho)$. 
Since $\rho \gg 1/N$, it follows from Lemma~\ref{lem:bern-bdd} that with probability $1 - N^{-\omega(1)}$,
$$
\big| |\mathcal{S}| - N \rho \big| 
\le \tilde O \left( \sqrt{N \rho} \, \right) .
$$
Conditional on $\mathcal{S}$, the nonzero entries of $v'$ are independent $\mathcal{N}(0, \frac{1}{N \rho})$ variables. 
As a result, it follows from a standard maximal inequality and Lemma~\ref{lem:gauss-lp} (with $p = 2, 4$) that 
$$
\|v'\|_\infty \le \tilde{O} \left( \frac{1}{ \sqrt{N \rho} } \right), \qquad
\left| \|v'\|^2 - \frac{|\mathcal{S}|}{N \rho} \right| \le \tilde{O} \left( \frac{ \sqrt{|\mathcal{S}|} }{N \rho} \right), \qquad
\left| \|v'\|_4^4 - \frac{3 |\mathcal{S}|}{(N \rho)^2} \right| \le \tilde{O} \left( \frac{ \sqrt{|\mathcal{S}|} }{(N \rho)^{2}} \right),
$$
with probability $1 - N^{-\omega(1)}$ conditional on $\mathcal{S}$. 
Therefore, with probability $1 - N^{-\omega(1)}$, we have
\begin{equation}
\left| \|v'\|^2 - 1 \right| \le \tilde O\left( \frac{1}{\sqrt{N \rho}} \right) , \qquad
\left| \|v'\|_4^4 - \frac{3}{N \rho} \right| \le \tilde{O} \left( \frac{ 1 }{(N \rho)^{3/2}} \right), 
\label{eq:bds-1}
\end{equation}
from which it follows that 
\begin{equation}
\|v\|_\infty = \frac{\|v'\|_\infty}{\|v'\|} \le \tilde{O} \left( \frac{1}{ \sqrt{N \rho} } \right), \qquad
\|v\|_4^4 = \frac{\|v'\|_4^4}{\|v'\|^4} \ge \frac{2}{N \rho} .
\label{eq:bds-2}
\end{equation}

With the above estimates established, we now verify the assumptions of Theorem~\ref{thm:spectral} (so that it can be applied). 
First, by \eqref{eq:bds-1}, 
$$
\left| \|v'\|_4^4 - \|v\|_4^4 \right| 
= \|v'\|_4^4 \cdot \left| 1 - \frac{1}{\|v'\|^4} \right| 
= \|v'\|_4^4 \cdot \left| \frac{(\|v'\|^2 - 1) (\|v'\|^2 + 1)}{\|v'\|^4} \right|
\le \tilde{O} \left( \frac{1}{(N \rho)^{3/2}} \right) ,
$$
and, as a result, 
$$
\left| \|v\|_4^4 - \frac{3}{N} \right| 
\ge \left| \frac{3}{N \rho} - \frac{3}{N} \right| - \left| \|v'\|_4^4 - \frac{3}{N \rho} \right| - \left| \|v'\|_4^4 - \|v\|_4^4 \right|
\ge \frac{3}{N} \frac{1 - \rho}{\rho} - \tilde{O} \left( \frac{ 1 }{(N \rho)^{3/2}} \right)
\ge \frac{2 c}{N} 
$$
since $\rho \le 1-c$. 
Second, combining \eqref{eq:bds-2} with the assumption $\xi = \frac{n \rho}{\sqrt{N}}  + \sqrt{\frac{n}{N}} \ll 1$, we see that $\eps = \frac{1}{ \|v\|_4^4 } \frac{n}{N^{3/2}}  + \frac{\|v\|_\infty}{ \|v\|_4^2 } \sqrt{\frac{n}{N}}$ defined in Theorem~\ref{thm:spectral} satisfies 
$$
\eps 
\le \tilde{O}
\bigg( N \rho \, \frac{n}{N^{3/2}} +  \sqrt{\frac{n}{N}} \, \bigg) = \tilde{O}
( \xi ) 
\ll 1.
$$
Consequently, by Theorem~\ref{thm:spectral}, up to a sign flip of $\hat{\ub}$, it holds with probability $1 - o(1)$ that 
\begin{equation*}
\left| \hat y_j^\top \hat{\ub} 
- v_j \right|
\le \tilde O \left( \eps \cdot \Big( |v_j| + \frac{1}{\sqrt{N}} \Big) \right)
\le \tilde O \left( \frac{\xi}{\sqrt{N \rho}} \right)   \quad \text{ for all } j \in [N],
\end{equation*}
and
$$
\|\Yh \hat{\ub} - v\| \le \tilde{O} (\xi) .
$$

\subsection{Proof of Corollary~\ref{cor:orth-br}}

By assumption, we have $\rho \gg \frac{1}{N}$ and $|\rho - \frac{1}{3}| \ge c$ for a constant $c \in (0, 0.1)$. 
Since $v' \sim \BR(N, \rho)$, by Lemma~\ref{lem:bern-bdd}, there is an event $\mathcal{E}$ of probability $1 - N^{-\omega(1)}$ on which 
$$
\big| \|v\|_0 - N \rho \big|
= \big| \|v'\|_0 - N \rho \big| 
\le \tilde O \left( \sqrt{N \rho} \, \right) ,
$$
where $v = v'/\|v'\|$. 
Note that the nonzero entries of $v$ are $\pm 1/\sqrt{\|v\|_0}$, so we obtain
$$
\|v\|_\infty^2 \le \frac{2}{N \rho},
\qquad
\|v\|_4^4 \ge \frac{1}{2N\rho},
\qquad
\left|\|v\|_4^4 - \frac{3}{N}\right| 
\ge \frac{3c}{2 N \rho}.
$$
Conditional on the event $\mathcal{E}$, the quantity $\eps$ defined in \eqref{eq:spec-assume} satisfies 
$$
\eps 
\le O
\bigg( N \rho \, \frac{n}{N^{3/2}} +  \sqrt{\frac{n}{N}} \bigg) \le O
\bigg( \frac{n \rho}{\sqrt{N}}  + \sqrt{\frac{n}{N}} \, \bigg) 
\ll 1
$$
by the assumption $\frac{n \rho}{\sqrt{N}}  + \sqrt{\frac{n}{N}} \ll 1$. 
By Theorem~\ref{thm:spectral}, up to a sign flip of $\hat{\ub}$, it holds with probability at least $1 - o(1)$ that 
\begin{equation*}
\left| \hat y_j^\top \hat{\ub} 
- v_j \right|
\le \tilde O \left( \eps \cdot \Big( |v_j| + \frac{1}{\sqrt{N}} \Big) \right)
\le \tilde O \left( \eps \cdot \frac{1}{\sqrt{N \rho}} \right)  
\ll \frac{1}{\sqrt{N \rho}} 
\end{equation*}
for all $j \in [N]$. 
Since the nonzero entries of $v$ are $\pm 1/\sqrt{\|v\|_0} \approx \pm 1/\sqrt{N \rho}$, we can exactly recover the support of $v$ and the signs of its entries by thresholding at the level $0.5 \cdot \max_{i \in [N]} \big| \hat y_i^\top \hat{\ub} \, \big|$. 
A normalization then recovers $v$.

\subsection{Proof of Theorem~\ref{thm:reduction}}
\label{sec:reduction-pf}

We first show that the test $\tilde\psi$ succeeds under distribution $\cP$ (in Problem~\ref{prob:detect}). We have $v \sim \BG(N, \rho)$, and so there is a random subset $\mathcal{S} \subset [N]$ with support size $|\mathcal{S}| \sim \mathsf{Binomial}(N, \rho)$ and, conditional on $\mathcal{S}$, the restriction of $v$ on $\mathcal{S}$ is a $\mathcal{N} \big(0, \frac{1}{\sqrt{N\rho}} I_{|\mathcal{S}|}\big)$ vector.
Since $\rho \gg 1/N$, it follows from Lemma~\ref{lem:bern-bdd} that with probability $1 - N^{-\omega(1)}$,
$$
\Big| |\mathcal{S}| - N \rho \Big|
\le \tilde O \left( \sqrt{N \rho} \right) .
$$
Further, conditional on $\mathcal{S}$, it follows from Lemma~\ref{lem:gauss-lp} (with $p = 1, 2$) that with probability $1 - N^{-\omega(1)}$, 
$$
\left| \|v\|^2 - \frac{|\mathcal{S}|}{N \rho} \right| \le \tilde O \left( \frac{\sqrt{|\mathcal{S}|}}{N \rho} \right) , \qquad
\left| \|v\|_1 - |\mathcal{S}| \sqrt{\frac{2}{\pi N \rho}} \, \right| \le \tilde O\left( \sqrt{ \frac{|\mathcal{S}|}{N \rho} } \, \right) . 
$$
The above bounds together imply that, with probability $1 - N^{-\omega(1)}$, 
$$
\left| \|v\|^2 - 1 \right| \le \tilde O \left( \frac{1}{\sqrt{N \rho}} \right) , \qquad
\left| \|v\|_1 - \sqrt{\frac{2}{\pi} N \rho} \right| \le \tilde O\left( 1 \right) . 
$$

Since $\|\tilde v - v\| \le c_3$ by assumption, we obtain 
$$
\big| \|\tilde v\| - 1 \big|
\le \big| \|v\| - 1 \big| + \|\tilde v - v\| 
\le 2 c_3 
$$
and
$$
\left| \|\tilde v\|_1 - \sqrt{\frac{2}{\pi} N \rho} \right|
\le \left| \|v\|_1 - \sqrt{\frac{2}{\pi} N \rho} \right| + \|\tilde v - v\|_1 
\le \tilde O(1) + \sqrt{N} \|\tilde v - v\| 
\le 2 c_3 \sqrt{N} . 
$$
Combining the above two bounds yields
$$
\left| \frac{\|\tilde v\|_1}{\|\tilde v\|} - \sqrt{\frac{2}{\pi} N \rho} \right|
\le \left| \|\tilde v\|_1 - \sqrt{\frac{2}{\pi} N \rho} \right|
+ \left| \frac{\|\tilde v\|_1}{\|\tilde v\|} - \|\tilde v\|_1 \right|
\le 2 c_3 \sqrt{N} + \|\tilde v\|_1 \left| \frac{1}{\|\tilde v\|} - 1 \right|
\le 5 c_3 \sqrt{N} . 
$$ 
If $\rho \le 1 - c_1$ and $c_3 = c_3(c_1) > 0$ is sufficiently small, then we have 
$$
\left| \frac{\|\tilde v\|_1}{\|\tilde v\|} - \sqrt{\frac{2}{\pi} N} \right| 
\ge \left| \sqrt{\frac{2}{\pi} N} - \sqrt{\frac{2}{\pi} N \rho} \right| - \left| \frac{\|\tilde v\|_1}{\|\tilde v\|} - \sqrt{\frac{2}{\pi} N \rho} \right|
\ge \sqrt{\frac{2}{\pi} N} \cdot \frac{c_1}{2} - 5 c_3 \sqrt{N}
\ge \frac{c_1}{4} \sqrt{N}.
$$
Therefore, the test $\tilde \psi$ succeeds in identifying $\cP$ with probability at least $1 - \delta - N^{-\omega(1)}$.

We now show that $\tilde\psi$ succeeds under distribution $\cQ$. Now the matrix $\tilde Y$ has i.i.d.\ $\mathcal{N}(0, \frac{1}{N} I_N)$ entries. 
Thus, by Lemma~\ref{lem:l2-l1}, for any constant $c_4 > 0$, 
$$
\PP \Bigg\{ \max_{u \in \RR^n} \Bigg| \frac{\|\tilde Y u\|}{\|u\|} - 1 \Bigg| \ge c_4 \Bigg\} \le N^{-\omega(1)} 
$$
and
$$
\PP \Bigg\{ \max_{u \in \RR^n} \Bigg| \frac{\|\tilde Y u\|_1}{\|u\|} - \sqrt{\frac{2 N}{\pi}} \Bigg| \ge c_4 \sqrt{N} \Bigg\} \le N^{-\omega(1)} ,
$$ 
provided $N \ge C_2 n$ for a constant $C_2 = C_2(c_4) > 0$. 
As a result, with probability $1 - N^{-\omega(1)}$, we have 
\begin{equation*}
\Bigg| \frac{\|\tilde Y u\|_1}{\|\tilde Y u\|} - \sqrt{\frac{2 N}{\pi}} \Bigg| 
\le \Bigg| \frac{\|\tilde Y u\|_1}{\|u\|} - \sqrt{\frac{2 N}{\pi}} \Bigg| 
+ \Bigg| \frac{\|\tilde Y u\|_1}{\|\tilde Y u\|} - \frac{\|\tilde Y u\|_1}{\|u\|} \Bigg| 
\le c_4 \sqrt{N} + \frac{\|\tilde Y u\|_1}{\|u\|} \Bigg| \frac{\|u\|}{\|\tilde Y u\|} - 1 \Bigg|
\le 2 c_4 \sqrt{N}  
\end{equation*}
for every nonzero $u \in \RR^n$, provided $c_4$ is sufficiently small. 
In particular, since $\tilde v$ is in the column span of $\tilde Y$, we have $\tilde v = \tilde Y u$ for some $u \in \RR^n$. 
The test $\tilde \psi$ then succeeds in identifying $\cQ$ with probability at least $1 - N^{-\omega(1)}$, provided $c_4 = c_4(c_1) > 0$ is sufficiently small.

\subsection{Proof of Theorem~\ref{thm:detect-upper}}
\label{sec:detect-upper-pf}

Under $\cP$ in Problem~\ref{prob:detect}, by Proposition~\ref{prop:gauss-spec-norm}, it holds with probability at least $1-\delta$ that
\begin{equation}
\left| \|\Mt\| - \Big| \|v\|_4^4 - \frac{3}{N} \Big| \right| 
= \left| \|M\| - \Big| \|v\|_4^4 - \frac{3}{N} \Big| \right|
\le \eta ,
\label{eq:bound-1}
\end{equation}
where $\eta$ is defined by \eqref{eq:def-eta}. 
Since $v \sim \BG(N, \rho)$ where $\frac{1}{N} \ll \rho \le 1 - c$, similar to the proof of Corollary~\ref{cor:orth-bg}, we can use Lemmas~\ref{lem:bern-bdd} and~\ref{lem:gauss-lp} to show that the following holds with probability $1 - N^{-\omega(1)}$:
\begin{itemize}
\item
$\|v\|_0 \le 2 N \rho;$

\item
$\|v\|_\infty \le \tilde{O} \big( \frac{1}{\sqrt{N \rho}} \big);$

\item
$\left| \|v\|^2 - 1 \right| \le \tilde{O} \big( \frac{1}{\sqrt{N \rho}} \big)$; 

\item
$\|v\|_4^4 \le \|v\|_0 \, \|v\|_\infty^4 \le \tilde{O} \big( \frac{1}{N \rho} \big)$;

\item
$\|v\|_6^6 \le  \|v\|_0 \, \|v\|_\infty^6 \le \tilde{O} \big( \frac{1}{(N \rho)^2} \big) ;$
\end{itemize}
and, crucially, 
\begin{equation}
\left|\|v\|_4^4 - \frac{3}{N}\right| 
\ge \left|\frac{3}{N \rho} - \frac{3}{N}\right| - \left|\|v\|_4^4 - \frac{3}{N \rho}\right| 
\ge \frac{3}{N} \frac{1 - \rho}{\rho} - \tilde{O} \left( \frac{ 1 }{(N \rho)^{3/2}} \right)
\ge \frac{2 c}{N \rho} .
\label{eq:bound-2}
\end{equation}
Then the quantity $\eta$ defined in \eqref{eq:def-eta} satisfies 
\begin{align*}
\frac{\eta}{\left| \|v\|_4^4 - \frac{3}{N} \right|}
&\le \frac{N \rho}{2 c} \cdot \tilde{O} \bigg( \frac{n}{N^{3/2}} 
+ \frac{\sqrt{n}}{N^{3/2} \sqrt{\rho}} 
+ \frac{ \sqrt{n}}{N^{3/2} \rho} + \frac{n}{N^2 \rho} + \frac{ n^{3/2} }{N^2} + \frac{1}{N^{3/2} \sqrt{\rho}} \bigg) \notag \\
&\le \tilde{O} \bigg( \frac{n \rho}{\sqrt{N}} + \sqrt{\frac{n}{N}} \, \bigg) 
\le \frac 12     \notag
\end{align*} 
by the assumptions $n \rho \ll \sqrt{N}$ and $n \ll N$. 
Combining the above bound with \eqref{eq:bound-1} and \eqref{eq:bound-2} gives
$$
\|\tilde M\| 
\ge \left|\|v\|_4^4 - \frac{3}{N}\right| - \eta 
\ge \frac 12 \left|\|v\|_4^4 - \frac{3}{N}\right| 
\ge \frac{c}{N \rho} . 
$$

Under $\cQ$ in Problem~\ref{prob:detect}, we continue to let $\Mg$ be defined by \eqref{eq:def-mg} so that
\begin{align*}
\|\Mt\| = \|\Mg\| 
&= \left\| \sum_{i=1}^N \left( \|\yg_i\|^2 - \frac{n-1}{N} \right) \yg_i \yg_i^\top - \frac{3}{N} I_n \right\| \\
&\le \left\| \sum_{i=1}^N \left( \|\yg_i\|^2 - \frac{n}{N} \right) \yg_i \yg_i^\top - \frac{2}{N} I_n \right\| 
+ \left\| \sum_{i=1}^N \frac{1}{N} \yg_i \yg_i^\top - \frac{1}{N} I_n \right\| , 
\end{align*}
where $\yg_i$ are i.i.d.\ $\mathcal{N}(0, \frac{1}{N} I_n)$ vectors. 
Therefore, we can apply Lemmas~\ref{lem:l5} and~\ref{lem:sum-mat} to bound the two terms respectively to obtain 
$$
\|\Mt\| \le \tilde{O} \left( \frac{n}{N^{3/2}} + \frac{n^{3/2}}{N^2} + \frac{\sqrt{n}}{N^{3/2}} \right) 
\le \tilde{O} \left( \frac{n}{N^{3/2}} \right) 
\le \frac{c}{4 N \rho}
$$
with probability $1 - N^{-\omega(1)}$, under the condition $n \rho \ll \sqrt{N}$.

\subsection{Preliminaries for the Lower Bounds}
\label{sec:prelim-lower}

We state and prove some preliminary results that are used in Section~\ref{sec:low-deg-pf}.

\begin{lemma}\label{lem:ld-hermite}
Consider the distribution $\cP$ in Problem~\ref{prob:alt} and suppose the first $D$ moments of $\nu$ are finite. Then
$$
\Adv_{\le D}^2 = \sum_{\substack{\alpha \in \NN^{N \times n} \\ |\alpha| \le D}} \left(\EE_{y \sim \mathcal{P}}[H_\alpha(y)]\right)^2.
$$
\end{lemma}

\begin{proof}
A polynomial $f: \RR^{N \times n} \to \RR$ of degree at most $D$ can be expanded as $f = \sum_{|\alpha| \le D} \hat f_\alpha H_\alpha$ where $\alpha \in \NN^{N \times n}$ and $\hat f_\alpha \in \RR$. Let $c_\alpha = \EE_{\mathcal{P}}[H_\alpha]$. Treating $\hat f$ and $c$ as vectors indexed by $\alpha$, we have from~\eqref{eq:adv} that
$$\Adv_{\le D} = \max_{\hat f} \frac{\langle \hat f,c \rangle}{\|\hat f\|}$$
and so the optimizer is $\hat f = c$, from which the result $\Adv_{\le D} = \|c\|$ follows.
\end{proof}

\section{Probability and Linear Algebra Tools}

We state some basic tools from probability and linear algebra, all of which are standard. 
Recall that (see, for example, Proposition~2.7.1 of \cite{vershynin2018high}) a random variable $X$ is sub-exponential with parameter $K>0$ if 
\begin{equation}
\EE[\exp(\lambda X)] \le \exp(K^2 \lambda^2), \quad \text{ for all } \lambda \in [-1/K, 1/K] . 
\label{eq:def-sub-exp}
\end{equation}

\begin{lemma}[Bernstein's inequality for sub-exponential distributions, Theorem~2.8.1 of \cite{vershynin2018high}]
\label{lem:bern}
Let $X_1$, $\dots$, $X_N$ be independent, zero-mean random variables. Suppose that $X_i$ is sub-exponential with parameter $K_i > 0$ for each $i \in [N]$. 
Then there exists a universal constant $C>0$ such that for any $\delta \in (0,1)$, we have
$$
\PP \left\{ \left| \sum_{i=1}^N X_i \right| \ge C \left( \sqrt{ \sum_{i=1}^N K_i^2 \cdot \log \frac{1}{\delta} } + \max_{i \in [N]} K_i \cdot \log \frac{1}{\delta} \right) \right\} \le \delta .
$$
\end{lemma}

\begin{lemma}[Bernstein's inequality for bounded distributions, Theorem~2.8.4 of \cite{vershynin2018high}]
\label{lem:bern-bdd}
Let $X_1, \dots, X_N$ be independent, zero-mean random variables such that $|X_i| \le K$ almost surely for a constant $K>0$ and each $i \in [N]$. 
Then there exists a universal constant $C>0$ such that for any $\delta \in (0,1)$, we have
$$
\PP \left\{ \left| \sum_{i=1}^N X_i \right| \ge C \left( \sqrt{ \sum_{i=1}^N \EE[X_i^2] \cdot \log \frac{1}{\delta} } + K \cdot \log \frac{1}{\delta} \right) \right\} \le \delta .
$$
\end{lemma}

\begin{lemma}[Matrix Bernstein with truncation]
\label{lem:tmb}
Let $S_1, \dots, S_N$ be independent $n_1 \times n_2$ random matrices. 
For each $i \in [N]$, suppose that there is an event $\mathcal{E}_i$ measurable with respect to $S_i$ such that $\PP\{\mathcal{E}_i\} \ge 1 - \delta_1$ for $\delta_1 \in (0,1)$ and $\left\| \tilde S_i - \EE[ \tilde S_i ] \right\| \le L$ almost surely, where $\tilde S_i := S_i \cdot \mathbf{1}\{ \mathcal{E}_i \}$. 
Let $\sigma > 0$ be defined so that 
$$
\sigma^2 = \left\| \sum_{i=1}^N \left( \EE[ \tilde S_i \tilde S_i^\top ] - \EE[ \tilde S_i ] \EE[ \tilde S_i^\top ] \right) \right\| \lor \left\| \sum_{i=1}^N \left( \EE[ \tilde S_i^\top \tilde S_i ] - \EE[ \tilde S_i^\top ] \EE[ \tilde S_i ] \right)  \right\| . 
$$
Moreover, suppose that 
$\|\EE[S_i] - \EE[\tilde S_i]\| \le \Delta$
for a fixed $\Delta > 0$ and all $i \in [N]$. 
Then for any $\delta_2 \in (0, 1)$, 
$$
\PP \left\{ \left\| \sum_{i=1}^N (S_i - \EE[S_i]) \right\| \ge 2 \sigma \sqrt{ \log \frac{n_1 + n_2}{\delta_2} } + \frac{4L}{3} \log \frac{n_1 + n_2}{\delta_2} + N \Delta \right\} \le N \delta_1 + \delta_2 . 
$$
\end{lemma}

\begin{proof}
Applying the matrix Bernstein inequality (Theorem~1.6.2 of \cite{tropp2015introduction}) to $\tilde S_i$, we obtain 
$$
\PP \left\{ \left\| \sum_{i=1}^N (\tilde S_i - \EE[\tilde S_i]) \right\| \ge 2 \sigma \sqrt{ \log \frac{n_1 + n_2}{\delta_2} } + \frac{4L}{3} \log \frac{n_1 + n_2}{\delta_2} \right\} \le \delta_2 . 
$$
Using that $\|\EE[S_i] - \EE[\tilde S_i]\| \le \Delta$ and that $S_i = \tilde S_i$ with probability at least $1-\delta_1$, we can complete the proof by a triangle inequality and a union bound. 
\end{proof}

\begin{lemma}[Hanson--Wright \cite{rudelson2013hanson}]
\label{lem:hs}
Let $X \sim \mathcal{N}(0, I_n)$ and fix a matrix $H \in \RR^{n \times n}$. 
There exists a universal constant $C>0$ such that for any $\delta \in (0, 1)$, 
$$
\PP \left\{ \Big| X^\top H X - \EE[ X^\top H X ] \Big| \ge C \left( \|H\|_F \sqrt{ \log(1/\delta) } + \|H\| \log(1/\delta) \right) \right\} \le \delta . 
$$
\end{lemma}

\begin{lemma}[Chi-square tail bound]
\label{lem:chi-sq}
Let $Y$ be a chi-square random variable with $n$ degrees of freedom. 
There exists a universal constant $C>0$ such that for any $\delta \in (0,1)$,
$$
\PP \left\{ |Y - n| \ge C \left( \sqrt{n \log(1/\delta)} + \log(1/\delta) \right) \right\} \le \delta . 
$$
\end{lemma}

\begin{lemma}[$\ell_p$-norm of a Gaussian vector]
\label{lem:gauss-lp}
Let $Z$ be a $\mathcal{N}(0, I_n)$ Gaussian vector, and let $Z_1$ denote its first entry. 
Fix a positive integer $p$. 
There exists a constant $C_p > 0$ such that for any $\delta \in (0,1)$,
$$
\PP \left\{ \left| \|Z\|_p^p - n \EE[|Z_1|^p] \right| \ge C_p \sqrt{n} \log^{p/2} (1/\delta) \right\} \le \delta . 
$$
\end{lemma}

\begin{proof}
Concentration of a polynomial of Gaussians is a well-studied topic (see, e.g., \cite{latala2006estimates}). 
The above simple bound follows immediately from Theorem~1.9 of \cite{schudy2012concentration}. 
\end{proof}

\begin{lemma}[Sample covariance concentration, (5.25) of \cite{vershynin}]
\label{lem:sum-mat}
Let $X_1, \dots, X_N$ be i.i.d.\ $\mathcal{N}(0, I_n)$ vectors. 
There exists a universal constant $C>0$ such that for any $\delta \in (0, 1)$, 
$$
\PP \left\{ \left\| \frac{1}{N} \sum_{i=1}^N X_i X_i^\top - I_n \right\| \ge C \left( \sqrt{ \frac{n + \log(1/\delta)}{N} } + \frac{n + \log(1/\delta)}{N} \right) \right\} \le \delta . 
$$
\end{lemma}

\begin{lemma}[Davis--Kahan, Theorem~4 of \cite{o2018random}]
\label{lem:dk}
Let $A$ and $B$ be $n \times n$ matrices. Let $u$ and $v$ be the leading left (or right) singular vectors of $A$ and $B$ respectively. 
Let $\Delta > 0$ denote the gap between the largest and the second largest singular value of $A$. Up to a sign flip of $v$, we have 
$$
\|u - v\| \le \sqrt{2}\, \sin \Theta(u, v) \le 2 \sqrt{2}\, \frac{\|A - B\|}{\Delta} . 
$$
\end{lemma}

\begin{lemma}
\label{lem:l2-l1}
Let $G$ be an $N \times n$ matrix with i.i.d.\ standard Gaussian entries. For any constant $c_1>0$, there exist constants $C_2, c_3 > 0$ such that if $N \ge C_2 n$, then  
\begin{equation}
\PP \bigg\{ \max_{u \in \RR^n} \bigg| \frac{\|G u\|}{\|u\|} - \sqrt{N} \bigg| \ge c_1 \sqrt{N}  \bigg\} \le 2 \exp( - N) 
\label{eq:l2-l2}
\end{equation}
and
\begin{equation}
\PP \bigg\{ \max_{u \in \RR^n} \bigg| \frac{\|G u\|_1}{\|u\|} - \sqrt{\frac{2}{\pi}} N \bigg| \ge c_1 N  \bigg\} \le 2 \exp( - c_3 N) . 
\label{eq:l2-l1}
\end{equation}
\end{lemma}

\begin{proof}
First, \eqref{eq:l2-l2} is a consequence of Theorem~4.6.1 of \cite{vershynin2018high}. 
Second, \eqref{eq:l2-l1} follows from Lemma~A.14 of \cite{qu2016finding}. 
To be more precise, the error probability in Lemma~A.14 of \cite{qu2016finding} has the form $2 \exp(-c_3 n)$ with $n$ in the exponent, but it is easy to see from the proof that the error probability is in fact at most $2 \exp(-c_3 N)$. 
\end{proof}

\section*{Acknowledgments}

We thank Oded Regev and Min Jae Song for helpful discussions that inspired the low-degree formula in Lemma~\ref{lem:formula}, and thank Yihong Wu for discussions on spectral perturbation. 
We thank Hongjie Chen, Tommaso d'Orsi, Daniel Hsu, Philippe Rigollet, Kaizheng Wang, and anonymous reviewers for helpful comments on an earlier version of this work.

\bibliographystyle{alpha}
\bibliography{main}

\end{document}